\RequirePackage{fix-cm}
\documentclass[11pt,letterpaper,11pt,oneside,leqno, aop,noinfoline, preprint]{imsart}
\usepackage[utf8]{inputenc}
\usepackage{color}
\usepackage{verbatim}
\usepackage{refstyle}
\usepackage{wrapfig}
\usepackage{mathrsfs}
\usepackage{mathtools}
\usepackage{enumitem}
\usepackage{amsmath}
\usepackage{amsthm}
\usepackage{amssymb}
\usepackage[unicode=true,
 bookmarks=true,bookmarksnumbered=false,bookmarksopen=false,
 breaklinks=false,pdfborder={0 0 0},pdfborderstyle={},backref=false,colorlinks=true]
 {hyperref}
\hypersetup{pdftitle={Liouville first-passage percolation: subsequential scaling limits at high temperature},
 pdfauthor={Jian Ding and Alexander Dunlap}}
\setattribute{journal}{name}{}

\makeatletter

\AtBeginDocument{\providecommand\subsecref[1]{\ref{subsec:#1}}}
\AtBeginDocument{\providecommand\secref[1]{\ref{sec:#1}}}
\AtBeginDocument{\providecommand\figref[1]{\ref{fig:#1}}}
\AtBeginDocument{\providecommand\eqref[1]{\ref{eq:#1}}}
\AtBeginDocument{\providecommand\lemref[1]{\ref{lem:#1}}}
\AtBeginDocument{\providecommand\criref[1]{\ref{cri:#1}}}
\AtBeginDocument{\providecommand\corrref[1]{\ref{corr:#1}}}
\AtBeginDocument{\providecommand\thmref[1]{\ref{thm:#1}}}
\AtBeginDocument{\providecommand\propref[1]{\ref{prop:#1}}}
\AtBeginDocument{\providecommand\enuref[1]{\ref{enu:#1}}}
\pdfpageheight\paperheight
\pdfpagewidth\paperwidth

\RS@ifundefined{subsecref}
  {\newref{subsec}{name = \RSsectxt}}
  {}
\RS@ifundefined{thmref}
  {\def\RSthmtxt{theorem~}\newref{thm}{name = \RSthmtxt}}
  {}
\RS@ifundefined{lemref}
  {\def\RSlemtxt{lemma~}\newref{lem}{name = \RSlemtxt}}
  {}

\numberwithin{equation}{section}
\numberwithin{figure}{section}
\numberwithin{table}{section}
\theoremstyle{plain}
\newtheorem{thm}{\protect\theoremname}[section]
\theoremstyle{remark}
\newtheorem{rem}[thm]{\protect\remarkname}
\theoremstyle{plain}
\newtheorem{criterion}[thm]{\protect\criterionname}
\theoremstyle{plain}
\newtheorem{lem}[thm]{\protect\lemmaname}
\theoremstyle{plain}
\newtheorem{cor}[thm]{\protect\corollaryname}
\theoremstyle{definition}
\newtheorem{defn}[thm]{\protect\definitionname}
\theoremstyle{plain}
\newtheorem{prop}[thm]{\protect\propositionname}

\newref{rem}{name=Remark~}
\newref{lem}{name=Lemma~}
\newref{thm}{name=Theorem~}
\newref{prop}{name=proposition~}
\newref{cond}{name=Condition~}
\newref{prp}{name=Property~}
\newref{corr}{name=Corollary~}
\newref{cri}{name=Criterion~}
\newref{sec}{name=Section~}
\newref{fig}{name=Figure~}
\newref{prop}{name=Proposition~}
\newref{subsec}{name=Subsection~}
\newref{eq}{name=,refcmd=(\ref{#1})}
\usepackage{tikz}
\usepackage{tikz-dimline}
\usetikzlibrary{patterns}
\pgfplotsset{compat=1.12}

\@ifundefined{showcaptionsetup}{}{%
 \PassOptionsToPackage{caption=false}{subfig}}
\usepackage{subfig}
\makeatother

\providecommand{\corollaryname}{Corollary}
\providecommand{\criterionname}{Criterion}
\providecommand{\definitionname}{Definition}
\providecommand{\lemmaname}{Lemma}
\providecommand{\propositionname}{Proposition}
\providecommand{\remarkname}{Remark}
\providecommand{\theoremname}{Theorem}

\begin{document}
\global\long\def\Var{\operatorname{Var}}

\global\long\def\Cov{\operatorname{Cov}}

\global\long\def\CV{\operatorname{CV}}

\global\long\def\Bernoulli{\operatorname{Bernoulli}}

\global\long\def\dist{\operatorname{dist}}

\global\long\def\diam{\operatorname{diam}}

\begin{frontmatter}

\title{Liouville first-passage percolation: subsequential scaling limits
at high temperature}

\runtitle{LFPP: subsequential scaling limits at high temperature}
\begin{aug}

\author{Jian Ding\thanksref{gjd}\ead[label=jd]{dingjian@wharton.upenn.edu}\and
Alexander Dunlap\thanksref{gajd}\ead[label=ajd]{ajdunl2@stanford.edu}}

\runauthor{J.~Ding and A.~Dunlap}

\thankstext{gjd}{Partially supported by NSF grant DMS-1455049 and an Alfred Sloan fellowship.}
\thankstext{gajd}{Partially supported by an NSF Graduate Research Fellowship.}

\affiliation{University of Pennsylvania and Stanford University}

\address{Statistics Department, The Wharton School\\
University of Pennsylvania\\
3730 Walnut Street\\
Philadelphia, PA 19104 USA\\
\printead{jd}}
\begin{aug}
\end{aug}

\address{Mathematics Department\\
Stanford University\\
450 Serra Mall, Building 380\\
Stanford, CA 94305 USA\\
\printead{ajd}}

\end{aug}

\begin{abstract}
Let $\{Y_{\mathfrak{B}}(x)\,:\,x\in\mathfrak{B}\}$ be a discrete
Gaussian free field in a two-dimensional box $\mathfrak{B}$ of side
length $S$ with Dirichlet boundary conditions. We study Liouville
first-passage percolation: the shortest-path metric in which each
vertex $x$ is given a weight of $e^{\gamma Y_{\mathfrak{B}}(x)}$
for some $\gamma>0$. We show that for sufficiently small but fixed
$\gamma>0$, for any sequence of scales $\{S_{k}\}$ there exists
a subsequence along which the appropriately scaled and interpolated
Liouville FPP metric converges in the Gromov–Hausdorff sense to a
random metric on the unit square in $\mathbf{R}^{2}$. In addition,
all possible (conjecturally unique) scaling limits are homeomorphic
by bi-Hölder-continuous homeomorphisms to the unit square with the
Euclidean metric.
\end{abstract}
\begin{keyword}[class=MSC]
\kwd[Primary ]{60K35}
\kwd[; secondary ]{60G60}
\kwd{60B43}
\end{keyword}

\begin{keyword}
\kwd{Liouville first-passage percolation}
\kwd{discrete Gaussian free field}
\kwd{Russo--Seymour--Welsh method}
\kwd{Liouville quantum gravity}
\end{keyword}
\end{frontmatter}

\section{Introduction}

We consider Liouville first-passage percolation; i.e., first-passage
percolation on the exponential of the discrete Gaussian free field.
Given a box (by which we mean a discrete rectangle) $\mathfrak{B}\subset\mathbf{Z}^{2}$,
define $\overline{\mathfrak{B}}$, the \emph{blow-up} of $\mathfrak{B}$,
as the box three times larger in each dimension centered around $\mathfrak{B}$,
and define $\partial\overline{\mathfrak{B}}$ to be the set of points
whose Euclidean distance from $\overline{\mathfrak{B}}$ is exactly
$1$. What we will call the discrete Gaussian free field on $\mathfrak{B}$
is the restriction to $\mathfrak{B}$ of the standard discrete Gaussian
free field with Dirichlet boundary conditions on $\overline{\mathfrak{B}}$.
This is the mean-zero Gaussian process $Y_{\mathfrak{B}}(x)$ such
that $Y_{\mathfrak{B}}(x)=0$ for all $x\in\partial\overline{\mathfrak{B}}$
and $\mathbf{E}Y_{\mathfrak{B}}(x)Y_{\mathfrak{B}}(y)=G_{\overline{\mathfrak{B}}}(x,y)$
for all $\ensuremath{x,y\in\overline{\mathfrak{B}}}$, where $G_{\overline{\mathfrak{B}}}(x,y)$
is the Green's function of simple random walk on $\overline{\mathfrak{B}}$.
(The constant $3$ in the definition of the blow-up is irrelevant
to the result—the point is that Dirichlet boundary conditions are
imposed on a box which is a constant fraction larger.)

Fix an inverse-temperature parameter $\gamma>0$. Let $\mathfrak{B}_{S}=[0,S)^{2}\cap\mathbf{Z}^{2}$.
We define the \emph{Liouville first-passage percolation} metric $\dist_{S}$
on $\mathfrak{B}_{S}$ by
\[
\dist_{S}(x_{1},x_{2})=\min_{\pi}\sum_{x\in\pi}e^{\gamma Y_{\mathfrak{B}_{S}}(x)},
\]
where $\pi$ ranges over all paths in $\mathfrak{B}_{S}$ connecting
$x_{1}$ and $x_{2}$. Given a sequence of normalizing constants $\kappa_{S}$,
we define a metric $\widetilde{\dist}_{S}$ on $[0,1]^{2}\subset\mathbf{R}^{2}$
by letting
\[
\widetilde{\dist}_{S}(x_{1},x_{2})=\frac{1}{\kappa_{S}}\dist_{S}(Sx_{1},Sx_{2})
\]
for each $x_{1},x_{2}\in[0,1]^{2}\cap\frac{1}{S}\mathbf{Z}^{2}$ and
extending to all $x_{1},x_{2}\in[0,1]^{2}$ by linear interpolation.
We will prove the following.
\begin{thm}
\label{thm:subseqconv}There is a $\gamma_{0}>0$ so that if $\gamma<\gamma_{0}$
then there exists a sequence of normalizing constants $\kappa_{S}$
so that, for every sequence of scales $S_{i}$, there is a subsequence
$\{S_{i_{j}}\}$ so that $\widetilde{\dist}_{S_{i_{j}}}$ converges
in distribution (using the Gromov–Hausdorff topology on the space
of metrics) to a limiting metric, which moreover is homeomorphic to
the Euclidean metric by a Hölder-continuous homeomorphism with Hölder-continuous
inverse.
\end{thm}

\begin{rem}
The $\gamma_{0}$ that we are able to establish is so small that calculating
a precise value would be unilluminating. Extending our result to a
``reasonable'' value of $\gamma_{0}$ is an interesting open problem. 
\end{rem}

\subsection{Background and related results}

Substantial effort to date (see~\cite{ADH15,GK12} and their references)
has been devoted to understanding classical first-passage percolation,
with independent and identically distributed edge or vertex weights.
We argue that FPP with strongly-correlated weights is also a rich
and interesting subject, involving questions both analogous to and
distinctive from those asked in the classical case. In particular,
since the Gaussian free field is in some sense the canonical strongly-correlated
random medium, we endeavor to study Liouville FPP—that is, FPP in\textbf{
$\mathbf{Z}^{2}$} with weights given by the exponential of DGFF.

More specifically, Liouville FPP is thought to play a key role in
understanding the random metric associated with the Liouville quantum
gravity (LQG)~\cite{P81,DS11,RV14}. It is a major open problem just
to give a rigorous definition of such a metric. Miller and Sheffield
have recently succeeded in giving such a definition for the case $\gamma=\sqrt{8/3}$;
see~\cite{MS15,MS14,MS15b,MS16a,MS16b} and their references. In
these papers, the authors focused on directly constructing the random
metric in the continuum setup. Other recent work has shown the existence
of scaling exponents for an attempt to construct LQG for $\gamma\in(0,2)$
via ``LQG structure graphs''~\cite{GHS16}.

We take an alternative approach which seeks to understand the random
metric of LQG via scaling limits of lattice approximations using the
DGFF, as proposed (and discussed in more detail) in~\cite{Ben10}.
We choose to work with the square lattice-based Liouville FPP both
for its simple formulation and for its relationship to classical FPP\hspace*{0in}.
Eventually one might wish to tweak the definition of the discrete
metric in order to obtain a scaling limit with more invariance properties.
However, the methods developed in this article are robust to reasonable
changes in the method of discretization. We make this precise by stating
the necessary conditions on the field in \subsecref{field-properties}.

Our result is similar in flavor to~\cite{LeGall07} and~\cite{LP08},
which proved, respectively, that the graph distance of random quadrangulations
has a subsequential scaling limit and that the all possible limiting
metrics are homeomorphic to a 2-sphere. (In our case, however, the
homeomorphism property is a byproduct of the compactness result.)
The uniqueness of the scaling limit, known as the Brownian map, was
proved in later works~\cite{LeGall10,LeGall13,Miermont13}.

A crucial ingredient in~\cite{LeGall07} is a bijection~\cite{CV81,Schaeffer98,BDG04}
between uniform quadrangulations and labeled trees. In particular,
such a bijection allows an explicit evaluation of the order of the
typical distance in the random quadrangulation. By contrast, in our
model, determining the FPP distance exponent seems to be a major challenge.
Indeed, recent works~\cite{DG15,DG16} have shown that the distance
exponent for Liouville FPP is strictly less than $1$ at high temperatures,
and also~\cite{DZ15} that there exists a family of log-correlated
Gaussian fields for which the weight exponent can be arbitrarily close
to $1$. This means that the distance exponent is not universal among
log-correlated Gaussian fields, so precisely computing this exponent
must involve rather subtle properties of the field. Our proof circumvents
this difficulty since it works without knowing the scaling exponent.

\subsection{Proof approach and the RSW method}

The framework of our proof (which we note bears little similarity
to the methods used in~\cite{DG15,DG16}) is a multiscale analysis
procedure relying on several relationships which we establish between
FPP distances at different scales. The key estimates are inductive
upper and lower bounds on crossing distances and geodesic lengths,
in which distances and lengths at a larger scale are estimated in
terms of distances at a smaller scale. Most of the lower bounds on
the larger-scale distances are achieved in~\secref{lowerbounds}
using percolation-type arguments, while the upper bounds on larger-scale
distances and lengths are carried out in~\secref{upperbounds} using
gluing arguments along with the lower bounds. In~\subsecref{diameter},
we use a chaining argument to get an upper bound on box \emph{diameter},
which combined with the lower bounds allows us to inductively bound
the crossing distance coefficient of variation in~\secref{variation}.
Finally, in \secref{limits}, we apply this coefficient of variation
bound to establish tightness, and thus subsequential convergence,
of the normalized FPP metrics. 

Carrying out the above strategy leads to a central problem: lower
bounds on crossing distances are obtained in terms of ``easy crossings''
(between the two longer sides) of rectangles, while upper bounds are
obtained in terms of ``hard crossings'' (between the two shorter
sides). (See \figref{easyhard}\begin{wrapfigure}{O}{1.2in}%
\begin{centering}
\subfloat[Easy crossing]{\begin{centering}
\tiny
\begin{tikzpicture}[x=0.3in,y=0.3in]
\draw [thick] (0,0) -- (3,0) -- (3,1) -- (0,1) -- cycle;
\draw [red,style={decorate,decoration={snake,amplitude=0.4}}] (1.4,0) -- (1.8,1);
\end{tikzpicture}
\par\end{centering}
}
\par\end{centering}
\begin{centering}
\subfloat[Hard crossing]{\begin{centering}
\tiny
\begin{tikzpicture}[x=0.3in,y=0.3in]
\draw [thick] (0,0) -- (3,0) -- (3,1) -- (0,1) -- cycle;

\draw [red,style={decorate,decoration={snake,amplitude=0.4}}] (0,0.4) -- (3,0.5);

\end{tikzpicture}
\par\end{centering}
}
\par\end{centering}
\caption{\label{fig:easyhard}}
\end{wrapfigure}%
) In order to play these bounds off of each other, we must establish
a relationship between easy and hard crossing distances. Results of
this type are known as RSW statements, and the key ingredient in our
results (\secref{rsw}, representing the bulk of the paper) is an
RSW theorem for the Liouville FPP setting.

We briefly review the history of the RSW method, an important technique
in planar statistical physics, which was initiated in~\cite{Russo78,SW78,Russo81}
in order to prove a positive hard crossing probability through a rectangle
in critical Bernoulli percolation. Recently, an RSW theory has been
developed for FK percolation; see e.g.~\cite{DHN11,BDC12,DST15}.
Most relevant to the present paper, an RSW theory was developed in~\cite{tassion}
for Voronoi percolation. In fact, the beautiful method in~\cite{tassion}
is widely applicable to percolation problems satisfying the FKG inequality,
mild symmetry assumptions, and weak correlation between well-separated
regions. For example, in~\cite{DRT16}, this method was used to give
a simpler proof of the result of~\cite{BDC12}, and in~\cite{DMT},
the authors proved an RSW theorem for the crossing probability of
level sets of planar Gaussian free field. The Liouville FPP model
has analogous symmetry and correlation properties, indicating that
the methods in~\cite{tassion} can apply in this setting as well.
Indeed, the geometric framework of our RSW proof is \emph{hugely}
inspired by~\cite{tassion}. Shortly after we posted this article,
in \cite{ATT17} the authors developed a method of comparing easy
and hard crossing probabilities in the study of the Poisson Boolean
percolation model, although their method does not seem to apply to
the geodesic of our FPP metric.

A main novelty of our result is that it seems to be the first RSW
theorem for random planar metrics (rather than for traditional crossing
probabilities for percolation problems). The use of RSW theory in
the metric setting has the potential to enrich both the application
and the theory of the RSW method, and we expect more applications
of RSW theory in the study of random planar metrics. One encounters
substantial challenges working with the FPP weights in our RSW result
even given the beautiful work of~\cite{tassion}: the proof method
of~\cite{tassion} is based on an intricate induction which becomes
even more delicate with the FPP weights taken into account. Besides
that, our FPP metric lacks a natural self-duality, which precludes
using the hypothesis of crossing square boxes as in the traditional
setup; rather, we start with ``easy'' crossings of rectangular boxes.
The difficulties are such that we are only able to relate \emph{different}
quantiles of the FPP distance in different scales, and we have to
apply our induction hypothesis on the variance of the FPP distance
to relate different quantiles at each scale. This introduces an additional
layer of complexity to our arguments.

\subsection{Acknowledgments}

We thank Steve Lalley for encouragement and useful discussions, and
an anonymous referee for a great number of helpful comments.

\section{Preliminaries\label{sec:preliminaries}}

\subsection{Notational conventions\label{subsec:notation}}

Here we introduce notation that we will use throughout the paper.

\subsubsection{Boxes}

Since we will be primarily working in the discrete setting, throughout
the paper, the notation $[a,b)$ will denote the set of integers between
$a$ and $b-1$, inclusive, and $[a,b]$ the set of integers between
$a$ and $b$, inclusive. When we need to refer to an interval of
real numbers, we will attach a subscript $\mathbf{R}$, as in $[a,b]_{\mathbf{R}}$,
etc. A \emph{box} or \emph{rectangle }(we use the terms interchangeably)
is a finite rectangular subset of $\mathbf{Z}^{2}$. We will denote
by $\mathscr{B}$ the set of all boxes in $\mathbf{Z}^{d}$. We will
say that a square box is \emph{dyadic }if its side-length is a power
of $2$ and the coordinates of its bottom-left corner are multiples
of its side-length. As in the introduction, the \emph{blow-up} of
a box $\mathfrak{B}$, denoted $\overline{\mathfrak{B}}$, is the
union of the nine translates of $\mathfrak{B}$ centered around $\mathfrak{B}$.
We say that a rectangular box is \emph{portrait} if its height is
greater than its width and \emph{landscape} if its width is greater
than its height. For boxes $\mathfrak{A}\subseteq\mathfrak{B}$, we
will use the notation $|\mathfrak{B}/\mathfrak{A}|$ to denote the
maximum of the width of $\mathfrak{B}$ divided by the width of $\mathfrak{A}$
and the height of $\mathfrak{B}$ divided by the height of $\mathfrak{A}$.

\subsubsection{Paths\label{subsec:pathnotation}}

Suppose $\pi$ is a path and $Y$ is a random field. Define 
\[
\psi(\pi;Y)=\sum_{x\in\pi}\exp(\gamma Y(x)).
\]
If $\mathfrak{R}$ is a rectangle, let
\[
\Psi_{\mathrm{LR}}(\mathfrak{R};Y)=\min_{\pi}\psi(\pi;Y),
\]
where $\pi$ ranges over all left–right crossings of $\mathfrak{R}$.
Define $\Psi_{\mathrm{BT}}$ analogously for bottom-top crossings.
Also put
\[
\Psi_{\text{easy}}(\mathfrak{R};Y)=\min_{\pi}\psi(\pi;Y)
\]
where $\pi$ ranges over all crossings between the longer sides of
$\mathfrak{R}$, and let
\[
\Psi_{\text{hard}}(\mathfrak{R};Y)=\min_{\pi}\psi(\pi;Y)
\]
where $\pi$ ranges over all crossings between the shorter sides of
$\mathfrak{R}$. (Hence $\pi_{\text{easy}}(\mathfrak{R};Y)=\pi_{\mathrm{LR}}(\mathfrak{R};Y)$
if $\mathfrak{R}$ is portrait, etc.) If a path $\pi$ crosses a box
$\mathfrak{R}$ in the easy (hard) direction, we say that $\pi$ is
an \emph{easy-crossing }(\emph{hard-crossing})\emph{ }of $\mathfrak{R}$,
and we say that $\pi$ \emph{easy-crosses }(\emph{hard-crosses}) $\mathfrak{R}$.
Define for all $x,y\in\mathfrak{R}$
\[
\Psi_{x,y}(\mathfrak{R};Y)=\min_{\pi}\psi(\pi;Y)
\]
where the minimum is taken over all paths $\pi$ connecting $x$ and
$y$ while remaining inside $\mathfrak{R}$. Finally, put
\begin{align*}
\Psi_{\partial}(\mathfrak{R};Y) & =\max_{x,y\in\partial\mathfrak{B}}\Psi_{x,y}(\mathfrak{R};Y),\text{ and} & \Psi_{\max}(\mathfrak{R};Y) & =\max_{x,y\in\mathfrak{B}}\Psi_{x,y}(\mathfrak{R};Y).
\end{align*}

We have now several times defined symbols of the form $\Psi_{\bullet}(\mathfrak{R};Y)$
as the minimum of $\psi(\cdot;Y)$ over some collection of paths.
In each case let $\pi_{\bullet}(\mathfrak{R};Y)$ be the path that
achieves the minimum; if there are multiple such paths (which will
almost surely not happen if the random variables defining the field
have a sufficiently continuous distribution), let one be chosen uniformly
at random, independently of everything else. We also need notation
for the quantile functions for these variables, so let
\[
\Theta_{\bullet}(\mathfrak{R};Y)[p]=\inf\{w\mid\mathbf{P}[\Psi_{\bullet}(\mathfrak{R};Y)\le w]\ge p\}.
\]

For a path $\pi$, let $|\pi|$ denote the length of $\pi$ (that
is, the number of vertices in $\pi$). For $S$ a power of $2$ (less
than the side-length of $\mathfrak{R}$), let $\|\pi\|_{S}$ denote
the number of dyadic square boxes of side-length $S$ entered by $\pi$,
counting each box \emph{once}, even if $\pi$ enters it multiple times.
Let $M_{\bullet;S}(\mathfrak{R};Y)=\|\pi_{\bullet}(\mathfrak{R};Y)\|_{S}$.%

Whenever the field is omitted in the $\Psi$ or $\Theta$ notation,
it will be assumed to be the Gaussian free field on the box in question,
defined as in the introduction as the discrete Gaussian free field
with Dirichlet boundary conditions on the boundary of the blow-up
of the box.

\subsubsection{Asymptotics}

Big-$O$, little-$o$, big-$\Omega$, and little-$\omega$ notation
will be employed, \emph{always} with the limit taken as $\gamma\to0$.
(We recall that we write $f(x)=\Omega(g(x))$ if $g(x)=O(f(x))$ and
$f(x)=\omega(g(x))$ if $g(x)=o(f(x))$.) Subscripts will be employed
to indicate that the limit holds for any \emph{fixed} value of the
variable(s) in the subscript, and uniformly in all other variables.
(For example, we could write $\sin(2^{K}\gamma)=o_{K}(1)$.) Most
importantly, the limit is \emph{always} \emph{uniform} over all scales.
We will also work with many constants throughout the proofs. The important
point regarding any constant is that it is independent of the scale.
Constants that will be referenced in later sections will be denoted
by a mnemonic subscript.

\subsection{Properties of the field\label{subsec:field-properties}}

While we have stated our results for first-passage percolation on
the discrete Gaussian free field, we do not require any particularly
fine properties of this field. In this section we collect the necessary
facts about the DGFF, and summarize them in Criteria~\ref{cri:centeredpos}–\ref{cri:localrandomness}.
However, before we can do this we first must precisely define the
DGFF, in particular the relationship between the DGFF defined on different
boxes.

\subsubsection{Coupling of fields in different boxes}

Although we defined the discrete Gaussian free field on a box in the
introduction, in order to perform the multi-scale analysis we use
in this article it will be convenient to couple the fields on all
different finite boxes in $\mathbf{Z}^{2}$ simultaneously. We recall
the \emph{Markov field property }of the Gaussian free field: that
if $Y$ is a Gaussian free field with Dirichlet boundary conditions
on $\mathfrak{B}$, then $Y_{\mathfrak{B}}-\mathbf{E}\left[Y_{\mathfrak{B}}\,\middle|\,\left(Y_{\mathfrak{B}}\restriction\partial\overline{\mathfrak{A}}\right)\right]$
defines a discrete Gaussian free field with Dirichlet boundary conditions
on $\overline{\mathfrak{A}}$, which moreover is independent of $Y_{\mathfrak{B}}\restriction(\overline{\mathfrak{B}}\setminus\overline{\mathfrak{A}})$.
Here, $\restriction$ denotes restriction of the field.

Let $\mathfrak{B}_{N}=[-N,N]^{2}$. Let $Y_{\mathfrak{B}_{N}}^{(N)}$
be a discrete Gaussian free field with Dirichlet boundary conditions
on $\overline{\mathfrak{B}_{N}}$. Now for all boxes $\mathfrak{B}\subset\mathfrak{B}_{N}$,
for each $x\in\mathfrak{B}$ define $Y_{\mathfrak{B}}^{(N)}(x)=Y_{\mathfrak{B}_{N}}^{(N)}(x)-\mathbf{E}\left[Y_{\mathfrak{B}_{N}}^{(N)}(x)\,\middle|\,\left(Y_{\mathfrak{B}_{N}}^{(N)}\restriction\partial\overline{\mathfrak{B}}\right)\right]$.
Now we note that whenever $N'\ge N$, the process $\{Y_{\mathfrak{B}}^{(N)}\mid\mathfrak{B}\subset\mathfrak{B}_{N}\}$
agrees in law with the process $\{Y_{\mathfrak{B}}^{(N')}\mid\mathfrak{B}\subset\mathfrak{B}_{N}\}$.
Indeed, if we put $Y_{\mathfrak{B}_{N}}^{(N)}$ and $Y_{\mathfrak{B}_{N'}}^{(N')}$
on the same probability space so that $Y_{\mathfrak{B}_{N}}^{(N)}=Y_{\mathfrak{B}_{N}}^{(N')}$,
then for all $x\in\mathfrak{B}$ we have
\begin{align*}
Y_{\mathfrak{B}}^{(N)}(x) & =Y_{\mathfrak{B}_{N}}^{(N)}(x)-\mathbf{E}\left[Y_{\mathfrak{B}_{N}}^{(N)}(x)\,\middle|\,\left(Y_{\mathfrak{B}_{N}}^{(N)}\restriction\partial\overline{\mathfrak{B}}\right)\right]\\
 & =Y_{\mathfrak{B}_{N}}^{(N')}(x)-\mathbf{E}\left[Y_{\mathfrak{B}_{N}}^{(N')}(x)\,\middle|\,\left(Y_{\mathfrak{B}_{N}}^{(N')}\restriction\partial\overline{\mathfrak{B}}\right)\right]\\
 & =Y_{\mathfrak{B}_{N'}}^{(N')}(x)-\mathbf{E}\left[Y_{\mathfrak{B}_{N'}}^{(N')}(x)\,\middle|\,\left(Y_{\mathfrak{B}_{N'}}^{(N')}\restriction\partial\overline{\mathfrak{B}_{N}}\right)\right]\\
 & \quad-\mathbf{E}\left[Y_{\mathfrak{B}_{N'}}^{(N')}(x)-\mathbf{E}\left[Y_{\mathfrak{B}_{N'}}^{(N')}(x)\,\middle|\,\left(Y_{\mathfrak{B}_{N'}}^{(N')}\restriction\partial\overline{\mathfrak{B}_{N}}\right)\right]\,\middle|\,\left(Y_{\mathfrak{B}_{N}}^{(N')}\restriction\partial\overline{\mathfrak{B}}\right)\right]\\
 & =Y_{\mathfrak{B}_{N'}}^{(N')}(x)-\mathbf{E}\left[Y_{\mathfrak{B}_{N'}}^{(N')}(x)\,\middle|\,\left(Y_{\mathfrak{B}_{N}}^{(N')}\restriction\partial\overline{\mathfrak{B}}\right)\right]\\
 & =Y_{\mathfrak{B}}^{(N')}(x),
\end{align*}
where the second-to-last equality is by the tower property of conditional
expectation and the independence statement in the Markov field property.
Thus, since all of the processes are Gaussian, using Kolmogorov's
extension theorem we can, on a single probability space, simultaneously
define $Y_{\mathfrak{B}}$ for every $\mathfrak{B}\in\mathscr{B}$
in such a way that whenever $\mathfrak{A}\subset\mathfrak{B}$, we
have, for all $x\in\mathfrak{A}$,
\begin{equation}
Y_{\mathfrak{A}}(x)=Y_{\mathfrak{B}}(x)-\mathbf{E}\left[Y_{\mathfrak{B}}(x)\,\middle|\,\left(Y_{\mathfrak{B}}\restriction\partial\overline{\mathfrak{A}}\right)\right].\label{eq:dgffcoupling}
\end{equation}
Henceforth, we will assume that the DGFFs on different boxes have
been coupled in this way, so that in particular \eqref{dgffcoupling}
holds.

\subsubsection{Description of the criteria for the field}

Throughout the paper, we will consider a collections of real-valued
random variables (the ``field''), denoted $\{Y_{\mathfrak{B}}(x)\,:\,\mathfrak{B}\in\mathscr{B},x\in\mathfrak{B}\}$,
and we will always assume the following five properties.
\begin{criterion}
\label{cri:centeredpos}The field $\{Y_{\mathfrak{B}}(x)\mid\mathfrak{B}\in\mathscr{B},x\in\mathfrak{B}\}$
is a centered Gaussian process which moreover is non-negatively correlated:
for all $\mathfrak{B}_{1},\mathfrak{B}_{2}\in\mathscr{B}$, $x_{1}\in\mathfrak{B}_{1}$,
$x_{2}\in\mathfrak{B}_{2}$, we have $\Cov(Y_{\mathfrak{B}_{1}}(x_{1}),Y_{\mathfrak{B}_{2}}(x_{2}))\ge0$.
\end{criterion}

\begin{criterion}
\label{cri:invariance}If $\theta$ is a Euclidean isometry of $\mathbf{R}^{2}$
which preserves $\mathbf{Z}^{2}$, then the indexed families of random
variables $\{Y_{\mathfrak{B}}(x)\mid\mathfrak{B}\in\mathscr{B},x\in\mathfrak{B}\}$
and $\{Y_{\theta(\mathfrak{B})}(\theta(x))\mid\mathfrak{B}\in\mathscr{B},x\in\mathfrak{B}\}$
agree in distribution.
\end{criterion}

\begin{criterion}
\label{cri:disjoint-independent}If $\overline{\mathfrak{B}_{1}}$
and $\overline{\mathfrak{B}_{2}}$ are disjoint, then $Y_{\mathfrak{B}_{1}}$
and $Y_{\mathfrak{B}_{2}}$ are independent.
\end{criterion}

\begin{criterion}
\label{cri:coursebound}There are constants $C,C_{\mathrm{F}}>0$
so that if $\mathfrak{A}\subset\mathfrak{B}$ are nested rectangles,
then we have, for all $u\ge0$,
\[
\mathbf{P}\left(\max_{x\in\mathfrak{A}}\left|Y_{\mathfrak{A}}(x)-Y_{\mathfrak{B}}(x)\right|\ge C_{\mathrm{F}}+u\right)\le\exp\left(-\frac{Cu^{2}}{\log|\mathfrak{B}/\mathfrak{A}|}\right).
\]
\end{criterion}

\begin{criterion}
\label{cri:localrandomness}There is an absolute constant $C$ so
that the following holds. For each rectangle $\mathfrak{B}$ with
a partition of its blow-up $\overline{\mathfrak{B}}=\bigsqcup\limits _{i=1}^{r}\mathfrak{B}_{i}$
into squares $\mathfrak{B}_{i}$ of uniform side-length $S$, there
is a stochastic process $\{Z_{i}\}_{i=1}^{r}$ so that $Z_{1},\ldots,Z_{r}$
are independent, $Y_{\mathfrak{B}}(x)\in\sigma(Z_{1},\ldots,Z_{r})$
for all $x\in\mathfrak{B}$, and whenever $1\le j\le r$ and $x\in\mathfrak{B}\setminus\overline{\mathfrak{B}_{j}}$,
we have 
\begin{align}
\Var(Y_{\mathfrak{B}}(x)\mid Z_{1},\ldots,\widehat{Z_{j}},\ldots,Z_{r}) & \le C\label{eq:resamplevar}\\
\Var(Y_{\mathfrak{B}}(x)-Y_{\mathfrak{B}}(y)\mid Z_{1},\ldots,\widehat{Z_{j}},\ldots,Z_{r}) & \le C\|x-y\|^{2}S^{2}/N^{4}\label{eq:resamplediffvar}
\end{align}
where the hat means that $Z_{j}$ is excluded and $N$ is the length
of the shorter side of $\mathfrak{B}$.
\end{criterion}

\begin{rem}
In the definition and use of the Markov field property above, we considered
$Y_{\mathfrak{B}}(x)$ for $x\in\overline{\mathfrak{B}}$ (i.e. not
in $\mathfrak{B}$ itself). This was important for defining the coupling
but in the sequel we will only consider the values of $Y_{\mathfrak{B}}$
on $\mathfrak{B}$ itself. 
\end{rem}

\begin{rem}
Although, in order to avoid the complexity of multiple cases, we will
not consider this case in detail, we invite the reader to check that
all of the arguments in the paper go through as well for \emph{continuous}
approximations of the GFF: that is, fields $\{Y_{\mathfrak{B}}(x)\,:\,\mathfrak{B}\in\mathfrak{B},x\in\mathfrak{B}\}$
satisfying Criteria~\ref{cri:centeredpos}–\ref{cri:localrandomness},
where the weight of a (continuous) path $\xi:[0,1]\to\mathfrak{B}$
is given by
\[
\int_{0}^{1}e^{\gamma Y_{\mathfrak{B}}(\xi(t))}|\xi'(t)|\,dt.
\]
In fact, certain technical parts of the argument (such as one case
in the proof of \lemref{setuplambda}, and the linear interpolation
given in \eqref{linear-interpolation} in the sequel) become unnecessary
in the continuous case.
\end{rem}

\subsubsection{Proof that the DGFF satisfies the criteria}

We now demonstrate that the DGFF indeed satisfies the criteria that
we have just laid out. (A much gentler introduction to these properties
is available in \cite{biskup}.) Coupled as above, the DGFF satisfies
Criteria~\ref{cri:centeredpos} and~\ref{cri:invariance}. To show
\criref{coursebound} for the DGFF, we first note that, by Fernique's
criterion (see~\cite{Fer75} and~\cite[Theorem 4.1]{A90} or \cite[Theorem 6.6]{biskup})
and a covariance estimate on the conditional expectation field, as
in \cite[Lemmas 3.5 and 3.10]{BDZ14}, we have a constant $C_{\mathrm{F}}$
so that
\[
\mathbf{E}\left[\max_{x\in\mathfrak{A}}\mathbf{E}\left[Y_{\mathfrak{B}}(x)\,\middle|\,Y_{\mathfrak{B}}\restriction\partial\overline{\mathfrak{A}}\right]\right]<C_{\mathrm{F}}.
\]
Moreover, the variance of $\mathbf{E}\left[Y_{\mathfrak{B}}(x)\,\middle|\,Y_{\mathfrak{B}}\restriction\partial\overline{\mathfrak{A}}\right]$
can be bounded (uniformly over $x\in\mathfrak{A}$) by a constant
times $\log|\mathfrak{B}/\mathfrak{A}|$. These two facts, along with
the Borell–TIS inequality (see, for example, \cite[Theorem 7.1]{L01},
\cite[Theorem 6.1]{biskup}, or \cite[Theorem 2.1]{A90}) imply that
\[
\mathbf{P}\left(\max_{x\in\mathfrak{A}}\mathbf{E}\left[Y_{\mathfrak{B}}(x)\,\middle|\,Y_{\mathfrak{B}}\restriction\partial\overline{\mathfrak{A}}\right]\ge C_{\mathrm{F}}+u\right)\le\exp\left(-\frac{Cu^{2}}{\log|\mathfrak{B}/\mathfrak{A}|}\right).
\]

Finally we will prove \criref{localrandomness} using the ``resistor''
definition of the DGFF (see for example \cite[p. 52]{LP:book}). For
each edge $e$ in the nearest-neighbor graph on $\mathfrak{B}$, let
$\xi(e)$ be a standard normal random variable, independent from $\xi(e')$
for each $e'\ne e$. Then, as in \cite[(2.25)]{LP:book} we have the
alternative definition of Gaussian free field on $\mathfrak{B}$ as
\[
Y_{\mathfrak{B}}(x)=\sum_{e}i_{x}(e)\xi(e),
\]
where $i_{x}(e)$ is the flow through $e$ of a unit electric current
from $x$ to $\partial\overline{\mathfrak{B}}$, where the lattice
is treated as an electrical network with unit resistance on each edge.
Let $Z_{i}=(i_{x}(e)\,:\,e\in\mathfrak{B}_{i})$. Now if $x\in\mathfrak{B}$,
we have 
\begin{equation}
\Var(Y_{\mathfrak{B}}(x)\mid Z_{1},\ldots,\widehat{Z_{j}},\ldots,Z_{r})=\sum_{e\in\mathfrak{B}_{j}}(i_{x}(e))^{2}.\label{eq:Xibound-proto}
\end{equation}
By~\cite[Proposition 2.2]{LP:book}, we have
\[
i_{x}(e)=\frac{G_{\mathfrak{\overline{B}}}(x,e_{+})}{\deg(e_{+})}-\frac{G_{\mathfrak{\overline{B}}}(x,e_{-})}{\deg(e_{-})},
\]
so
\[
|i_{x}(e)|=\frac{1}{4}|G_{\overline{\mathfrak{B}}}(x,e_{+})-G_{\mathfrak{\overline{B}}}(x,e_{-})|,
\]
where $e_{-}$ and $e_{+}$ denote the two endpoints of $e$ and $G_{\overline{\mathfrak{B}}}$
denotes the Green's function for simple random walk stopped on the
boundary of $\overline{\mathfrak{B}}$. But by~\cite[Proposition 4.6.2(b), Theorem 4.4.4]{lawlerlimic},
we have
\[
G_{\overline{\mathfrak{B}}}(x,y)=\mathbf{E}^{x}[a(Q_{\tau_{\overline{\mathfrak{B}}}},y)]-a(x,y),
\]
where $\{Q_{t}\}$ is a simple random walk, $\tau_{\overline{\mathfrak{B}}}$
is the hitting time of $\partial\overline{\mathfrak{B}}$, $\mathbf{E}^{x}$
is the expectation with respect to the law of $\{Q_{t}\}$ started
at $x$, and
\[
a(x,y)=\frac{2}{\pi}\log|x-y|+\frac{2C_{1}+\log8}{\pi}+O(|x-y|^{-2}),
\]
where $C_{1}\approx0.577$ is the Euler–Mascheroni constant (usually
denoted $\gamma$) and the big-O notation is taken as $x-y\to\infty$.
An easy computation implies that, if $x\in\mathfrak{B}\setminus\overline{\mathfrak{B}_{j}}$,
then $|i_{x}(e)|\le C_{2}/S$ for some constant $C_{2}$, where $S$
is the side-length of $\mathfrak{B}_{j}$. Combining this with \eqref{Xibound-proto}
shows that $\Var(Y_{\mathfrak{B}}(x)\mid Z_{1},\ldots,\widehat{Z_{j}},\ldots,Z_{r})$
is bounded by a constant for all $x\in\mathfrak{B}\setminus\overline{\mathfrak{B}_{j}}$.
This completes the proof of \criref{localrandomness} for DGFF.

Finally, calculations similar to those in the proofs of Corollary~4.4.5
and Lemma~6.3.3 in~\cite{lawlerlimic} show that there is a constant
$C_{3}$ so that if $|x-y|=1$, then
\begin{align*}
|i_{x}(e)-i_{y}(e)| & =\frac{1}{4}|G_{\overline{\mathfrak{B}}}(x,e_{+})-G_{\mathfrak{\overline{B}}}(x,e_{-})-G_{\overline{\mathfrak{B}}}(y,e_{+})+G_{\mathfrak{\overline{B}}}(y,e_{-})|\le\frac{C_{3}}{N^{2}},
\end{align*}
where as above $N$ is the length of the shorter side of $\mathfrak{B}$.
By the triangle inequality, we then have that for any $x,y\in\mathfrak{B}$,
\[
|i_{x}(e)-i_{y}(e)|\le C_{3}\frac{\|x-y\|}{N^{2}}.
\]
Therefore, we have
\begin{align*}
\Var(Y_{\mathfrak{B}}(x)-Y_{\mathfrak{B}}(y)\mid Z_{1},\ldots,\widehat{Z_{j}},\ldots,Z_{r}) & =\sum_{e\in\mathfrak{B}_{j}}\left(i_{x}(e)-i_{y}(e)\right)^{2}\\
 & \le C_{3}^{2}\|x-y\|^{2}S^{2}/N^{4},
\end{align*}
establishing \eqref{resamplediffvar}.

\subsubsection{Further properties of the field}

We now record some important consequences of Criteria~\ref{cri:centeredpos}–\ref{cri:localrandomness}
that we will use throughout the paper. The first is a translation
of \criref{coursebound} into the exponentiated setting. Indeed, \criref{coursebound}
implies that
\begin{equation}
\max_{x\in\mathfrak{A}}\left|\frac{e^{\gamma Y_{\mathfrak{B}}(x)}}{e^{\gamma Y_{\mathfrak{A}}(x)}}\right|=1+o(1)\label{eq:relatescale}
\end{equation}
in probability (as $\gamma\to0$). More precisely, there is an absolute
constant $u_{0}>1$ so that if $u\ge u_{0}$ then we have
\begin{multline}
\mathbf{P}\left(\max_{x\in\mathfrak{A}}\left|\frac{e^{\gamma Y_{\mathfrak{B}}(x)}}{e^{\gamma Y_{\mathfrak{A}}(x)}}\right|\ge u\right)+\mathbf{P}\left(\max_{x\in\mathfrak{A}}\left|\frac{e^{\gamma Y_{\mathfrak{B}}(x)}}{e^{\gamma Y_{\mathfrak{A}}(x)}}\right|\le\frac{1}{u}\right)\\
=\mathbf{P}\left(\max_{x\in\mathfrak{A}}|\gamma Y_{\mathfrak{B}}(x)-\gamma Y_{\mathfrak{A}}(x)|\ge\log u\right)\\
\le\exp\left(-\omega(1)\cdot\frac{(\log u)^{2}}{\log|\mathfrak{B}/\mathfrak{A}|}\right).\label{eq:tailgff}
\end{multline}

Another key ingredient, implied by \criref{centeredpos}, is the celebrated
FKG inequality.
\begin{lem}[FKG inequality]
\label{lem:FKG}If $f$ and $g$ are increasing functions of $Y=\{Y_{\mathfrak{B}}(x)\mid\mathfrak{B}\in\mathscr{B},x\in\mathfrak{B}\}$,
then
\begin{equation}
\mathbf{E}f(Y)g(Y)\ge\mathbf{E}f(Y)\mathbf{E}g(Y).\label{eq:mainfkg}
\end{equation}
\end{lem}

See~\cite{pitt} for a proof of the FKG inequality for general Gaussian
processes with non-negative correlations, of which \eqref{mainfkg}
is an application. The following corollary is typical of our applications
of the FKG inequality.
\begin{cor}
\label{corr:stretch} Let $a>b$ and $S,k\in\mathbf{N}$, and put
$\mathfrak{A}=[0,aS)\times[0,bS)$ and $\mathfrak{B}=[0,(ka-(k-1)b)S)\times[0,bS)$.
Then 
\[
\mathbf{P}[\Psi_{\mathrm{LR}}(\mathfrak{B})\le2ky]\ge\mathbf{P}[\Psi_{\mathrm{LR}}(\mathfrak{A})\le y]^{2k-1}-o_{k}(1).
\]
\end{cor}

\begin{proof}
\begin{figure}[t]
\begin{centering}
\tiny
\begin{tikzpicture}[x=0.5in,y=0.5in]
\draw[step=1,thin] (0,-1) grid (6,1);
\draw[white] (0,-1) -- (0,0);
\draw[white] (0,-1) -- (1,-1);
\draw[white] (5,-1) -- (6,-1);
\draw[white] (6,-1) -- (6,0);
\draw[red,style={decorate,decoration={snake,amplitude=0.8}}] (0,0.7) -- (2,0.7);
\draw[red,style={decorate,decoration={snake,amplitude=0.8}}] (1,0.2) -- (3,0.2);
\draw[red,style={decorate,decoration={snake,amplitude=0.8}}] (2,0.5) -- (4,0.5);
\draw[red,style={decorate,decoration={snake,amplitude=0.8}}] (3,0.8) -- (5,0.8);
\draw[red,style={decorate,decoration={snake,amplitude=0.8}}] (4,0.4) -- (6,0.4);
\draw[red,style={decorate,decoration={snake,amplitude=0.8}}] (1.6,-1) -- (1.1,1);
\draw[red,style={decorate,decoration={snake,amplitude=0.8}}] (2.4,-1) -- (2.7,1);
\draw[red,style={decorate,decoration={snake,amplitude=0.8}}] (3.8,-1) -- (3.7,1);
\draw[red,style={decorate,decoration={snake,amplitude=0.8}}] (4.1,-1) -- (4.4,1);
\end{tikzpicture}
\par\end{centering}
\caption{\label{fig:hardglue}Gluing strategy in \corrref{stretch} for $a=2b$
and $k=5$.}
\end{figure}
This follows from \eqref{relatescale} and the \nameref{lem:FKG}
by a ``gluing argument,'' illustrated in \figref{hardglue}.
\end{proof}

\section{Inductive hypothesis}

The key ingredient for all of our results is an inductive bound on
the coefficient of variation for the FPP crossing distance of a rectangle.
For any random variable $X$, we define the notation
\[
\CV^{2}(X)=\frac{\Var X}{(\mathbf{E}X)^{2}}.
\]

\begin{thm}
\label{thm:maintheorem}Let $\delta>0$. If $\gamma$ is sufficiently
small compared to $\delta$, then for all boxes $\mathfrak{R}$ of
aspect ratio between $1/2$ and $2$ inclusive, we have $\CV^{2}(\Psi_{\mathrm{LR}}(\mathfrak{R}))<\delta$.
\end{thm}

The bulk of the paper will be devoted to proving \thmref{maintheorem}
by induction on the scale. Actually, we will have to use the slightly
stronger inductive hypothesis that the coefficient of variation is
below a fixed $\delta_{0}$. The following lemma, which is an easy
consequence of Chebyshev's inequality, will be key to our induction.
\begin{lem}
\label{lem:relatequantiles}Let $X$ and $Y$ be nonnegative random
variables. Define $F_{X}(x)=\mathbf{P}(X\le x)$ and $F_{Y}(y)=\mathbf{P}(Y\le y)$;
let $\Theta_{X}=F_{X}^{-1}$ and $\Theta_{Y}=F_{Y}^{-1}$ be the corresponding
quantile functions. If $\CV^{2}(X)<\delta<p<1$ and $\CV^{2}(Y)<\varepsilon<q<1$,
then there are constants $0<A\le B$, depending only on $\delta,\varepsilon,p,q$
(and not on the random variables $X,Y$) so that
\begin{equation}
A\cdot\frac{\Theta_{X}(p)}{\Theta_{Y}(q)}\le\frac{\mathbf{E}X}{\mathbf{E}Y}\le B\cdot\frac{\Theta_{X}(p)}{\Theta_{Y}(q)}.\label{eq:concentration-main}
\end{equation}
Suppose moreover that $\delta<p'$ and $\varepsilon<q'$. Then there
is a constant $B'>0$, depending only on $\delta,\varepsilon,p,q,p',q'$,
so that
\begin{equation}
\frac{\Theta_{X}(p')}{\Theta_{Y}(q')}\le B'\cdot\frac{\Theta_{X}(p)}{\Theta_{Y}(q)}.\label{eq:concentration-quantiles}
\end{equation}
\end{lem}

While the statement of \lemref{relatequantiles} is rather involved,
the content of the lemma is simply that upper bounds on the coefficients
of variation of two random variables let us multiplicatively relate
non-extreme quantiles and means of the random variables.

\section{Crossing quantile lower bounds\label{sec:lowerbounds}}

Our goal in this section is to obtain lower bounds on quantiles of
the left-right crossing distance of a large box in terms of the easy
crossing quantiles of smaller boxes. We first define and introduce
basic properties of what we call \emph{passes}, which represent smaller
boxes through which a path through a larger box must cross.

Let $K,L\ge2$ and let $S=2^{s}$. Let $\mathfrak{R}=[0,KS)\times[0,LS)$.
\begin{defn}
A \emph{pass} $\mathfrak{P}$ of $\mathfrak{R}$ at scale $S$ is
a $2S\times S$ or $S\times2S$ dyadic subrectangle of $\mathfrak{R}$.
\end{defn}

\begin{lem}
\label{lem:adjacent-pass}Let $\pi$ be a left–right crossing of $\mathfrak{R}$.
If $\pi$ enters an $S\times S$ box $\mathfrak{C}$ such that $\overline{\mathfrak{C}}\subseteq\mathfrak{R}$,
then $\pi$ must easy-cross a pass that intersects $\overline{\mathfrak{C}}$.
\end{lem}

\begin{proof}
Since $\pi$ is a left–right crossing of $\mathfrak{R}$, $\pi$ must
at some point leave $\overline{\mathfrak{C}}$. And it is easy to
see that in order to easy-cross from the inside to the outside of
the annulus $\overline{\mathfrak{C}}\setminus\mathfrak{C}$, $\pi$
must easy-cross a pass intersecting $\overline{\mathfrak{C}}$.
\end{proof}
\begin{defn}
For a path $\pi$, let $\mathcal{P}(\pi)$ be a maximum-size collection
of passes $\mathfrak{P}$ easy-crossed by $\pi$ such that the $\overline{\mathfrak{P}}$s
are disjoint. (See \figref{passes}.)
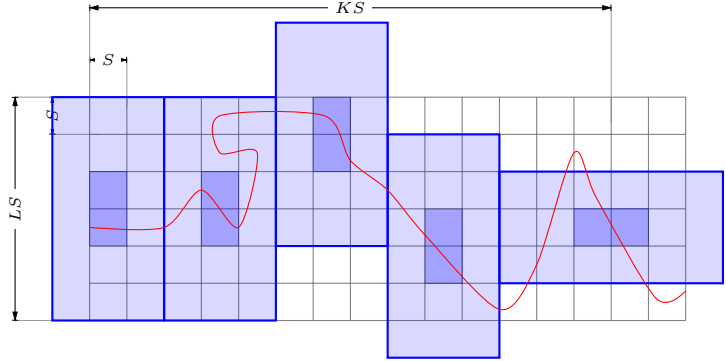
\begin{figure}[t]
\begin{centering}
\tiny
\begin{tikzpicture}[x=0.39in,y=0.39in]
\dimline[extension start length=0.25in,extension end length=0.25in]{(-0.5,2.5)}{(-0.5,3)}{$S$}
\dimline[extension start length=0.25in,extension end length=0.25in]{(0,3.5)}{(0.5,3.5)}{$S$}
\dimline[extension start length=0.5in,extension end length=0.5in]{(-1,0)}{(-1,3)}{$LS$}
\dimline[extension start length=0.6in,extension end length=0.6in]{(0,4.2)}{(7,4.2)}{$KS$}
\draw[step=0.5,gray,thin] (0,0) grid (8,3);
\draw[fill=blue,opacity=0.25] (0,1) rectangle (0.5,2);
\draw[thick,blue,fill=blue,fill opacity=0.15] (-0.5,0) rectangle (1,3);
\draw[fill=blue,opacity=0.25] (1.5,1) rectangle (2,2);
\draw[thick,blue,fill=blue,fill opacity=0.15] (1,0) rectangle (2.5,3);
\draw[fill=blue,opacity=0.25] (3,2) rectangle (3.5,3);
\draw[thick,blue,fill=blue,fill opacity=0.15] (2.5,1) rectangle (4,4);
\draw[fill=blue,opacity=0.25] (4.5,0.5) rectangle (5,1.5);
\draw[thick,blue,fill=blue,fill opacity=0.15] (4,-0.5) rectangle (5.5,2.5);
\draw[fill=blue,opacity=0.25] (6.5,1) rectangle (7.5,1.5);
\draw[thick,blue,fill=blue,fill opacity=0.15] (5.5,0.5) rectangle (8.5,2);
\draw[red] plot [smooth] coordinates { (0,1.25) (1,1.25) (1.5,1.75) (2,1.25) (2.25,2.25) (1.75,2.25) (1.75, 2.75) (3.15,2.75) (3.5,2.15) (4,1.75) (4.5,1.15) (5.5,0.15) (6,0.75) (6.5,2.25)  (6.8,1.65) (7.6,0.3) (8,0.4) };
\end{tikzpicture}
\par\end{centering}
\caption{\label{fig:passes}$\mathcal{P}(\pi)$ for a crossing $\pi$. The
darker boxes are the $\mathfrak{P}_{i}$s while the lighter, surrounding
boxes are the $\overline{\mathfrak{P}_{i}}s$. }
\end{figure}
 For $N\le|\mathcal{P}(\pi)|$, let $\mathcal{P}_{N}(\pi)=\mathcal{P}(\xi)$
where $\xi$ is the minimal initial subpath of $\pi$ such that $|\mathcal{P}(\xi)|\ge N$.
\end{defn}

\begin{prop}
\label{prop:longpathpasses}There is a constant $c_{\mathrm{PD}}$
so that $|\mathcal{P}(\pi)|\ge c_{\mathrm{PD}}\|\pi\|_{S}$. (The
subscript stands for ``pass density.'')
\end{prop}

\begin{proof}
This follows easily from \lemref{adjacent-pass} and the fact that
passes are of a fixed size.
\end{proof}
\begin{lem}
\label{lem:leftrightpasses}If $\pi$ is a left–right crossing of
$\mathfrak{R}$, then $|\mathcal{P}(\pi)|\ge K/6$.
\end{lem}

\begin{proof}
In order for $\pi$ to cross each column of width $S$, it must easy-cross
a pass contained entirely within that column, and the blow-up of this
pass can have width at most $6$.
\end{proof}
\begin{lem}
\label{lem:subgraph-count}Let $G$ be a graph of maximum degree $d$,
and let $\{a_{1},\ldots,a_{M}\}$ be an arbitrary subset of the vertices
of $G$. Then the number of $n$-vertex connected subgraphs $H$ of
$G$ containing at least one $a_{i}$ is at most $Md^{2n}$.
\end{lem}

\begin{proof}
It is easy to see that every connected graph on $n$ vertices contains
a circuit of length at most $2n$ that visits every vertex. Thus the
number of subgraphs $H$ as specified in the statement is bounded
by the number of walks of length at most $2n$ starting at one of
the $a_{i}$s, which is evidently bounded by $Md^{2n}$.
\end{proof}
\begin{prop}
\label{prop:passchoice}We have constants $C_{\mathrm{p}}$ and $d_{\mathrm{p}}$
so that
\[
\left|\left\{ \mathcal{P}_{N}(\pi)\,:\,\text{\ensuremath{\pi\ }a left–right crossing of \ensuremath{\mathfrak{R}} such that \ensuremath{|\mathcal{P}(\pi)|\ge N}}\right\} \right|\le C_{\mathrm{p}}Ld_{\mathrm{p}}^{2N}.
\]
\end{prop}

\begin{proof}
Define a graph $G$ on the set of all passes inside $\mathfrak{R}$
by saying that two passes are adjacent if they could occur as adjacent
passes in a $\mathcal{P}(\pi)$. It is easy to see using \lemref{adjacent-pass}
that $G$ has bounded degree. Then by definition, $\mathcal{P}_{N}(\pi)$
induces an $N$-element connected subgraph of $G$, which in particular
contains a pass that intersects the first six columns, of which there
are at most a constant times $L$. \lemref{subgraph-count} then implies
the desired result.
\end{proof}
Before we prove the main proposition of this section, we need a version
of the Chernoff bound.
\begin{lem}
\label{lem:hugedeviations}Let $p<\frac{1}{2}$ and $X_{1},\ldots,X_{N}$
be iid $\Bernoulli(p)$ random variables. Then $\mathbf{P}\left[\frac{1}{N}\sum_{i=1}^{N}X_{i}\ge\frac{1}{2}\right]\le\left(4p\right)^{N/2}$.
\end{lem}

\begin{proof}
We have 
\begin{multline*}
\mathbf{P}\left[\frac{1}{N}\sum_{i=1}^{N}X_{i}\ge\frac{1}{2}\right]=\mathbf{P}\left[\exp\left(\lambda\cdot\sum_{i=1}^{N}X_{i}\right)\ge e^{\lambda N/2}\right]\\
\le\frac{\left(\mathbf{E}e^{\lambda X_{i}}\right)^{N}}{e^{\lambda N/2}}=\left(\frac{pe^{\lambda}+1-p}{e^{\lambda/2}}\right)^{N}.
\end{multline*}
Putting $\lambda=\log\frac{1-p}{p}$ and using the fact that $p<1/2$
yields the result.
\end{proof}
Now we can prove an inductive lower bound on the crossing LFPP distance.
\begin{prop}
\label{prop:lowerbound-tail}Let $S=2^{s}$ and let $K,L\in\mathbf{N}$.
Let $\mathfrak{R}=[0,KS)\times[0,LS)$ and $\mathfrak{A}=[0,S)\times[0,2S)$.
Then, for any $p\in(0,1/2)$ and any $u\ge u_{0}$ (defined in \eqref{tailgff})
, we have
\begin{multline*}
\mathbf{P}\left[\min_{\pi}\psi(\pi;Y_{\mathfrak{R}})\le\frac{N}{2u}\Theta_{\mathrm{easy}}(\mathfrak{A})[p]\right]\\
\le C_{\mathrm{p}}L\left(2d_{\mathrm{p}}^{2}\sqrt{p}\right)^{N}+KL\exp\left(-\omega(1)\cdot\frac{(\log u)^{2}}{\log(K\vee L)}\right)
\end{multline*}
where the minimum is taken over all paths $\pi$ from left to right
in $\mathfrak{R}$ with $|\mathcal{P}(\pi)|\ge N$.
\end{prop}

\begin{proof}
As long as $\gamma$ is sufficiently small compared to $K$ and $L$,
and $u\ge u_{0}$, we have
\begin{align}
\mathbf{P}\bigg[\min_{\pi}\mathrlap{\psi(\pi;Y_{\mathfrak{R}})\le\frac{N}{2u}\Theta_{\mathrm{easy}}(\mathfrak{A})[p]\bigg]}\label{eq:lowerbound-tail-start}\\
 & \le\mathbf{P}\bigg[\min_{\pi}\frac{1}{N}\sum_{\mathfrak{P}\in\mathcal{P}_{N}(\pi)}\mathbf{1}\{\Psi_{\mathrm{easy}}(\mathfrak{P};Y_{\mathfrak{R}})\le\tfrac{1}{u}\Theta_{\mathrm{easy}}(\mathfrak{A})[p]\}\ge\frac{1}{2}\bigg]\nonumber \\
 & \le\mathbf{P}\bigg[\min_{\pi}\frac{1}{N}\sum_{\mathfrak{P}\in\mathcal{P}_{N}(\pi)}\mathbf{1}\{\Psi_{\mathrm{easy}}(\mathfrak{P})\le\Theta_{\mathrm{easy}}(\mathfrak{A})[p]\}\ge\frac{1}{2}\bigg]\nonumber \\
 & \qquad\qquad+KL\exp\left(-\omega(1)\cdot\frac{(\log u)^{2}}{\log(K\vee L)}\right),\nonumber 
\end{align}
where in the second inequality we use \eqref{tailgff} and \propref{passchoice}.
Now let $X_{1},\ldots,X_{N}$ be iid copies of $\mathbf{1}\{\Psi_{\mathrm{easy}}(\mathfrak{A})\le\Theta_{\mathrm{easy}}(\mathfrak{A})[p]\}$.
By \criref{disjoint-independent}, the first term on the right-hand
side of \eqref{lowerbound-tail-start} is bounded above by
\[
C_{\mathrm{p}}Ld_{\mathrm{p}}^{2N}\mathbf{P}\bigg[\frac{1}{N}\sum_{i=1}^{N}X_{i}\ge\frac{1}{2}\bigg],
\]
which is bounded above by $C_{\mathrm{p}}L\left(2d_{\mathrm{p}}^{2}\sqrt{p}\right)^{N}\:$
according to \lemref{hugedeviations}. This completes the proof.
\end{proof}
\begin{cor}
\label{corr:quantile-easy-relation}Fix a scale $S=2^{s}$ and let
$K,L\in\mathbf{N}$. Let $\mathfrak{R}=[0,KS)\times[0,LS)$ and $\mathfrak{A}=[0,S)\times[0,2S)$.
Then we have
\[
\Theta_{\mathrm{LR}}(\mathfrak{R})\left[C_{\mathrm{p}}L\left(2d_{\mathrm{p}}^{2}\sqrt{p}\right)^{K/3}+o_{K,L,p}(1)\right]\ge\frac{K}{12u_{0}}\Theta_{\mathrm{easy}}(\mathfrak{A})[p]
\]
and
\[
\mathbf{E}\Psi_{\mathrm{LR}}(\mathfrak{R})\ge\frac{K}{12u_{0}}\Theta_{\mathrm{easy}}(\mathfrak{A})[p]\cdot\left(1-C_{\mathrm{p}}L\left(2d_{\mathrm{p}}^{2}\sqrt{p}\right)^{K/2u_{0}}-o_{K,L,p}(1)\right).
\]
\end{cor}

\begin{proof}
If $\pi$ is a path from left to right in $\mathfrak{R}$, then by
\lemref{leftrightpasses}, we have $|\mathcal{P}(\pi)|\ge K/6$. \propref{lowerbound-tail}
then implies the first equation. The second equation follows immediately
from the first.
\end{proof}
We conclude this section with an inductive version of \corrref{quantile-easy-relation},
showing that some easy crossing quantile grows like $S^{\Omega(1)}$
in the scale $S$.
\begin{prop}
\label{prop:exponential-crossings}There are constants $p_{\mathrm{pl}},q_{\mathrm{pl}},a_{\mathrm{pl}}\in(0,1)$
and a constant $C_{\mathrm{pl}}>0$ so that, if $p<p_{\mathrm{pl}}$
and $\gamma$ is sufficiently small compared to $p$, then, putting
$\mathfrak{R}_{t}=[0,2^{t})\times[0,2^{t+1})$, for any $s>t$ we
have
\[
\Theta_{\mathrm{easy}}(\mathfrak{R}_{t})[p_{\mathrm{pl}}]\le C_{\mathrm{pl}}a_{\mathrm{pl}}^{s-t}\Theta_{\mathrm{easy}}(\mathfrak{R}_{s})[q_{\mathrm{pl}}].
\]
(The subscript $\mathrm{pl}$ stands for ``power-law.'')
\end{prop}

\begin{proof}
Fix a large constant $K$, to be chosen later. Write $s=t+nk+r$,
where $0\le r<k=\log_{2}K$. Let $R=2^{r}$. We can calculate, using
\corrref{quantile-easy-relation},

\begin{multline*}
\Theta_{\mathrm{easy}}(\mathfrak{R}_{t})[p]\le\frac{12u_{0}}{K}\Theta_{\mathrm{easy}}(\mathfrak{R}_{t+k})\left[2C_{\mathrm{p}}K\left(d_{\mathrm{p}}^{2}\sqrt{p}\right)^{K/3}+o_{K}(1)\right]\\
\le\frac{12u_{0}}{K}\Theta_{\mathrm{easy}}(\mathfrak{R}_{t+k})[p],
\end{multline*}
where in the second inequality we use the assumption that $p$ is
sufficiently small, $K$ is sufficiently large, and $\gamma$ is sufficiently
small (compared to $p$ and $K$). By induction we obtain
\begin{align*}
\Theta_{\mathrm{easy}}(\mathfrak{R}_{t})[p] & \le\left(\frac{12u_{0}}{K}\right)^{n}\Theta_{\mathrm{easy}}(\mathfrak{R}_{t+nk})[p]\\
 & =\left(\frac{(12u_{0})^{1/k}}{2}\right)^{kn}\Theta_{\mathrm{easy}}(\mathfrak{R}_{t+nk})[p].
\end{align*}
Thus, applying \corrref{quantile-easy-relation} once more, we get
\begin{multline*}
\Theta_{\mathrm{easy}}(\mathfrak{R}_{t+nk})[p]\le\frac{12u_{0}}{R}\Theta_{\mathrm{easy}}(\mathfrak{R}_{t+nk+r})\left[4R\left(d_{\mathrm{p}}^{2}\sqrt{p}\right)^{R/3}+o_{K}(1)\right]\\
\le\frac{12u_{0}}{R}\Theta_{\mathrm{easy}}(\mathfrak{R}_{t+nk+r})[q_{\mathrm{pl}}],
\end{multline*}
where $p$, $K$, $\gamma$ are restricted so that $q_{\mathrm{pl}}$
can be chosen to be less than $1$. Thus we get the desired inequality
with $a_{\mathrm{pl}}=(12u_{0})^{1/k}/2\in(0,1)$ as long as $K$
is sufficiently large.
\end{proof}

\section{RSW result\label{sec:rsw}}

We will prove the following RSW result relating easy crossings to
hard crossings of $2\times1$ rectangles.
\begin{thm}
\label{thm:rsw}There are constants $\delta_{\mathrm{RSW}}>0$, $C_{\mathrm{RSW}}<\infty$,
$p_{\mathrm{RSW}}>0$ so that 
\begin{equation}
p_{\mathrm{RSW}}\le1/(32\cdot d_{\mathrm{p}}^{2})^{2}\label{eq:prswsmall}
\end{equation}
and, if $\gamma$ is sufficiently small then the following holds.
Let $\mathfrak{R}=[0,S)\times[0,2S)$. Suppose that, for all $\mathfrak{A}\subseteq\mathfrak{R}$
of aspect ratio between $1/2$ and $2$ inclusive, we have
\begin{equation}
\CV^{2}(\Psi_{\mathrm{easy}}(\mathfrak{A}))<\delta_{\mathrm{RSW}}.\label{eq:cvbound-1}
\end{equation}
Then
\[
\Theta_{\mathrm{hard}}(\mathfrak{R})[p_{\mathrm{RSW}}]\le C_{\mathrm{RSW}}\Theta_{\mathrm{easy}}(\mathfrak{R})[p_{\mathrm{RSW}}].
\]
\end{thm}

Our argument is based on the beautiful proof of the RSW result established
for Voronoi percolation in~\cite{tassion}. While our proof has the
same structure and uses many of the same geometric constructions,
two factors make our setting substantially more complicated than the
Voronoi percolation case:
\begin{enumerate}
\item We need to take the weights of crossings into account.
\item We do not have as strong a duality theory in the first-passage percolation
setting, so rather than comparing crossings for a square and a rectangle,
we compare crossings for the easy and hard directions of rectangles.
\end{enumerate}
While our argument is self-contained, a reader first equipped with
a thorough understanding of the arguments in~\cite{tassion} will
find it much easier to follow.

\subsection{Scale and aspect ratio setup}

Fix $p_{0}\in(0,p_{\mathrm{pl}})$, with $p_{\mathrm{pl}}$ as in
\propref{exponential-crossings}.

We will work with rectangles in the portrait orientation with aspect
ratio\textbf{ $\eta=1+2^{-t_{0}}$}, where $t_{0}$ is \emph{fixed}
but will be chosen later. It will be convenient to work at a series
of fixed scales where there are no rounding problems, so for $i\in\mathbf{N}$,
let $u_{i}=[i/2]$, $U_{i}=2^{u_{i}}$, and
\begin{equation}
T_{i}=2^{t_{0}+8}(3/2)^{2\{i/2\}}U_{i}=256\cdot(3/2)^{2\{i/2\}}\cdot2^{t_{0}+[i/2]},\label{eq:Ti-def}
\end{equation}
where $[x]$ is the integer part of $x$ and $x=[x]+\{x\}$. (By way
of illustration, we note that the first few terms of the sequence
$T_{0},T_{1},T_{2},\ldots$ are $2^{t_{0}+7}$ times $2,3,4,6,8,12,\ldots$.
The factor of $256$ is to accommodate occasions when we need to split
up the boxes in certain constructions and is not important to the
argument.) In particular, $T_{i+1}\in\left[4T_{i}/3,3T_{i}/2\right]$
for each $i$ and $\eta T_{i}\in\mathbf{N}$ for each $i$, and if
$j\ge i$, we have the simple estimates
\begin{equation}
\sqrt{2}^{j-i-1}T_{i}\le T_{j}\le\sqrt{2}^{j-i+1}T_{i}.\label{eq:Tij-estimates}
\end{equation}

Let $S_{i}=2^{s_{i}}=2^{t_{0}+9+\lceil i/2\rceil}$ be the least dyadic
integer greater than or equal to $T_{i}$. Let 
\begin{align*}
\mathfrak{R}_{i} & =[0,S_{i})\times[0,2S_{i}), & \mathfrak{R}_{i}^{(\eta)} & =[0,T_{i})\times[0,\eta T_{i}),
\end{align*}
and put 
\begin{align*}
w_{i}^{(\eta)} & =\Theta_{\text{easy}}(\mathfrak{R}_{i}^{(\eta)})[p_{0}], & W_{i}^{(\eta)} & =\max\limits _{j\le i}w_{i}^{(\eta)}.
\end{align*}
It will be convenient to put $w_{i}^{(\eta)}=W_{i}^{(\eta)}=0$ when
$i<0$.

In this section we will need notation for new types of crossing distances.
Let $\mathfrak{B}$ be a box with bottom-left corner $(x_{0},y_{0})$
and top-right corner $(x_{1},y_{1})$. If $I_{1},I_{2}\subset\mathbf{Z}$,
then let $\Psi_{I_{1},I_{2}}(\mathfrak{B};Y)$ denote the minimum
$Y$-weight of a crossing $\pi$in $\mathfrak{B}$ connecting $\{x_{0}\}\times(y_{0}+I_{1})$
and $\{x_{1}\}\times(y_{0}+I_{2})$. Let 
\begin{align*}
\Psi_{\mathrm{L},a,b} & =\Psi_{\{x_{0}\}\times[0,y_{1}-y_{0}),\{x_{1}\}\times(y_{0}+[a,b))}\\
\Psi_{a,b,\mathrm{R}} & =\Psi_{\{x_{0}\}\times(y_{0}+[a,b)),\{x_{1}\}\times[0,y_{1}-y_{0})}.
\end{align*}
Finally, define $\Psi_{\text{X};a}(\mathfrak{B};Y)$ to be the minimum
$Y$-weight of a crossing in $\mathfrak{B}$ that connects the four
segments $\{x_{0}\}\times[y_{0},\frac{y_{0}+y_{1}}{2}-a)$, $\{x_{1}\}\times[y_{0},\frac{y_{0}+y_{1}}{2}-a)$,
$\{x_{0}\}\times[\frac{y_{0}+y_{1}}{2}+a,y_{1})$, and $\{x_{1}\}\times[\frac{y_{0}+y_{1}}{2}+a,y_{1})$
(thus forming an ``X'' shape with each arm terminating at least
a distance $a$ from the horizontal midline of the box). We moreover
extend the $\pi$ and $\Theta$ notation accordingly as in \subsecref{notation}.
This notation is concordant with the $\mathcal{X}$ and $\mathcal{H}$
notation in~\cite{tassion}.

We aim to prove \thmref{rsw}, which is about portrait $1\times2$
rectangles; however, we will argue using rectangles which are portrait
but very close to square. In order to conclude, we will need to relate
the $w_{i}^{(\eta)}$s and the crossing quantiles for $1\times2$
rectangles. We do this in the following lemma and corollary.
\begin{lem}
\label{lem:squish}Let $S$ and $b<k$ be natural numbers and put
$\mathfrak{A}=[0,bS)\times[0,(b+1)S)$ and $\mathfrak{B}=[0,bS)\times[0,kS)$.
Then
\[
\mathbf{P}[\Psi_{\mathrm{easy}}(\mathfrak{B})\le y]\le2k\mathbf{P}[\Psi_{\mathrm{easy}}(\mathfrak{A})\le y]+o_{k}(1).
\]
\end{lem}

\begin{proof}
Divide the rectangle $[0,bS)\times[0,kS)$ into $k$ $bS\times S$
landscape subrectangles. Now a left–right crossing of $[0,bS)\times[0,kS)$
must either cross horizontally within a block of $b+1$ of these subrectangles
($k-b$ such blocks) or else cross vertically a block of $b$ of these
subrectangles ($k-b+1$ such blocks). Each of these events has probability
at most
\[
\mathbf{P}[\Psi_{\mathrm{LR}}(\mathfrak{A})\le y]+o_{|\mathfrak{B}/\mathfrak{A}|}(1)
\]
(using \eqref{relatescale}) so the conclusion of the lemma follows
from a union bound.
\end{proof}
\begin{cor}
\label{corr:squishdown-1}For any fixed $\eta>1$ the following holds.
There are constants $C_{\mathrm{str}}(\eta)<\infty$ and $p_{\mathrm{str}}(\eta)\in(0,1)$
so that, if $\gamma$ is sufficiently small, then $w_{i}^{(\eta)}\le C_{\mathrm{str}}(\eta)\cdot\Theta_{\mathrm{easy}}(\mathfrak{R}_{i})[p_{\mathrm{str}}(\eta)]$.
\end{cor}

\subsection{Gluing}

We now begin in earnest the proof of our RSW result.
\begin{lem}
\label{lem:setuplambda}There is a $p_{1}>0$, depending only on $p_{0}$,
so that the following holds. Let $y\ge w_{i}^{(\eta)}$ and let 
\begin{align}
f_{y}(\alpha,\beta) & =\mathbf{P}[\Psi_{\text{L},\alpha,\beta}(\mathfrak{R}_{i}^{(\eta)})\le y],\nonumber \\
g_{w,y}(\alpha) & =f_{w}(0,\alpha)-f_{y}(\alpha,\eta T_{i}/2),\text{ and}\nonumber \\
\lambda=\lambda_{i}^{y} & =\frac{\eta T_{i}}{8}\wedge\min\{\alpha\in\{1,\ldots,\eta T_{i}\}\mid g_{w_{i}^{(\eta)},y}(\alpha)\ge p_{0}/4\}.\label{eq:lambdadef}
\end{align}
Then $\lambda$ is a well-defined element of $[0,\eta T_{i}/8]$ and
the following two statements both hold:
\end{lem}

\begin{enumerate}
\item \label{enu:option}Either
\begin{enumerate}[ref=\theenumi (\alph*)]
\item \label{enu:very-close}$\mathbf{P}[\Psi_{\mathrm{LR}}([0,2T_{i})\times[0,\eta T_{i}))\le3y]\ge p_{1}$,
or
\item \label{enu:above-expensive}$\mathbf{P}[\Psi_{\mathrm{L},\lambda,\eta T_{i}}(\mathfrak{R}_{i}^{(\eta)})\le y]\ge p_{1}.$
\end{enumerate}
\item \label{enu:difference}If $\lambda<\eta T_{i}/8$, then
\[
\mathbf{P}[\Psi_{\mathrm{L},0,\lambda}(\mathfrak{R}_{i}^{(\eta)})\le w_{i}^{(\eta)}]\ge\frac{p_{0}}{4}+\mathbf{P}[\Psi_{\mathrm{L},\lambda,\eta T_{i}}(\mathfrak{R}_{i}^{(\eta)})\le y].
\]
\end{enumerate}
\begin{rem}
Note that $f_{y}(\alpha,\beta)$ is increasing in $y$, so $g_{w,y}(\alpha)$
is decreasing in $y$ and thus $\lambda_{i}^{y}$ is increasing in
$y$. Moreover, for each $i$, there is a $y_{i}^{*}$ so that
\begin{equation}
\lambda_{i}^{y_{i}^{*}}=\eta T_{i}/8.\label{eq:blowupthestartweight}
\end{equation}
\end{rem}

\begin{proof}
First note that $g_{w_{i}^{(\eta)},y}$ is increasing, we have $g_{w_{i}^{(\eta)},y}(0)<0$,
and $g_{w_{i}^{(\eta)},y}(\eta T_{i})=f_{w_{i}^{(\eta)}}(\eta T_{i})\ge p_{0}/2$.
Thus $\lambda$ is well-defined by the definition in the statement
of the theorem. Note that symmetry implies that, for any $\alpha\in(0,\ldots,\eta T_{i}/2)$,
\begin{multline*}
p_{0}/2\le f_{w_{i}^{(\eta)}}(0,\eta T_{i}/2)\le f_{w_{i}^{(\eta)}}(0,\alpha)+f_{w_{i}^{(\eta)}}(\alpha,\eta T_{i}/2)\\
\le f_{w_{i}^{(\eta)}}(0,\alpha)+f_{y}(\alpha,\eta T_{i}/2),
\end{multline*}
so (using \eqref{lambdadef}) 
\[
f_{w_{i}^{(\eta)}}(0,\lambda)\ge p_{0}/4+f_{y}(\lambda,\eta T_{i}/2)\ge p_{0}/4
\]
whenever $\lambda<\eta T_{i}/8$, and
\begin{equation}
f_{w_{i}^{(\eta)}}(0,\lambda-1)-f_{y}(\lambda-1,\eta T_{i}/2)<p_{0}/4.\label{eq:nextdown}
\end{equation}
In particular, statement~\ref{enu:difference} holds.

The proof of statement~\ref{enu:option} comes down to two cases,
depending on the value of $g_{w_{i}^{(\eta)},y}(\lambda)$.

\emph{Case 1. }If $g_{w_{i}^{(\eta)},y}(\lambda)\ge3p_{0}/8,$ then
this along with~\eqref{nextdown} implies that
\begin{multline*}
p_{0}/8<f_{w_{i}^{(\eta)}}(0,\lambda)-f_{y}(\lambda,\eta T_{i})-[f_{w_{i}^{(\eta)}}(0,\lambda-1)-f_{y}(\lambda-1,\eta T_{i}/2)]\\
\le f_{y}(\lambda-1,\eta T_{i}/2)-f_{y}(\lambda,\eta T_{i})\le\mathbf{P}[\Psi_{\mathrm{L},\lambda-1,\lambda}(\mathfrak{R}_{i}^{(\eta)})\le y].
\end{multline*}
In words, this means that the probability of a crossing of weight
at most $y$ from the left side of $\mathfrak{R}_{i}^{(\eta)}$ to
the \emph{point} with coordinates $(T_{i},\eta T_{i}/2+\lambda-1)$
is at least $p_{0}/8$. But then (using the \nameref{lem:FKG} and
\eqref{relatescale})
\begin{multline*}
\mathbf{P}[\Psi_{\mathrm{LR}}([0,2T_{i})\times[0,\eta T_{i}))\le3y]\\
\ge\mathbf{P}[\Psi_{\mathrm{L},\lambda-1,\lambda}(\mathfrak{R}_{i}^{(\eta)})\le y]\mathbf{P}[\Psi_{\lambda-1,\lambda,\mathrm{R}}(\mathfrak{R}_{i}^{(\eta)})\le y]-o(1)>(\tfrac{p_{0}}{8})^{2}-o(1),
\end{multline*}
so as long as $p_{1}\le(p_{0}/8)^{2}$, then statement~\ref{enu:very-close}
holds.

\emph{Case 2. }Now suppose $g_{w_{i}^{(\eta)},y}(\lambda)<3p_{0}/8$.
This means that we have
\begin{multline*}
p_{0}/2\le f_{w_{i}^{(\eta)}}(0,\lambda)+f_{w_{i}^{(\eta)}}(\lambda,\eta T_{i}/2)\le f_{w_{i}^{(\eta)}}(0,\lambda)+f_{y}(\lambda,\eta T_{i}/2)\\
\le g_{w_{i}^{(\eta)},y}(\lambda)+2f_{y}(\lambda,\eta T_{i}/2)\le3p_{0}/8+2f_{y}(\lambda,\eta T_{i}/2),
\end{multline*}
so $f_{y}(\lambda,\eta T_{i}/2)\ge p_{0}/16$. So as long as $p_{1}<p_{0}/16$,
statement~\ref{enu:above-expensive} holds.
\end{proof}
\begin{lem}
\label{lem:getthex}If statement~\enuref{above-expensive} of Lemma~\ref{lem:setuplambda}
holds, and $\gamma$ is sufficiently small, then there is a $p_{2}>0$,
depending only on $p_{1}$, so that the following holds. Let $y\ge w_{i}^{(\eta)}$.
Suppose that 
\begin{equation}
\eta-\frac{\lambda_{i}^{y}}{32T_{i}}<1.\label{eq:etasmall1}
\end{equation}
Then if 
\begin{align}
\mu & =\mu_{i}^{y}\in(\tfrac{1}{16}\lambda_{i}^{y},\tfrac{1}{8}\lambda_{i}^{y}),\text{ and}\label{eq:mu-def}\\
\nu & =\nu_{i}^{y}=2\lambda_{i}^{y}-\mu_{i}^{y},\label{eq:nu-def}
\end{align}
then
\[
\mathbf{P}[\Psi_{\mathrm{X};(\nu-\mu)/2}(\mathfrak{R}_{i}^{(\eta)})\le9y]\ge p_{2}.
\]
\end{lem}

\begin{proof}
Note that 
\begin{multline}
\{\Psi_{[\eta T_{i}/2+\nu/2,\eta T_{i}),[\eta T_{i}/2+\nu/2,\eta T_{i})}(\mathfrak{R}_{i}^{(\eta)})\le2y\}\\
\supset\{\Psi_{\mathrm{L},\nu/2,\eta T_{i}/2}(\mathfrak{R}_{i}^{(\eta)})\le y\}\cap\{\Psi_{\nu/2,\eta T_{i}/2,\mathrm{R}}(\mathfrak{R}_{i}^{(\eta)})\le y\}.\label{eq:crosstop}
\end{multline}
By combining \eqref{crosstop}, the \nameref{lem:FKG}, \eqref{relatescale},
and statement~\ref{enu:above-expensive} of \lemref{setuplambda},
we have
\begin{align}
\mathbf{P}\mathrlap{\left[\Psi_{[\eta T_{i}/2+\nu/2,\eta T_{i}],[\eta T_{i}/2+\nu/2,\eta T_{i}]}(\mathfrak{R}_{i}^{(\eta)})\le3y\right]}\label{eq:crosstopprop}\\
 & \ge\mathbf{P}[\Psi_{\mathrm{L},\nu/2,\eta T_{i}/2}(\mathfrak{R}_{i}^{(\eta)})\le y]^{2}-o(1)\nonumber \\
 & \ge\mathbf{P}[\Psi_{\mathrm{L},\lambda,\eta T_{i}/2}(\mathfrak{R}_{i}^{(\eta)})\le y]^{2}-o(1)\nonumber \\
 & \ge p_{1}^{2}-o(1).\nonumber 
\end{align}
Let $\tilde{\mathfrak{R}}_{i}^{(\eta)}=[0,T_{i})\times[\mu,\eta T_{i})$.
By \eqref{etasmall1} and \eqref{mu-def}, $\tilde{\mathfrak{R}}_{i}^{(\eta)}$
is landscape. Let $E$ be the event that there is a left–right path
in $\mathfrak{R}_{i}^{(\eta)}$ connecting the intervals $[\eta T_{i}/2+\nu/2,\eta T_{i})$
on each side, of weight at most $3y$, that enters the box $\mathfrak{R}_{i}^{(\eta)}\setminus\tilde{\mathfrak{R}}_{i}^{(\eta)}$.
Then we have that
\begin{align*}
\mathbf{P}[\Psi_{\mathrm{L},\nu/2,\eta T_{i}/2}\mathrlap{(\tilde{\mathfrak{R}}_{i}^{(\eta)})\le3y]+\mathbf{P}[E]}\\
 & \ge\mathbf{P}[\Psi_{[\eta T_{i}/2+\nu/2,\eta T_{i}],[\eta T_{i}/2+\nu/2,\eta T_{i}]}(\mathfrak{R}_{i}^{(\eta)})\le3y]\\
 & \ge p_{1}^{2}-o(1),
\end{align*}
where for the second inequality we use \eqref{crosstopprop}. Thus,
either 
\begin{equation}
\mathbf{P}[\Psi_{\mathrm{L},\nu/2,\eta T_{i}/2}(\tilde{\mathfrak{R}}_{i}^{(\eta)})\le3y]\ge p_{1}^{2}/2-o(1)\label{eq:case1}
\end{equation}
or 
\begin{equation}
\mathbf{P}[E]\ge p_{1}^{2}/2-o(1).\label{eq:case2}
\end{equation}
We consider each case in turn.

\emph{Case 1. }First suppose that \eqref{case2} holds. Let $E_{1}=E$
and define $E_{2}$ to be a copy of $E$ which is vertically flipped
and translated upwards by $\mu$. Then the intersection of $E_{1}$
and $E_{2}$ is contained in $\{\Psi_{\mathrm{X};(\nu-\mu)/2}(\mathfrak{R}_{i}^{(\eta)})\le6y\}$.
This is because the path in $E_{2}$ must cross the path from $E_{1}$
once on its way from $0\times[\mu,\mu+\eta T_{i}/2-\nu/2)$ to $(\mathfrak{R}_{i}^{(\eta)}+(0,\mu))\setminus\mathfrak{R}_{i}^{(\eta)}$,
and another time on its way from $(\mathfrak{R}_{i}^{(\eta)}+(0,\mu))\setminus\mathfrak{R}_{i}^{(\eta)}$
to $\{T_{i}-1\}\times[\mu,\mu+\eta T_{i}/2-\nu/2)$. (See \figref{Ebig}.)
Thus we have
\[
\mathbf{P}[\Psi_{\mathrm{X};(\nu-\mu)/2}(\mathfrak{R}_{i}^{(\eta)})\le6y]\ge(p_{1}^{2}/2)^{2}-o(1)
\]
by the \nameref{lem:FKG}. This proves the lemma in this case, as
long as $p_{2}\le p_{1}^{4}/4-o(1)$.
\begin{figure}[t]
\begin{centering}
\tiny
\begin{tikzpicture}[x=0.75in,y=0.75in]
\dimline[extension start length = 0.55*0.25in, extension end length = 0.55*0.25in]{(3,-0.25)}{(0,-0.25)}{$T_i$}
\dimline[extension start length = 0.55*0.5in, extension end length = 0.55*0.5in]{(-0.5,0)}{(-0.5,4)}{$\eta T_i$}
\dimline[extension start length = 0.55*0.25in, extension end length = 0.55*0.25in]{(-0.25,2)}{(-0.25,3.1)}{$\lambda$}
\dimline[extension start length = 0.55*0.25in, extension end length = 0.55*0.25in]{(-0.25,0)}{(-0.25,2)}{$\tfrac{\eta T_i}{2}$}
\dimline[extension start length = 0.55*0.25in, extension end length = 0.55*0.25in]{(3.25,1)}{(3.25,0)}{$\mu$}
\dimline[extension start length = 0.55*0.25in, extension end length = 0.55*0.25in]{(3.25,3.1)}{(3.25,1.9)}{$\nu$}
\dimline[extension start length = 0.55*0.5in, extension end length = 0.55*0.5in,label style={above=1pt},line style={arrows = dimline reverse-dimline reverse}]{(3.5,2)}{(3.5,1.9)}{$\tfrac{\nu-\mu}{2}$}
\draw [red] (0,0) -- (3,0) -- (3,4) -- (0,4) -- cycle;
\draw [blue] (0,1) -- (3,1) -- (3,5) -- (0,5) -- cycle;
\draw [blue, ultra thick] (0,1) -- (0,1.9);
\draw [blue, ultra thick] (3,1) -- (3,1.9);
\draw [red, ultra thick] (0,3.1) -- (0,4);
\draw [red, ultra thick] (3,3.1) -- (3,4);
\draw [red] (-0.125,2) -- (0.125,2);
\draw [blue] (2.875,3) -- (3.125,3);
\draw [red] (2.875,2) -- (3.125,2);
\draw [thin,blue] plot [smooth] coordinates { (0,1.25)  (1,4.38)  (1.7,4.69)  (3,1.09)};
\draw [thin,red] plot [smooth] coordinates { (0,3.8) (0.7,0.72) (3,3.3) };
\draw [opacity=0] (-1.5,0)--(4.5,0);
\end{tikzpicture}
\par\end{centering}
\caption{Setting up for the \nameref{lem:FKG} when $\mathbf{P}[E]\ge p_{1}^{2}/2$.
Combining the red crossing with the lower pieces of the blue crossing
gives an ``X'' shape inside the red $T_{i}\times\eta T_{i}$ box
with endpoints at least distance $(\nu-\mu)/2$ from the midline.
In this and future figures in this section, the origin $(0,0)$ is
at the bottom left.\label{fig:Ebig}}
\end{figure}
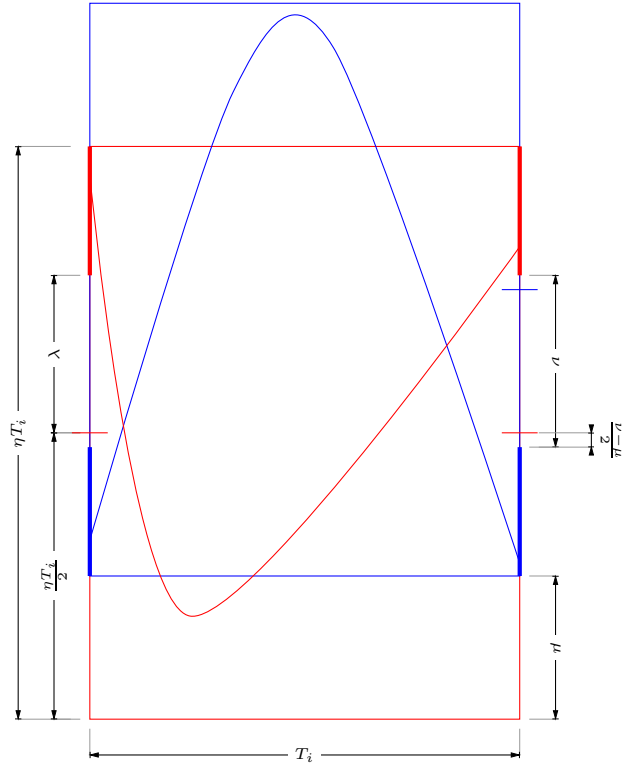
\begin{figure}[t]
\begin{centering}
\tiny
\begin{tikzpicture}[x=0.45in,y=0.45in]
\dimline[extension start length = 0.55*0.25in, extension end length = 0.55*0.25in]{(3,-0.25)}{(0,-0.25)}{$T_i$}
\dimline[extension start length = 0.55*0.25in, extension end length = 0.55*0.25in]{(3.25,2.7)}{(3.25,2)}{$\frac{\nu}{2}$}
\dimline[extension start length = 0.55*0.25in, extension end length = 0.55*0.25in]{(3.25,2)}{(3.25,1.3)}{$\frac{\nu}{2}$}
\dimline[extension start length = 0.55*0.25in, extension end length = 0.55*0.25in]{(-0.25,0)}{(-0.25,2)}{$\tfrac{\eta T_i}{2}$}
\dimline[extension start length = 0.55*0.25in, extension end length = 0.55*0.25in]{(-0.25,2)}{(-0.25,4)}{$\tfrac{\eta T_i}{2}$}
\draw [black] (0,0) -- (3,0) -- (3,4) -- (0,4) -- cycle;
\draw [black, ultra thick] (0,2.7) -- (0,4);
\draw [black, ultra thick] (3,2.7) -- (3,4);
\draw [black, ultra thick] (0,0) -- (0,1.3);
\draw [black, ultra thick] (3,0) -- (3,1.3);
\draw [black] (-0.125,2) -- (0.125,2);
\draw [black] (2.875,2) -- (3.125,2);
\draw [thin,red] plot [smooth] coordinates { (0,3.8) (1.4,2.5) (3,2.9)};
\draw [thin,red] plot [smooth] coordinates { (0,0.8) (0.4,1.4) (3,0.1)};
\draw [thin,blue] plot [smooth] coordinates { (0.3,0) (1.2,0.7) (1.8,2.7) (0.3,3.8) (0.5,4) };
\draw [opacity=0] (-1.5,0)--(4.5,0);
\end{tikzpicture}
\par\end{centering}
\caption{Setting up for the \nameref{lem:FKG} when $\mathbf{P}[\Psi_{\mathrm{L},\nu/2,\eta T_{i}/2}(\tilde{\mathfrak{R}}_{i}^{(\eta)})\le y]\ge p_{1}^{2}/2$.
Combining a piece of the blue vertical crossing with the red horizontal
crossings gives an ``X'' shape inside the $T_{i}\times\eta T_{i}$
box with endpoints at least distance $\lambda\ge(\nu-\mu)/2$ from
the midline.\label{fig:Esmall}}
\end{figure}
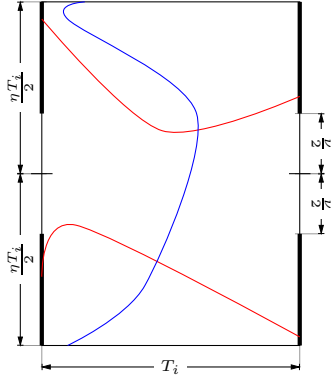

\emph{Case 2. }We are left with the case when \eqref{case1} holds,
which in particular means that
\begin{equation}
\mathbf{P}[\Psi_{\mathrm{LR}}(\tilde{\mathfrak{R}}_{i}^{(\eta)})\le3y]\ge p_{1}^{2}/2-o(1).\label{eq:crosshard1}
\end{equation}
Observe that the event $\{\Psi_{\mathrm{X};\nu/2}(\mathfrak{R}_{i}^{(\eta)})\le9y\}$
contains the intersection of the three events
\begin{align*}
\{\Psi_{[\eta T_{i}/2+\nu/2,\eta T_{i}),[\eta T_{i}/2+\nu/2,\eta T_{i})}(\mathfrak{R}_{i}^{(\eta)}) & \le3y\},\\
\{\Psi_{[0,\eta T_{i}/2-\nu/2),[0,\eta T_{i}/2-\nu/2)}(\mathfrak{R}_{i}^{(\eta)}) & \le3y\},\text{ and}\\
\{\Psi_{\mathrm{BT}}(\mathfrak{R}_{i}^{(\eta)}) & \le3y\}
\end{align*}
(see \figref{Esmall}), so
\begin{align}
\mathbf{P}[\mathrlap{\Psi_{\mathrm{X};\nu/2}(\mathfrak{R}_{i}^{(\eta)})\le9y]}\label{eq:X1}\\
 & \ge\mathbf{P}[\Psi_{[\eta T_{i}/2+\nu/2,\eta T_{i}),[\eta T_{i}/2+\nu/2,\eta T_{i})}(\mathfrak{R}_{i}^{(\eta)})\le3y]^{2}\mathbf{P}[\Psi_{\mathrm{BT}}(\mathfrak{R}_{i}^{(\eta)})\le3y]\nonumber \\
 & \ge p_{1}^{4}\mathbf{P}[\Psi_{\mathrm{BT}}(\mathfrak{R}_{i}^{(\eta)})\le3y]\nonumber 
\end{align}
by symmetry, \eqref{crosstopprop}, and the \nameref{lem:FKG}. Now
by \eqref{etasmall1} and the definition of $\mu$, we have $\mu/2\ge\lambda_{i}^{y}/32>(\eta-1)T_{i},$
so $2T_{i}-\eta T_{i}+\mu>\eta T_{i}.$ Hence, by \corrref{stretch}
applied with $k=2$ (recalling that $\tilde{\mathfrak{R}}_{i}^{(\eta)}$
is landscape), we have
\[
\mathbf{P}[\Psi_{\mathrm{BT}}(\mathfrak{R}_{i}^{(\eta)})\le3y]\ge\mathbf{P}[\Psi_{\mathrm{LR}}(\tilde{\mathfrak{R}}_{i}^{(\eta)})\le y]^{3}-o(1)\ge p_{1}^{6}/8-o(1),
\]
where the second inequality is by \eqref{crosshard1}. Combining this
last inequality with~\eqref{X1}, we obtain
\[
\mathbf{P}[\Psi_{\mathrm{X};(\nu-\mu)/2}(\tilde{\mathfrak{R}}_{i}^{(\eta)})\le9y]\ge\mathbf{P}[\Psi_{\mathrm{X};\nu/2}(\tilde{\mathfrak{R}}_{i}^{(\eta)})\le9y]\ge p_{1}^{10}/8-o(1),
\]
completing the proof of the lemma in this case, as long as $p_{2}\le p_{1}^{10}/8-o(1)$\@.
\end{proof}
\begin{lem}
\label{lem:usethex}There is a $p_{3}$, depending only on $p_{1}$
and $p_{2}$, so that if $\gamma$ is sufficiently small compared
to $p_{1}$ and $p_{2}$ then the following statement holds. Suppose
that $y\ge w_{i}^{(\eta)}$, $\eta<\frac{256}{255}$, and $z\ge0$
are such that either
\begin{enumerate}
\item \label{enu:nontrivial}$z\ge w_{i-1}^{(\eta)}$ and $\lambda_{i}^{y}\le\frac{7}{4}\lambda_{i-1}^{z}$
and $\eta-\frac{\lambda_{i-1}^{z}}{32T_{i-1}}<1$ (in which the third
inequality says that \eqref{etasmall1} holds at scale $i-1$ with
weight $z$), or
\item $\lambda_{i}^{y}=\eta T_{i}/8.$
\end{enumerate}
Then
\[
\mathbf{P}[\Psi_{\mathrm{LR}}([0,5T_{i}/4)\times[0,\eta T_{i}))\le55y+11z]\ge p_{3}.
\]
\end{lem}

\begin{proof}
If
\[
\mathbf{P}[\Psi_{\mathrm{LR}}([0,2T_{i})\times[0,\eta T_{i}))\le3y]\ge p_{1},
\]
(i.e. if statement~\enuref{very-close} from \lemref{setuplambda}
holds) then there is nothing more to show as long as $p_{3}\le p_{1}-o(1)$,
since horizontally crossing a $T_{i}\times\eta T_{i}$ box implies
horizontally crossing a $\tfrac{5}{4}T_{i}\times\eta T_{i}$ box.
Similarly, if 
\[
\mathbf{P}[\Psi_{\mathrm{LR}}([0,2T_{i-1})\times[0,\eta T_{i-1}))\le3z]\ge p_{1},
\]
then we have (using \eqref{Ti-def}) 
\begin{align*}
\mathbf{P}[\Psi_{\mathrm{LR}}\mathrlap{([0,5T_{i}/4)\times[0,\eta T_{i}))\le4z]}\\
 & \ge\mathbf{P}[\Psi_{\mathrm{LR}}([0,2T_{i-1})\times[0,\eta T_{i-1}))\le3z]-o(1)\\
 & \ge p_{1}-o(1),
\end{align*}
so there is nothing more to show as long as long as $p_{3}\le p_{1}-o(1)$.

Thus we may henceforth assume that statement~\enuref{above-expensive}
from \lemref{setuplambda} holds for both $i$ (with weight $y=y$)
and $i-1$ (with weight $y=z$). The rest of the proof is divided
into two cases.
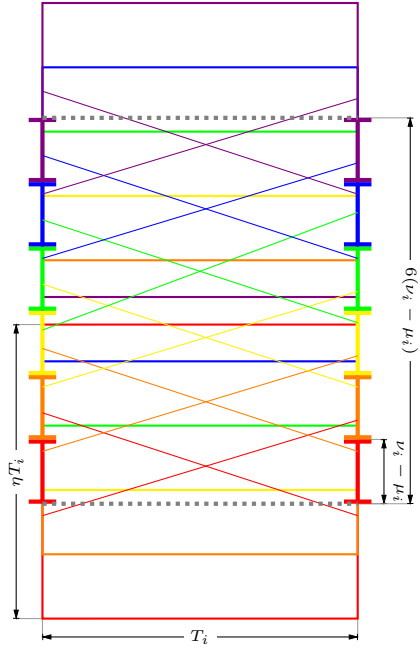
\begin{figure}[t]
\begin{centering}
\tiny
\begin{tikzpicture}[x=0.5*1.1in,y=0.35*1.1in]
\dimline[extension start length=1.1*0.125in,extension end length=1.1*0.125in]{(3,-0.25)}{(0,-0.25)}{$T_i$}
\dimline[extension start length=1.1*0.5*0.25in,extension end length=1.1*0.5*0.25in,label style={above=1pt}]{(3.25,2+0.4375)}{(3.25,2-0.4375)}{$\nu_i-\mu_i$}
\dimline[extension start length=1.1*0.5*0.5in,extension end length=1.1*0.5*0.5in]{(3.5,2+5*0.875+0.4375)}{(3.5,2-0.4375)}{$6(\nu_i-\mu_i)$}
\dimline{(-0.25,0)}{(-0.25,4)}{$\eta T_i$}
\draw [red,thick] (0,0) -- (3,0) -- (3,4) -- (0,4) -- cycle;
\draw [orange,thick] (0,0.875*1) -- (3,0.875*1) -- (3,4+0.875*1) -- (0,4+0.875*1) -- cycle;
\draw [yellow,thick] (0,0.875*2) -- (3,0.875*2) -- (3,4+0.875*2) -- (0,4+0.875*2) -- cycle;
\draw [green,thick] (0,0.875*3) -- (3,0.875*3) -- (3,4+0.875*3) -- (0,4+0.875*3) -- cycle;
\draw [blue,thick] (0,0.875*4) -- (3,0.875*4) -- (3,4+0.875*4) -- (0,4+0.875*4) -- cycle;
\draw [violet,thick] (0,0.875*5) -- (3,0.875*5) -- (3,4+0.875*5) -- (0,4+0.875*5) -- cycle;
\draw [|-|,red, ultra thick] (0,2-0.4375) -- (0,2+0.4375);
\draw [|-|,orange, ultra thick] (0,2+0.875*1-0.4375) -- (0,2+0.875*1+0.4375);
\draw [|-|,yellow, ultra thick] (0,2+0.875*2-0.4375) -- (0,2+0.875*2+0.4375);
\draw [|-|,green, ultra thick] (0,2+0.875*3-0.4375) -- (0,2+0.875*3+0.4375);
\draw [|-|,blue, ultra thick] (0,2+0.875*4-0.4375) -- (0,2+0.875*4+0.4375);
\draw [|-|,violet, ultra thick] (0,2+0.875*5-0.4375) -- (0,2+0.875*5+0.4375);
\draw [|-|,red, ultra thick] (3,2-0.4375) -- (3,2+0.4375);
\draw [|-|,orange, ultra thick] (3,2+0.875*1-0.4375) -- (3,2+0.875*1+0.4375);
\draw [|-|,yellow, ultra thick] (3,2+0.875*2-0.4375) -- (3,2+0.875*2+0.4375);
\draw [|-|,green, ultra thick] (3,2+0.875*3-0.4375) -- (3,2+0.875*3+0.4375);
\draw [|-|,blue, ultra thick] (3,2+0.875*4-0.4375) -- (3,2+0.875*4+0.4375);
\draw [|-|,violet, ultra thick] (3,2+0.875*5-0.4375) -- (3,2+0.875*5+0.4375);
\draw [thin,red] plot [smooth] coordinates { (0,1.4)  (3,2.7) };
\draw [thin,red] plot [smooth] coordinates {(0,2.8)  (3,1.4) };
\draw [thin,orange] plot [smooth] coordinates { (0,1.4+0.875*1)  (3,2.7+0.875*1) };
\draw [thin,orange] plot [smooth] coordinates { (0,2.8+0.875*1) (3,1.4+0.875*1) };
\draw [thin,yellow] plot [smooth] coordinates {(0,1.4+0.875*2) (3,2.7+0.875*2)};
\draw [thin,yellow] plot [smooth] coordinates { (0,2.8+0.875*2)  (3,1.4+0.875*2) };
\draw [thin,green] plot [smooth] coordinates { (0,1.3+0.875*3) (3,2.9+0.875*3) };
\draw [thin,green] plot [smooth] coordinates { (0,2.8+0.875*3)  (3,1.4+0.875*3) };
\draw [thin,blue] plot [smooth] coordinates { (0,1.4+0.875*4) (3,2.7+0.875*4) };
\draw [thin,blue] plot [smooth] coordinates { (0,2.8+0.875*4)  (3,1.4+0.875*4) };
\draw [thin,violet] plot [smooth] coordinates { (0,1.4+0.875*5)  (3,2.7+0.875*5) } ;
\draw [thin,violet] plot [smooth] coordinates { (0,2.8+0.875*5) (3,1.4+0.875*5) };
\draw [gray, dotted, ultra thick] (0,2-0.4375)--(3,2-0.4375);
\draw [gray, dotted, ultra thick] (0,2+0.875*5+0.4375)--(3,2+0.875*5+0.4375);
\end{tikzpicture}
\par\end{centering}
\caption{\label{fig:stackxs}A vertical crossing between the two dotted lines
is obtained by combining the ``X'' shapes, which must cross because
their endpoints must straddle the interval of their color.}
\end{figure}
\begin{figure}[t]
\begin{centering}
\tiny
\begin{tikzpicture}[x=0.8in,y=0.8in]
\dimline{(3,-0.5)}{(0,-0.5)}{$T_i$}
\dimline{(3,-0.25)}{(1,-0.25)}{$T_{i-1}$}
\dimline[extension start length=0.2in,extension end length=0.2in]{(3.25,2.7)}{(3.25,2)}{$\lambda_i$}
\dimline[extension start length=0.2in,extension end length=0.2in]{(-0.25,0)}{(-0.25,2)}{$\frac{\eta T_i}{2}$}
\dimline[extension start length=0.2in,extension end length=0.2in]{(-0.25,2)}{(-0.25,4)}{$\frac{\eta T_i}{2}$}
\dimline[extension start length=0.2in,extension end length=0.2in]{(0.75,2)}{(0.75,3)}{$\nu_{i-1}-\mu_{i-1}$}
\dimline{(0.5,2.5-4/3)}{(0.5,2.5+4/3)}{$\eta T_{i-1}$}
\draw [red,thick] (0,0) -- (3,0) -- (3,4) -- (0,4) -- cycle;
\draw [blue,thick] (1,0) -- (4,0) -- (4,4) -- (1,4) -- cycle;
\draw [green] (1,2.5-4/3) -- (1,2.5+4/3) -- (3,2.5+4/3) -- (3,2.5-4/3) -- cycle; 
\draw [green,thick] (0.875,2.5) -- (1.125,2.5);
\draw [green, thick] (2.875,2.5) -- (3.125,2.5);
\draw [green,ultra thick] (1,2) -- (1,3);
\draw [green,ultra thick] (3,2) -- (3,3);
\draw [red, very thick] (3,2) -- (3,2.7);
\draw [blue, very thick] (1,2) -- (1,2.7);
\draw [gray, dotted] (0,2)--(4,2);
\draw [thin,red] plot [smooth] coordinates { (0,1.9) (1.8,0.6) (3,2.4) };
\draw [thin,blue] plot [smooth] coordinates { (4,0.1) (3.7,3.9) (2.3,1.9) (1.4,3.7) (1,2.1) };
\draw [thin,green] plot [smooth] coordinates { (1,1.4) (1.4,2.8) (1,3.4) };
\draw [thin,green] (1.4,2.8) -- (2,3.6);
\draw [thin,green] plot [smooth] coordinates { (3,1.3) (2,3.6) (3,3.65) };
\end{tikzpicture}
\par\end{centering}
\caption{\label{fig:joinstrat}The red and blue crossings are guaranteed to
be joined by the green ``X'' shape, since the red and blue crossings
must remain within the red and blue boxes and end on the thick red
and blue lines, respectively, while the green ``X'' must have endpoints
off of the thick green lines, which contain the thick red and blue
lines.}
\end{figure}
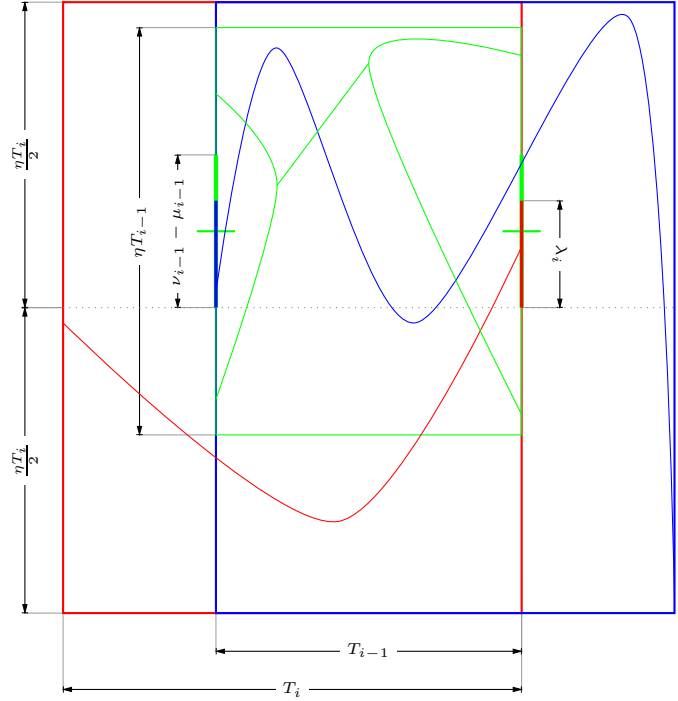

\emph{Case 1. }If $\lambda_{i}^{y}=\eta T_{i}/8$, then $\eta-\frac{1}{32T_{i}}\lambda_{i}^{y}=\frac{255}{256}\eta<1$,
so \lemref{getthex} implies that
\[
\mathbf{P}[\Psi_{\mathrm{X};(\nu_{i}^{y}-\mu_{i}^{y})/2}(\mathfrak{R}_{i}^{(\eta)})\le9y]\ge p_{2}.
\]
Since $\nu_{i}^{y}-\mu_{i}^{y}=2(\lambda_{i}^{y}-\mu_{i}^{y})\ge\frac{7}{4}\lambda_{i}^{y}=\frac{7}{32}\eta T_{i}$
(recalling \eqref{mu-def} and \eqref{nu-def}), the intersection
of six vertically-translated copies of the event 
\[
\{\Psi_{\mathrm{X};(\nu_{i}^{y}-\mu_{i}^{y})/2}(\mathfrak{R}_{i}^{(\eta)})\le9y\}
\]
contains a translate of the event 
\[
\{\Psi_{\mathrm{BT}}([0,T_{i})\times[0,\frac{21}{16}\eta T_{i}))\le54y\}
\]
(as illustrated in \figref{stackxs}), and so also contains a translate
of the event 
\[
\{\Psi_{\mathrm{BT}}([0,T_{i})\times[0,\frac{5}{4}T_{i}))\le54y+7z\}.
\]
So, by the \nameref{lem:FKG} and \eqref{relatescale},
\[
\mathbf{P}[\Psi_{\mathrm{BT}}([0,T_{i})\times[0,\frac{5}{4}T_{i}))\le55y+8z]\ge p_{2}^{6}-o(1).
\]
This completes the proof of the lemma in the case when $\lambda_{i}^{y}=\eta T_{i}/8$,
as long as $p_{3}\le p_{2}^{6}-o(1)$.

\emph{Case 2. }Thus we can assume that $\lambda_{i}^{y}<\frac{\eta T_{i}}{8}$,
so assumption~\enuref{nontrivial} holds, which is to say that $z\ge w_{i-1}^{(\eta)}$
and $\lambda_{i}^{y}\le\frac{7}{4}\lambda_{i-1}^{z}$ and \eqref{etasmall1}
holds at scale $i-1$ with weight $z$. %
Now consider $\mathfrak{R}^{1}=\mathfrak{R}_{i}^{(\eta)}$ and $\mathfrak{R}^{2}=\mathfrak{R}_{i}^{(\eta)}+(T_{i}-T_{i-1},0)$,
and 
\[
\tilde{\mathfrak{R}}=\mathfrak{R}_{i-1}^{(\eta)}+\left(T_{i}-T_{i-1},\frac{1}{2}(\eta T_{i}-\eta T_{i-1}+\nu_{i-1}-\mu_{i-1})\right).
\]
Note that, since
\begin{multline*}
\frac{\eta T_{i}-\eta T_{i-1}+\nu_{i-1}-\mu_{i-1}}{2}+\eta T_{i-1}\le\frac{\eta T_{i}+\eta T_{i-1}}{2}+\lambda_{i-1}\\
\le\frac{\eta}{2}(T_{i}+T_{i-1})+\frac{\eta}{8}T_{i-1}\le\frac{\eta}{2}(T_{i}+\frac{3}{4}T_{i})+\frac{3}{32}\eta T_{i}<\eta T_{i},
\end{multline*}
we have $\tilde{\mathfrak{R}}\subset\mathfrak{R}^{1}\cap\mathfrak{R}_{}^{2}.$
Since $\nu_{i-1}^{z}-\mu_{i-1}^{z}=2(\lambda_{i-1}^{z}-\mu_{i-1}^{z})\ge\frac{7}{4}\lambda_{i-1}^{z}\ge\lambda_{i}^{y}$,
the event
\[
\{\Psi_{\mathrm{X};(\nu_{i-1}^{z}-\mu_{i-1}^{z})/2}(\tilde{\mathfrak{R}})\le9z\}\cap\{\Psi_{\mathrm{L},0,\lambda_{i}^{y}}(\mathfrak{R}^{1})\le y\}\cap\{\Psi_{0,\lambda_{i}^{y},\mathrm{R}}(\mathfrak{R}^{2})\le y\}
\]
is contained in, up to coarse field error (i.e. the error bounded
in \eqref{relatescale}), the event
\[
\{\Psi_{\mathrm{LR}}([0,2T_{i}-T_{i-1})\times[0,\eta T_{i}))\le2y+9z\},
\]
since the crossings in the two larger rectangles must both intersect
the ``X'' shape in the smaller rectangle, as they both must end
on an interval that is contained in an interval that must be straddled
by the endpoints of the ``X''. (See \figref{joinstrat}.) Hence,
by the \nameref{lem:FKG} and \eqref{relatescale}, we have
\begin{align*}
\mathbf{P}\mathrlap{[\Psi_{\mathrm{LR}}([0,2T_{i}-T_{i-1})\times[0,\eta T_{i}))\le3y+10z]}\\
 & \ge\mathbf{P}[\Psi_{\mathrm{X};(\nu_{i-1}^{z}-\mu_{i-1}^{z})/2}(\tilde{\mathfrak{R}})\le9z]\cdot\mathbf{P}[\Psi_{\mathrm{L},0,\lambda_{i}^{y}}(\mathfrak{R}^{1})\le y]^{2}-o(1)\\
 & \ge p_{1}^{2}\mathbf{P}[\Psi_{\mathrm{X};(\nu_{i-1}^{z}-\mu_{i-1}^{z})/2}(\tilde{\mathfrak{R}})\le9z]-o(1).
\end{align*}
Now by \lemref{getthex}, recalling our assumption that \eqref{etasmall1}
holds at scale $i-1$ with weight $z$, if $\gamma$ is sufficiently
small compared to $p_{2}$, we have
\[
\mathbf{P}[\Psi_{\mathrm{X};(\nu_{i-1}^{z}-\mu_{i-1}^{z})/2}(\tilde{\mathfrak{R}})\le9z]\ge p_{2}.
\]
So
\begin{align*}
\mathbf{P}\mathrlap{\left[\Psi_{\mathrm{LR}}([0,5T_{i}/4)\times[0,\eta T_{i}))\le4y+11z\right]}\\
 & \ge\mathbf{P}\left[\Psi_{\mathrm{LR}}([0,2T_{i}-T_{i-1})\times[0,\eta T_{i}))\le3y+10z\right]-o(1)\\
 & \ge p_{1}^{2}p_{2}/2-o(1),
\end{align*}
completing the proof of the lemma in the case when $\lambda_{i}^{y}<\eta T_{i}/8$,
as long as $p_{3}\le\frac{1}{2}p_{1}^{2}p_{2}-o(1)$.
\end{proof}
\begin{lem}
\label{lem:ring}There are constants $c_{1}$ and $p_{4}$, depending
only on $p_{3}$, so that the following holds. Let $j\ge i+8$. Suppose
that $\eta\le9/8$, $\lambda_{j}^{w_{j}^{(\eta)}}\le\eta T_{i}$,
$\gamma$ is sufficiently small compared to $p_{3}$, and
\begin{equation}
\mathbf{P}\left[\Psi_{\mathrm{LR}}([0,5T_{i}/4)\times[0,\eta T_{i}))\le y\right]\ge p_{3}.\label{eq:startcross}
\end{equation}
Then 
\[
\mathbf{P}[\Psi_{\mathrm{LR}}([0,2T_{j})\times[0,\eta T_{j}))\le2w_{j}^{(\eta)}+c_{1}y]\ge p_{4}-o_{j-i}(1).
\]
\end{lem}

\begin{proof}
Since $j\ge i+8$, we have $T_{j}\ge16T_{i}$, so $\lambda_{j}^{y}\le\eta T_{i}<\eta T_{j}/8$,
so by statement~\ref{enu:difference} of \lemref{setuplambda}, we
have
\begin{equation}
\mathbf{P}\left[\Psi_{\mathrm{L},0,\lambda_{j}^{w_{j}^{(\eta)}}}(\mathfrak{R}_{j}^{(\eta)})\le w_{i}^{(\eta)}\right]\ge p_{0}/4\label{eq:cross1}
\end{equation}
and
\begin{equation}
\mathbf{P}\left[\Psi_{0,\lambda_{j}^{w_{j}^{(\eta)}},\mathrm{R}}(\mathfrak{R}_{j}^{(\eta)}+(T_{j},0))\le w_{i}^{(\eta)}\right]\ge p_{0}/4.\label{eq:cross2}
\end{equation}
Now we can build an annulus $\mathfrak{L}$, whose inner square has
width $\eta T_{i}$ and whose outer square has width $3\eta T_{i}$,
inside $\mathfrak{R}_{j}^{(\eta)}\cup(\mathfrak{R}_{j}^{(\eta)}+(T_{j},0))$,
such that $\mathfrak{L}$ surrounds $T_{j}\times\left[\eta T_{j}/2,\eta T_{j}/2+\lambda_{j}^{w_{j}^{(\eta)}}\right]$.
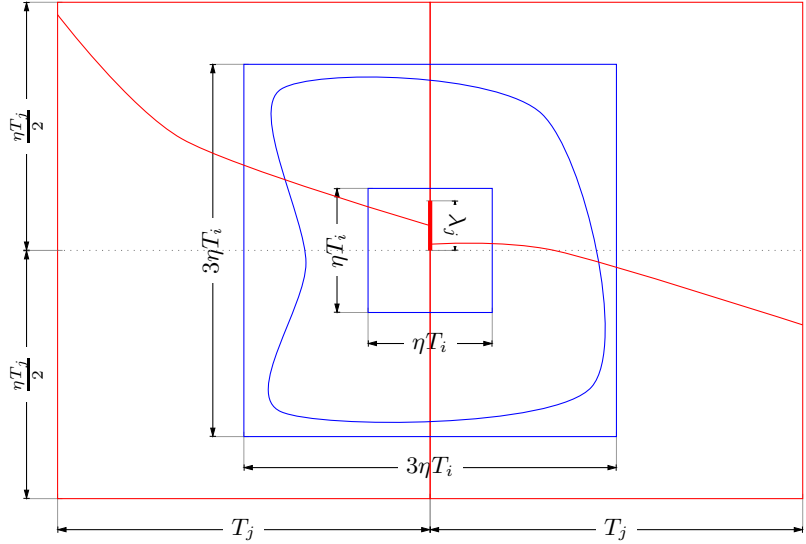
\begin{figure}[t]
\begin{centering}
\begin{tikzpicture}[x=0.65in,y=0.65in]
\dimline[extension start length=0.75*0.2in,extension end length=0.75*0.2in]{(3.2,2.4)}{(3.2,2)}{$\lambda_j$}
\dimline[extension start length=0.75*0.25in,extension end length=0.75*0.25in]{(4.5,0.25)}{(1.5,0.25)}{$3\eta T_i$}
\dimline[extension start length=0.75*0.25in,extension end length=0.75*0.25in]{(3.5,1.25)}{(2.5,1.25)}{$\eta T_i$}
\dimline[extension start length=0.75*0.25in,extension end length=0.75*0.25in]{(3,-0.25)}{(0,-0.25)}{$T_j$}
\dimline[extension start length=0.75*0.25in,extension end length=0.75*0.25in]{(6,-0.25)}{(3,-0.25)}{$T_j$}
\dimline[extension start length=0.75*0.25in,extension end length=0.75*0.25in]{(-0.25,0)}{(-0.25,2)}{$\frac{\eta T_j}{2}$}
\dimline[extension start length=0.75*0.25in,extension end length=0.75*0.25in]{(-0.25,2)}{(-0.25,4)}{$\frac{\eta T_j}{2}$}
\dimline[extension start length=0.75*0.25in,extension end length=0.75*0.25in]{(1.25,0.5)}{(1.25,3.5)}{$3\eta T_i$}
\dimline[extension start length=0.75*0.25in,extension end length=0.75*0.25in]{(2.25,1.5)}{(2.25,2.5)}{$\eta T_i$}

\draw [red] (0,0) -- (3,0) -- (3,4) -- (0,4) -- cycle;
\draw [red] (3,0) -- (6,0) -- (6,4) -- (3,4) -- cycle;
\draw [red, ultra thick] (3,2) -- (3,2.4);
\draw [gray,dotted] (0,2) -- (6,2);
\draw [blue] (2.5,1.5) -- (3.5,1.5) -- (3.5,2.5) -- (2.5,2.5) -- cycle;
\draw [blue] (1.5,0.5) -- (4.5,0.5) -- (4.5,3.5) -- (1.5,3.5) -- cycle;
\draw [thin,blue] plot [smooth cycle] coordinates { (1.8,0.7) (2,1.9) (1.8,3.3) (3.9,3.1) (4.3,0.9) };
\draw [thin,red] plot [smooth] coordinates { (0,3.9)  (1,2.9)  (3,2.2) };
\draw [thin,red] plot [smooth] coordinates { (6,1.4) (4,2) (3,2.05) };

\end{tikzpicture}
\par\end{centering}
\caption{\label{fig:ring}Geometric construction in the proof of \lemref{ring}.
The two red crossings are connected by the blue circuit. (Again we
omit the weight subscripts.)}
\end{figure}

By~\eqref{startcross}, our upper bound on $\eta$, and \corrref{stretch},
we have constants $c_{2}$ and $p_{5}$, depending only on $p_{3}$
and the facts that $\eta$ is a constant amount less than $5/4$ and
that $\gamma$ is sufficiently small, so that 
\[
\mathbf{P}[\Psi_{\mathrm{LR}}([0,3\eta T_{i})\times[0,\eta T_{i}))\le c_{2}y]\ge p_{5}.
\]
Let $E$ denote the event that there is a circuit of $Y_{\mathfrak{R}_{j}^{(\eta)}\cup(\mathfrak{R}_{j}^{(\eta)}+(T_{j},0))}$-weight
at most $5c_{2}y$ around $\mathfrak{L}$, and let $E_{1},E_{2},E_{3},E_{4}$
denote rotated and translated copies of $\{\Psi_{\mathrm{LR}}([0,3\eta T_{i})\times[0,\eta T_{i}))\le c_{2}y\}$
whose intersection, up to coarse field error, contains $E$. Now $\mathbf{P}[E_{\alpha}]=\mathbf{P}[\Psi_{\mathrm{LR}}([0,3\eta T_{i})\times[0,\eta T_{i}))\le c_{2}y]\ge p_{5},$
and so, by the \nameref{lem:FKG} and \eqref{relatescale}, we have
\begin{equation}
\mathbf{P}[E]\ge p_{5}^{4}-o_{j-i}(1).\label{eq:ann1}
\end{equation}

Since, up to coarse field error, we have
\begin{multline*}
\{\Psi_{\mathrm{LR}}([0,2T_{j})\times[0,\eta T_{j}))\le2w_{j}^{(\eta)}+5c_{2}y\}\\
\supset E\cap\{\Psi_{\mathrm{L},0,\lambda_{j}^{w_{j}^{(\eta)}}}(\mathfrak{R}_{j}^{(\eta)})\ge w_{i}^{(\eta)}\}\cap\{\Psi_{0,\lambda_{j}^{w_{j}^{(\eta)}},\mathrm{R}}(\mathfrak{R}_{j}^{(\eta)}+(T_{j},0))\le w_{i}^{(\eta)}\}
\end{multline*}
(see \figref{ring}), the \nameref{lem:FKG} inequality, along with~(\ref{eq:cross1}),~(\ref{eq:cross2}),
and~(\ref{eq:ann1}), tells us that
\[
\mathbf{P}[\Psi_{\mathrm{LR}}([0,2T_{j})\times[0,\eta T_{j}))\le3w_{j}^{(\eta)}+5c_{2}y]\ge(p_{0}/4)^{2}(p_{5}/2)^{4}-o_{j-i}(1),
\]
establishing the lemma with $c_{1}=5c_{2}$ and $p_{4}=(p_{0}/4)^{2}(p_{5}/2)^{4}$.
\end{proof}

\subsection{Multiscale analysis}

We now turn to the multiscale analysis involved in the proof of \thmref{rsw}.
\begin{lem}
\label{lem:use-independence}Let $c_{3}$ be so large that
\begin{equation}
(1-p_{4}^{15})^{\left\lfloor c_{3}/4\right\rfloor }\le p_{0}/8.\label{eq:c3cond}
\end{equation}
Suppose that $\eta\le\frac{256}{255}$ and that \eqref{startcross}
holds for $i$ and $y$. For any $\Delta\ge6$, there is a $j\in[i+\Delta,i+\Delta+c_{3}]$
so that, if $\gamma$ is sufficiently small \emph{relative to $\Delta$,
then}
\begin{equation}
\lambda_{j}^{21W_{j}^{(\eta)}+10c_{1}y}\ge\eta T_{i}.\label{eq:indep-concl}
\end{equation}
\end{lem}

\begin{proof}
Let $\tilde{j}=i+\Delta+c_{3}$. Suppose for the sake of contradiction
that, for all $i+\Delta<j\le\tilde{j}$, we have $\lambda_{j}^{w_{j}^{(\eta)}}<\eta T_{i}$,
and moreover that we have 
\begin{equation}
\lambda_{\tilde{j}}^{w_{\tilde{j}}^{(\eta)}+2(10W_{\tilde{j}}+5c_{1}y)}<\eta T_{i}.\label{eq:topone}
\end{equation}
Then \lemref{ring} implies that
\[
\mathbf{P}[\Psi_{\mathrm{LR}}([0,2T_{j})\times[0,\eta T_{j}))\le2w_{j}^{(\eta)}+c_{1}y]\ge p_{4}-o_{\Delta}(1)
\]
for each $i+\Delta<j\le\tilde{j}$. By \corrref{stretch}, this yields
\begin{equation}
\mathbf{P}[\Psi_{\mathrm{LR}}([0,3\eta T_{j})\times[0,\eta T_{j}))\le10w_{j}^{(\eta)}+5c_{1}y]\ge p_{4}^{5}-o_{\Delta}(1).\label{eq:Eprob}
\end{equation}
Let
\begin{align*}
J_{1} & =\{\Psi_{\mathrm{L},0,\eta T_{i}}(\mathfrak{R}_{\tilde{j}}^{(\eta)})\le w_{\tilde{j}}^{(\eta)}\}\text{ and}\\
J_{2} & =\{\Psi_{\mathrm{L},\eta T_{i},\eta T_{\tilde{j}}/2}(\mathfrak{R}_{\tilde{j}}^{(\eta)})\le w_{\tilde{j}}^{(\eta)}+2(10W_{\tilde{j}}+5c_{1}y)\}.
\end{align*}
Then~\eqref{topone} and \lemref{setuplambda}(\enuref{difference})
imply that we have
\begin{equation}
\mathbf{P}[J_{1}]-\mathbf{P}[J_{2}]\ge p_{0}/4.\label{eq:init-assumption}
\end{equation}
Let $E$ be the event that there is a path in $\mathfrak{R}_{\tilde{j}}^{(\eta)}$
of weight at most \textbf{$2(10W_{\tilde{j}}+5c_{1}y)$}, from $\{T_{\tilde{j}}-1\}\times[\eta T_{\tilde{j}}/2+\eta T_{i},\eta T_{\tilde{j}}]$
to $\{T_{j}-1\}\times[0,\eta T_{j}/2]$. Note that $J_{1}\cap E\subset J_{2},$
so
\begin{equation}
\mathbf{P}[J_{2}]\ge\mathbf{P}[J_{1}\cap E]\ge\mathbf{P}[J_{1}]\mathbf{P}[E]\label{eq:combination}
\end{equation}
by the \nameref{lem:FKG}. Combining~\eqref{init-assumption} and~\eqref{combination},
we get that
\begin{equation}
\mathbf{P}[E^{c}]\ge\mathbf{P}[J_{1}]\mathbf{P}[E^{c}]\ge p_{0}/4.\label{eq:Eprob2}
\end{equation}

For each $i+\Delta\le j<\tilde{j}$ such that $j\in4\mathbf{Z}$,
let $E_{1}^{j},E_{2}^{j},E_{3}^{j}$ be the events that there are
hard crossings—of weight at most \textbf{$10w_{j}^{(\eta)}+5c_{1}y$}—in
respectively, three rectangles of shorter side-length $\eta T_{j}$
and longer side-length $3\eta T_{j}$, that together form a ``C''
shape connecting $\{T_{\tilde{j}}-1\}\times[\eta T_{\tilde{j}}/2+\eta T_{i},\eta T_{\tilde{j}})$
to $\{T_{\tilde{j}}-1\}\times[0,\eta T_{j}/2)$, and which moreover
are chosen so that the blow-ups of the rectangles only intersect other
rectangles corresponding to the same $j$. The setup is illustrated
in \figref{cshapes}. 
\begin{figure}[p]
\begin{centering}
\tiny
\begin{tikzpicture}[x=0.73in,y=0.73in]
\draw [red] (-2,-2) -- (3,-2) -- (3,6) -- (-2,6) -- cycle;
\draw [red, ultra thick] (3,2) -- (3,2.25);
\draw [gray,dotted] (-2,2) -- (3,2);
\draw [blue,fill=blue,fill opacity=0.02] (3,2.5) -- (2.5,2.5) -- (2.5,2) -- (3,2)--  (3,1.5) -- (2,1.5) -- (2,3) -- (3,3) --   cycle;
\draw [blue,fill=blue,fill opacity=0.02] (3,3.25) -- (1.75,3.25) -- (1.75,1.25) -- (3,1.25)-- (3,-0.75)-- (-0.75,-0.75) -- (-0.75,5.25) -- (3,5.25) -- cycle;
\dimline[extension start length=0.25*0.7in, extension end length=0.25*0.7in]{(3,-2.25)}{(-2,-2.25)}{$T_{\tilde{j}}$}
\dimline[extension start length=0.25*0.7in, extension end length=0.25*0.7in]{(-2.25,-2)}{(-2.25,2)}{$\frac{\eta T_{\tilde{j}}}{2}$}
\dimline[extension start length=0.25*0.7in, extension end length=0.25*0.7in]{(-2.25,2)}{(-2.25,6)}{$\frac{\eta T_{\tilde{j}}}{2}$}
\dimline[extension start length=0.25*0.7in, extension end length=0.25*0.7in]{(3.25,2.25)}{(3.25,2)}{$\eta T_i$}
\dimline[extension start length=0.5*0.7in, extension end length=0.5*0.7in]{(3.5,2.5)}{(3.5,2)}{$\eta T_{i+\Delta}$}
\dimline[extension start length=0.75*0.7in, extension end length=0.75*0.7in]{(3.75,3)}{(3.75,2)}{$3\eta T_{i+\Delta}$}
\dimline[extension start length=1*0.7in, extension end length=1*0.7in]{(4,3.25)}{(4,2)}{$\eta T_{i+\Delta+4}$}
\dimline[extension start length=1.25*0.7in, extension end length=1.25*0.7in]{(4.25,5.25)}{(4.25,2)}{$3\eta T_{i+\Delta+4}$}
\end{tikzpicture}
\par\end{centering}
\caption{\label{fig:cshapes}Geometric construction in the proof of \lemref{use-independence}.
(In reality there would be \emph{many} more half-annuli!)}
\end{figure}

By~\eqref{Eprob} we have
\[
\mathbf{P}[E_{\alpha}^{j}]\ge p_{4}^{5}-o_{\Delta}(1).
\]
Let $\tilde{E}_{1}^{j},\tilde{E}_{2}^{j},\tilde{E}_{3}^{j}$ be defined
in the same way as $E_{1}^{j},E_{2}^{j},E_{3}^{j}$, except with the
requirement that the $Y_{\mathfrak{R}_{\tilde{j}}^{(\eta)}}$-weight
of the paths be at most \textbf{$2(10w_{j}^{(\eta)}+5c_{1}y)$,} rather
than that the weight of the paths with respect to the GFF in their
own rectangles be at most \textbf{$10w_{j}^{(\eta)}+5c_{1}y$}.

For each $j$, we have that $\tilde{E}_{1}^{j}\cap\tilde{E}_{2}^{j}\cap\tilde{E}_{3}^{j}\subset E$.
Let $Z=\max\limits _{\alpha,j,x}W_{\alpha}^{j}(x)$, where $W_{\alpha}^{j}$
is the coarse field correction term for the rectangle in $E_{\alpha}^{j}$
as in \eqref{relatescale}. Now note that $\tilde{E}_{1}^{j}\cap\tilde{E}_{2}^{j}\cap\tilde{E}_{3}^{j}\supset E_{1}^{j}\cap E_{2}^{j}\cap E_{3}^{j}\cap\{Z\le2\}$,
so we can compute, using the independence of the fine fields,
\begin{align*}
\mathbf{P}[E^{c}] & \le\mathbf{P}\left[\bigcap_{\substack{i+\Delta\le j<\tilde{j}\\
j\in4\mathbf{Z}
}
}(\tilde{E}_{1}^{j}\cap\tilde{E}_{2}^{j}\cap\tilde{E}_{3}^{j})^{c}\right]\\
 & \le\mathbf{P}\left[\bigcap_{\substack{i+\Delta\le j<\tilde{j}\\
j\in4\mathbf{Z}
}
}\left((E_{1}^{j}\cap E_{2}^{j}\cap E_{3}^{j})^{c}\cup\{Z>2\}\right)\right]\\
 & \le o_{\Delta}(1)+\prod_{\substack{i+\Delta\le j<\tilde{j}\\
j\in4\mathbf{Z}
}
}(1-\mathbf{P}[E_{1}^{j}\cap E_{2}^{j}\cap E_{3}^{j}])\\
 & \le\prod_{\substack{i+\Delta\le j<\tilde{j}\\
j\in4\mathbf{Z}
}
}(1-\mathbf{P}[E_{1}^{j}]^{3})+o_{\Delta}(1)\\
 & \le(1-p_{4}^{15})^{\left\lfloor c_{3}/4\right\rfloor }+o_{\Delta}(1).
\end{align*}
Now if $\gamma$ is small enough (relative to $\Delta$), then combined
with~\eqref{c3cond} this implies that $\mathbf{P}[E^{c}]<p_{0}/4,$
contradicting~\eqref{Eprob2}. So either, for some $i+\Delta<j\le\tilde{j}$,
we have $\lambda_{j}^{w_{j}^{(\eta)}}\ge\eta T_{i}$, or else we have
$\lambda_{\tilde{j}}^{w_{\tilde{j}}^{(\eta)}+2(10W_{\tilde{j}}^{(\eta)}+5c_{1}y)}\ge\eta T_{i}$,
implying \eqref{indep-concl} in any case.
\end{proof}
\begin{lem}
\label{lem:superlinear}Write $f(k)=\lambda_{k}^{y_{k}}$ for some
sequence $y_{k}\ge w_{k}^{(\eta)}$. Suppose that $f(k_{0})\ge a\eta T_{k_{0}}$.
Then if $c$ is so large that $\left(\frac{7}{4}\right)^{c}\frac{1}{\sqrt{2}^{c-1}}>\frac{1}{a},$
then there is a $k\in(k_{0},k_{0}+c]$ so that $f(k)\le\frac{7}{4}f(k-1)$
and $f(k-1)\ge a\eta T_{k-1}.$
\end{lem}

\begin{proof}
If $f(k)\ge\frac{7}{4}f(k-1)$ for all $k_{0}<k\le k_{0}+c$, then
we have (using~(\ref{eq:Tij-estimates}))
\[
\frac{1}{2}\eta T_{k_{0}+c}\ge f(k_{0}+c)\ge\left(7/4\right)^{c}a\eta T_{k_{0}}\ge\left(7/4\right)^{c}a\eta\frac{1}{\sqrt{2}^{c+1}}T_{k_{0}+c},
\]
contradicting our assumption on $c$. Therefore, there is some $k\in(k_{0},k_{0}+c]$
so that $f(k)\le\frac{7}{4}f(k-1)$. Moreover, if we choose the \emph{first}
such $k$, then we have 
\[
\frac{f(j)}{T_{j}}\ge\frac{7}{4}\cdot\frac{f(j-1)}{T_{j}}\ge\frac{7}{4}\cdot\frac{f(j-1)}{3T_{j-1}/2}\ge\frac{f(j-1)}{T_{j-1}}
\]
for all $k_{0}<j<k$, so by induction we have $f(k-1)\ge a\eta T_{k-1}$. 
\end{proof}
\begin{lem}
\label{lem:rsw-inductive-step}Let $c_{5}=\max\{1386,660c_{1}\}.$
Fix $\Delta\ge6$ and suppose that 
\begin{equation}
1<\eta\le1+\frac{1}{32\sqrt{2}^{\Delta+c_{3}+1}}.\label{eq:etasmall}
\end{equation}
Then if $\gamma$ is sufficiently small relative to $\Delta$, then
there is a $\chi(\Delta)\ge\Delta$ so that if (\ref{eq:startcross})
holds for $i$ and $y$, then there is a $k\in[i+\Delta,i+\chi(\Delta)]$
so that (\ref{eq:startcross}) holds for $i=k$ and $y=c_{5}(W_{k}^{(\eta)}+y)$.
\end{lem}

\begin{proof}
By \lemref{use-independence}, there is a $j\in[i+\Delta,i+\Delta+c_{3}]$
so that (using \eqref{Tij-estimates}) we have $\lambda_{j}^{21W_{j}^{(\eta)}+10c_{1}y}\ge\eta T_{i}\ge\frac{\eta T_{j}}{\sqrt{2}^{\Delta+c_{3}+1}}.$
Let $\xi(\Delta)$ be so large that $\frac{(7/4)^{\xi(\Delta)}}{\sqrt{2}^{\xi(\Delta)-1}}>\sqrt{2}^{\Delta+c_{3}+1}.$
Then if we put $f(k)=\lambda_{k}^{11eW_{k}^{(\eta)}+5ec_{1}y}$ and
let $\chi(\Delta)=\xi(\Delta)+c_{3}$, by \lemref{superlinear} there
is some $k\in(j,j+\xi(\Delta)]\subset[i+\Delta,i+\Delta+c_{3}+\xi(\Delta)]=[i+\Delta,i+\chi(\Delta)]$
so that
\[
\lambda_{k}^{21W_{k}^{(\eta)}+10c_{1}y}\le\frac{7}{4}\lambda_{k-1}^{21W_{k-1}^{(\eta)}+10c_{1}y}
\]
and
\[
\frac{1}{32}\lambda_{k-1}^{21W_{k-1}^{(\eta)}+10c_{1}y}\ge\frac{\eta T_{k-1}}{32\sqrt{2}^{\Delta+c_{3}+1}}\ge\frac{T_{k-1}}{32\sqrt{2}^{\Delta+c_{3}+1}}>(\eta-1)T_{k-1},
\]
with the last inequality by \eqref{etasmall}. Thus the hypotheses
of \lemref{usethex} hold with $i=k$, $y=21W_{k}^{(\eta)}+10c_{1}y$,
and $z=21W_{k-1}^{(\eta)}+10c_{1}y$ (where the left-hand sides are
in the notation of the statement of \lemref{usethex} and the right-hand
sides are in the notation of the present proof). This means that
\[
\mathbf{P}[\Psi_{\mathrm{LR}}([0,\tfrac{5T_{k}}{4})\times[0,\eta T_{k}))\ge55(21W_{k}^{(\eta)}+10c_{1}y)+11(21W_{k-1}^{(\eta)}+10c_{1}y)]\ge p_{3},
\]
which is to say that \eqref{startcross} holds with $y=c_{5}(W_{k}^{(\eta)}+y)$
(where again the left-hand side is in the notation of \eqref{startcross}
and the right-hand side is in the present notation).
\end{proof}
\begin{lem}
\label{lem:tassion-conclusion}Fix $\Delta\ge6$ and suppose that
$\eta-1\le2^{-(\Delta+c_{3}+7)}.$ Then there is an increasing sequence
$1=i_{1},i_{2},i_{3},\ldots$ so that
\begin{equation}
i_{r+1}\in[i_{r}+\Delta,i_{r}+\chi(\Delta)],\label{eq:irbound}
\end{equation}
and, for each $r$, \eqref{startcross} holds for $i=i_{r}$ and
\begin{equation}
y=\sum_{s=1}^{r}c_{5}^{r+1-s}(W_{i_{s}}^{(\eta)}\vee y_{1}^{*}),\label{eq:recursion}
\end{equation}
where we recall that $y_{1}^{*}$ is defined to be the quantity satisfying
\eqref{blowupthestartweight}.
\end{lem}

\begin{proof}
According to \eqref{blowupthestartweight}, we have $\lambda_{1}^{y_{1}^{*}}=\eta T_{1}/8.$
This means that \lemref{usethex} applies, so~\eqref{startcross}
holds for $i=1$ and $y=42y_{1}^{*}$. In other words, if we put $i_{1}=1$,
then the conclusion of the lemma holds for $r=1$.

Now we claim that once we have chosen a suitable $i_{r}$, then we
can also choose a suitable $i_{r+1}$. Indeed, if \eqref{startcross}
holds for $i=i_{r}$ and 
\[
y=\sum_{s=1}^{r}c_{5}^{r+1-s}(W_{i_{s}}^{(\eta)}\vee y_{1}^{*}),
\]
then \lemref{rsw-inductive-step} implies that there is an $i_{r+1}\in[i_{r}+\Delta,i_{r}+\chi(\Delta)]$
so that~\eqref{startcross} holds for $i=i_{r+1}$ and
\[
y=c_{5}\left(W_{i_{r+1}}+\sum_{s=1}^{r}c_{5}^{r+1-s}(W_{i_{s}}^{(\eta)}\vee y_{1}^{*})\right)\le\sum_{s=1}^{r+1}c_{5}^{r+2-s}(W_{i_{s}}^{(\eta)}\vee y_{1}^{*}),
\]
hence for $y=\sum_{s=1}^{r+1}c_{5}^{r+2-s}(W_{i_{s}}^{(\eta)}\vee y_{1}^{*})$
as well. This finishes the inductive step of the proof of the lemma.
\end{proof}
The next lemma uses the fact that our desired results at a given scale
imply the same results at constant multiples of the scale to extend
\lemref{tassion-conclusion} to all scales, and also to better-shaped
boxes.
\begin{lem}
\label{lem:roundit}Fix $\Delta\ge6$ and suppose that $\eta-1\le2^{-(\Delta+c_{3}+7)}.$
We have constants $p(\Delta)$ and $C(\Delta)$ so that for each $i\ge1$,
we have 
\[
\Theta_{\mathrm{hard}}(\mathfrak{R}_{i})\left[p(\Delta)\right]\le C(\Delta)\sum_{j=0}^{\lfloor i/\Delta\rfloor}c_{5}^{j}(W_{i-1-j\Delta}^{(\eta)}\vee y_{1}^{*}).
\]
\end{lem}

\begin{proof}
By \lemref{tassion-conclusion}, there is an $i_{r}$ so that $i-1-\chi(\Delta)\le i_{r}\le i-1$
and 
\begin{equation}
\mathbf{P}[\Psi_{\mathrm{LR}}([0,5T_{i_{r}}/4)\times[0,\eta T_{i_{r}}))\le y_{r}]\ge p_{3},\label{eq:startcross-2}
\end{equation}
where
\[
y_{r}=\sum_{\alpha=1}^{r}c_{5}^{r+1-\alpha}(W_{i_{\alpha}}^{(\eta)}\vee y_{1}^{*}).
\]
Note that \eqref{irbound} implies that, for each $\alpha$, we have
$i_{q}\le i_{r}-(r-\alpha)\Delta$. This means that
\begin{align*}
y_{r}\le\sum_{\alpha=1}^{r}c_{5}^{r+1-\alpha}(W_{i_{r}-(r-\alpha)\Delta}^{(\eta)}\vee y_{1}^{*}) & =\sum_{j=0}^{r-1}c_{5}^{j+1}(W_{i_{r}-j\Delta}^{(\eta)}\vee y_{1}^{*})\\
 & \le\sum_{j=0}^{r-1}c_{5}^{j+1}(W_{i-1-j\Delta}^{(\eta)}\vee y_{1}^{*}).
\end{align*}
Now \corrref{stretch} and \eqref{relatescale} imply the desired
result.
\end{proof}
We are finally ready to prove our RSW result.
\begin{proof}[Proof of \thmref{rsw}]
Choose $\kappa$ so large that
\begin{align}
a_{\mathrm{pl}}^{\kappa} & <\tfrac{1}{4c_{5}}\text{ and}\label{eq:kappacond}\\
c_{5}^{\frac{1}{2\kappa}} & <1/a_{\mathrm{pl}}.\label{eq:kappacond2}
\end{align}
Put $\Delta=\lceil2\kappa\rceil$ and apply \lemref{roundit}. Fix
$\eta$ as in the statement of that lemma; then we have
\begin{align}
\Theta_{\mathrm{hard}}\mathrlap{(\mathfrak{R}_{i})[p(\Delta)]}\label{eq:rswfinalsplit}\\
 & \le C(\Delta)\sum_{j=0}^{\lfloor\frac{i}{2\kappa}\rfloor}c_{5}^{j}(W_{i-1-j\Delta}^{(\eta)}\vee y_{1}^{*})\nonumber \\
 & \le C_{\mathrm{str}}(\eta)C(\Delta)\sum_{j=0}^{\lfloor\frac{i}{2\kappa}\rfloor}c_{5}^{j}\left(\max_{\alpha\le i-1-j\Delta}\Theta_{\mathrm{easy}}(\mathfrak{R}_{\alpha})[p_{\mathrm{str}}(\eta)]\vee y_{1}^{*}\right)\nonumber \\
 & \le C_{\mathrm{str}}(\eta)C(\Delta)\sum_{j=0}^{\lfloor\frac{i}{2\kappa}\rfloor}c_{5}^{j}\max_{\alpha\le i-1-j\Delta}\Theta_{\mathrm{easy}}(\mathfrak{R}_{\alpha})[p_{\mathrm{str}}(\eta)]+C_{3}y_{1}^{*}\sum_{j=0}^{\lfloor\frac{i}{2\kappa}\rfloor}c_{5}^{j},\nonumber 
\end{align}
with the second inequality by \corrref{squishdown-1}.

Our goal is to relate the sums in \eqref{rswfinalsplit} to a quantile
of an easy crossing of $\mathfrak{R}_{i}$, and our primary tool will
be the \emph{a priori} power-law lower bound of sufficiently small
crossing quantiles given in \propref{exponential-crossings}. However,
\propref{exponential-crossings} only relates very small quantiles,
and the quantiles in \eqref{rswfinalsplit} (coming from \corrref{squishdown-1})
are very large. This is the reason for the assumption \eqref{cvbound-1}:
by applying \eqref{concentration-quantiles}, this assumption lets
us relate very small and very large quantiles, assuming $\delta$
is chosen sufficiently small.

Now we put this plan into action. For each $j$, we have
\[
\max_{\alpha\le i-1-j\Delta}\Theta_{\mathrm{easy}}(\mathfrak{R}_{\alpha})[p_{\mathrm{str}}(\eta)]\le C\max_{\alpha\le i-1-j\Delta}\Theta_{\mathrm{easy}}(\mathfrak{R}_{\alpha})[p_{\mathrm{pl}}]
\]
(with $p_{\mathrm{pl}}$ as in \propref{exponential-crossings}) by
\eqref{cvbound-1} and \eqref{concentration-quantiles}, choosing
$\delta$ small enough (depending on $p_{\mathrm{str}}$ and $p_{\mathrm{pl}}$)
so that the necessary assumptions hold. But then by \propref{exponential-crossings},
we have
\[
\max_{\alpha\le i-1-j\Delta}\Theta_{\mathrm{easy}}(\mathfrak{R}_{\alpha})[p_{\mathrm{str}}(\eta)]\le CC_{\mathrm{pl}}a_{\mathrm{pl}}^{j\Delta+1}\cdot\Theta_{\mathrm{easy}}(\mathfrak{R}_{i})[q_{\mathrm{pl}}].
\]
This gives us 
\begin{align}
\sum_{j=0}^{\lfloor\frac{i}{2\kappa}\rfloor}c_{5}^{j}\max_{\alpha\le i-1-j\Delta}\Theta_{\mathrm{easy}}(\mathfrak{R}_{\alpha})[p_{\mathrm{str}}(\eta)] & \le CC_{\mathrm{pl}}\Theta_{\mathrm{easy}}(\mathfrak{R}_{i})[q]\sum_{j=0}^{\lfloor\frac{i}{2\kappa}\rfloor}c_{5}^{j}a_{\mathrm{pl}}^{j\Delta+1}\label{eq:rswfinalcomb1}\\
 & \le C'\Theta_{\mathrm{easy}}(\mathfrak{R}_{i})[q],\nonumber 
\end{align}
where in the last inequality we use \eqref{kappacond}. Moreover,
we have
\begin{equation}
\sum_{j=0}^{\lfloor\frac{i}{2\kappa}\rfloor}c_{5}^{j}\le\frac{c_{5}^{\frac{i}{2\kappa}+1}-1}{c_{5}-1}\le C''\Theta_{\mathrm{easy}}(\mathfrak{R}_{i})[q],\label{eq:rswfinalcomb2}
\end{equation}
with the last inequality by \eqref{kappacond2} and \propref{exponential-crossings}.

Now choose 
\[
p_{\mathrm{RSW}}\le\min\{p(\Delta),p_{\mathrm{pl}},(32\cdot d_{\mathrm{\mathrm{p}}}^{2})^{-2}\}.
\]
Plugging \eqref{rswfinalcomb1} and \eqref{rswfinalcomb2} into \eqref{rswfinalsplit},
we obtain that \eqref{prswsmall} holds and that
\begin{align*}
\Theta_{\mathrm{hard}}(\mathfrak{R}_{S})[p_{\mathrm{RSW}}] & \le\Theta_{\mathrm{hard}}(\mathfrak{R}_{S})[p(\Delta)]\\
 & \le C'''\Theta_{\mathrm{easy}}(\mathfrak{R}_{S})[q]\\
 & \le C'''\Theta_{\mathrm{easy}}(\mathfrak{R}_{S})[p_{\mathrm{RSW}}].
\end{align*}
Here, the second inequality is by \eqref{concentration-quantiles}
and \eqref{cvbound-1} as long as $\delta$ is sufficiently small
compared to $p_{\mathrm{RSW}}$ .
\end{proof}

\section{Upper bounds on FPP distance and geodesic length\label{sec:upperbounds}}

In this section we derive upper bounds on the crossing distance, geodesic
length, and box diameter.

\subsection{Crossing distance upper bound}

We want to derive a right-tail bound on the crossing distance in terms
of the hard crossing distance at a smaller scale. We show this by
showing that hard crossings from smaller scales can be glued together
to get a crossing at a larger scale, and that there are many nearly-independent
opportunities for this to happen, so we get good control on the right
tail of the crossing distance.

Let $\mathfrak{R}=[0,KS)\times[0,LS)$. Let $\mathfrak{C}=[0,S)^{2}$
and $\mathfrak{A}=[0,S)\times[0,2S)$. Index the dyadic subboxes of
$\mathfrak{R}$ having side length $S$ by row and column according
to the following layout: 
\[
\begin{array}{ccc}
\mathfrak{\mathfrak{C}}_{11} & \cdots & \mathfrak{\mathfrak{C}}_{1L}\\
\vdots & \ddots & \vdots\\
\mathfrak{\mathfrak{C}}_{K1} & \cdots & \mathfrak{\mathfrak{C}}_{KL}
\end{array}
\]

\begin{prop}
\label{prop:tailbound}If $u>0$, we have 
\begin{equation}
\mathbf{P}[\Psi_{\mathrm{LR}}(\mathfrak{R})\ge2uK\mathbf{E}\Psi_{\mathrm{hard}}(\mathfrak{A})]\le u^{-L/3}+o_{K,L}(1).\label{eq:tailindep-1}
\end{equation}
Moreover, if $u\ge u_{0}$ (defined in \eqref{tailgff}), then we
have 
\begin{equation}
\mathbf{P}[\Psi_{\mathrm{LR}}(\mathfrak{R})\ge2uK\mathbf{E}\Psi_{\mathrm{hard}}(\mathfrak{A})]\le u^{-L/4}+\exp\left(-\omega(1)\cdot\frac{(\log u)^{2}}{\log(K\vee L)}\right).\label{eq:longtailindep}
\end{equation}
Finally, as long as $L\ge10$ we have
\begin{equation}
\mathbf{E}\Psi_{\mathrm{LR}}(\mathfrak{R})^{3}\le O_{K,L}(1)(\mathbf{E}\Psi_{\mathrm{hard}}(\mathfrak{A})^{2})^{3/2}.\label{eq:cross3momentapriori}
\end{equation}
\end{prop}

\begin{proof}
For each $0\le j\le L-1$ such that $3\mid j$, let $\Psi_{j}=\Psi_{\mathrm{LR}}((0,jS)+[0,KS)\times[0,S)).$
(Note here that $(0,jS)$ is a point in $\mathbf{Z}^{2}$, not an
open interval.) Then for each $j$, by \eqref{relatescale} and the
strategy illustrated in \figref{hardglue} we have
\[
\Psi_{j}\le\sum_{i=1}^{K-1}\Psi_{\mathrm{hard}}(\mathfrak{\mathfrak{C}}_{j,i}\cup\mathfrak{\mathfrak{C}}_{j,i+1})+\sum_{i=2}^{K-1}\Psi_{\mathrm{hard}}(\mathfrak{\mathfrak{C}}_{j,i}\cup\mathfrak{\mathfrak{C}}_{j\pm1,i}).
\]
Thus we have
\begin{equation}
\mathbf{E}\Psi_{j}\le(1+o_{K,L}(1))(2K-3)\mathbf{E}\Psi_{\mathrm{hard}}(\mathfrak{A})\le2K\mathbf{E}\Psi_{\mathrm{hard}}(\mathfrak{A})\label{eq:tailexp}
\end{equation}
as long as $\gamma$ is sufficiently small compared to $K$ and $L$.
Applying Markov's inequality gives us 
\[
\mathbf{P}[\Psi_{j}\ge2uK\mathbf{E}\Psi_{\mathrm{hard}}(\mathfrak{A})]\le1/u.
\]
Since up to coarse field error we have $\Psi_{\mathrm{LR}}(\mathfrak{R})\le\min_{j}\Psi_{j}$,
and the set $\{\Psi_{j}\mid0\le j\le L-1\text{ and \ensuremath{3} divides \ensuremath{j}}\}$
is independent, we have \eqref{tailindep-1} by \eqref{relatescale}
and \eqref{longtailindep} by \eqref{tailgff} and the assumption
that $u\ge u_{0}$. Finally, the Cauchy–Schwarz inequality and \lemref{minmoment}
below give us 
\begin{equation}
\mathbf{E}\Psi_{\mathrm{LR}}(\mathfrak{R})^{3}\le O_{K,L}(1)\cdot(\mathbf{E}\Psi_{j}^{2})^{3/2}\le O_{K,L}(1)(\mathbf{E}\Psi_{\mathrm{hard}}(\mathfrak{A})^{2})^{3/2}\label{eq:crossmomentapriori}
\end{equation}
as long as $L\ge10$.
\end{proof}
\begin{lem}
\label{lem:minmoment}Let $Y_{1},\ldots,Y_{k}$ be iid random variables
such that that $\mu=\mathbf{E}Y_{i}<\infty$. Let $Z=\min\{Y_{1},\ldots,Y_{k}\}$.
Then for any $a<k$, we have $\mathbf{E}Z^{a}\le\left(1+\frac{a}{k-a}\right)\mu^{a}.$
\end{lem}

\begin{proof}
Simply compute
\begin{multline*}
\mathbf{E}Z^{a}=\int_{0}^{\infty}\mathbf{P}(Y^{a}\ge u)\,du=\int_{0}^{\infty}\mathbf{P}(Y_{1}\ge u^{1/a})^{k}\,du\\
\le\int_{0}^{\infty}\left(1\wedge\frac{\mu}{u^{1/a}}\right)^{k}\,du=\left(1+\frac{a}{k-a}\right)\mu^{a}.\qedhere
\end{multline*}
\end{proof}
\begin{cor}[of \propref{tailbound}]
\label{corr:polyupper}If $\gamma$ is sufficiently small, then there
are constants $C<\infty$ and $b_{\mathrm{pl}}=1+o(1)$ so that for
any $K$ and $S$ we have
\[
\mathbf{E}\Psi_{\mathrm{hard}}([0,2^{r}S)\times[0,2^{r+1}S))\le Cb_{\mathrm{pl}}^{r}\mathbf{E}\Psi_{\mathrm{hard}}([0,S)\times[0,2S)).
\]
Moreover, $b_{\mathrm{pl}}$ can be made arbitrarily close to $1$
by making $\gamma$ sufficiently small.
\end{cor}

\begin{proof}
By \eqref{tailexp}, in the notation of \propref{tailbound} we have
$\mathbf{E}\Psi_{\mathrm{hard}}(\mathfrak{R})\le(2+o_{K,L}(1))K\mathbf{E}\Psi_{\mathrm{hard}}(\mathfrak{A}).$
The statement then follows by induction on the scale after choosing
$K,L$ sufficiently large and $\gamma$ sufficiently small.
\end{proof}

\subsection{Expected geodesic length upper bound}

Let $\mathfrak{R}=[0,KS)\times[0,LS)$ with $K=2^{k}$ and $L=2^{l}$.
Let $\mathfrak{A}=[0,S)\times[0,2S)$. We want to show that a left–right
crossing of $\mathfrak{R}$ will typically not enter too many dyadic
$S\times S$ subboxes of $\mathfrak{R}$. Our strategy will be to
show that a path that enters many boxes will likely have a higher
weight than the tail-bound value obtained from the ``default'' paths
in \propref{tailbound}. Recall the notation $M_{\bullet;S}$ defined
in \subsecref{pathnotation}.
\begin{prop}
For any $u>0$ and $p\in(0,1)$, we have
\begin{multline*}
\mathbf{P}\left[M_{\mathrm{LR};S}(\mathfrak{R})\ge K\max\left\{ 1,4uu_{0}c_{\mathrm{PD}}\frac{\mathbf{E}\Psi_{\mathrm{hard}}(\mathfrak{A})}{\Theta_{\mathrm{easy}}(\mathfrak{A})[p]}\right\} \right]\\
\le u^{-L/3}+C_{\mathrm{p}}L\left(2d_{\mathrm{p}}^{2}\sqrt{p}\right)^{K}+o_{K,L}(1).
\end{multline*}
\end{prop}

\begin{proof}
By \propref{tailbound}, with probability at least $1-u^{-L/3}-o_{K,L}(1),$
we have
\[
\Psi_{\mathrm{LR}}(\mathfrak{R})\le2uK\mathbf{E}\Psi_{\mathrm{hard}}(\mathfrak{A}).
\]
On the other hand, by \propref{lowerbound-tail} and \propref{longpathpasses},
with probability at least $1-C_{\mathrm{p}}L\left(2d_{\mathrm{p}}^{2}\sqrt{p}\right)^{N}-o_{K,L}(1)$
we have
\[
\min_{\|\pi\|_{S}\ge c_{\mathrm{PD}}N}\psi(\pi;Y_{\mathfrak{R}})>\frac{N}{2u_{0}}\Theta_{\mathrm{easy}}(\mathfrak{A})[p].
\]
Thus if
\[
\frac{N}{2u_{0}}\Theta_{\mathrm{easy}}(\mathfrak{A})[p]\ge2uK\mathbf{E}\Psi_{\mathrm{hard}}(\mathfrak{A}),
\]
then with probability at least $1-u^{-L/3}-C_{\mathrm{p}}L\left(2d_{\mathrm{p}}^{2}\sqrt{p}\right)^{N}-o_{K,L}(1),$
we have $M_{\mathrm{LR};S}(\mathfrak{R})\le c_{\mathrm{PD}}N.$ Putting
\[
N=K\max\left\{ 1,4u_{0}u\frac{\mathbf{E}\Psi_{\mathrm{hard}}(\mathfrak{A})}{\Theta_{\mathrm{easy}}(\mathfrak{A})[p]}\right\} 
\]
yields the desired result.
\end{proof}
\begin{prop}
\label{prop:crossinglengthexp}There is a $\delta_{0}>0$ and a $C_{\mathrm{CL}}>0$
so that the following holds. If $\CV^{2}(\Psi_{\mathrm{easy}}(\mathfrak{E}))<\delta_{\mathrm{RSW}}$
whenever $\mathfrak{E}\subseteq[0,S)\times[0,2S)$ has aspect ratio
between $1/2$ and $2$ inclusive, and $\CV^{2}(\Psi_{\mathrm{hard}}(\mathfrak{A}))<\delta<\delta_{0}$,
then we have
\[
\mathbf{E}M_{\mathrm{LR};S}(\mathfrak{R})\le K\left(C_{\mathrm{CL}}+L\left[2^{-L/3}+C_{\mathrm{p}}L\left(2d_{\mathrm{p}}^{2}\sqrt{p_{\mathrm{RSW}}}\right)^{K}\right]\right)+o_{K,L}(1).
\]
\end{prop}

\begin{rem}
Note that \eqref{prswsmall} implies that the third term decays geometrically
as $K\to\infty$.
\end{rem}

\begin{proof}
Putting $p=p_{\mathrm{RSW}}$ in the previous lemma, we have, for
any $u>0$,
\begin{align*}
\mathbf{E}M_{\mathrm{LR};S}(\mathfrak{R}) & \le K\max\left\{ 1,4u_{0}uc_{\mathrm{PD}}\frac{\mathbf{E}\Psi_{\mathrm{hard}}(\mathfrak{A})}{\Theta_{\mathrm{easy}}(\mathfrak{A})[p_{\mathrm{RSW}}]}\right\} \\
 & \qquad+KL\left[u^{-L/3}+C_{\mathrm{p}}L\left(2d_{\mathrm{\mathrm{p}}}^{2}\sqrt{p_{\mathrm{RSW}}}\right)^{K}\right]+o_{K,L}(1).
\end{align*}
Then, since our assumption implies that the hypothesis of \thmref{rsw}
holds at scale $S$, putting $u=2$ we obtain
\begin{align*}
\mathbf{E}M_{\mathrm{LR};S}(\mathfrak{R}) & \le K\max\left\{ 1,8u_{0}c_{\mathrm{PD}}\frac{\mathbf{E}\Psi_{\mathrm{hard}}(\mathfrak{A})}{\Theta_{\mathrm{easy}}(\mathfrak{A})[p_{\mathrm{RSW}}]}\right\} \\
 & \qquad+KL\left[2^{-L/3}+C_{\mathrm{p}}L\left(2d_{\mathrm{p}}^{2}\sqrt{p_{\mathrm{RSW}}}\right)^{K}\right]+o_{K,L}(1)\\
 & \le K\max\left\{ 1,8u_{0}c_{\mathrm{PD}}C_{\mathrm{RSW}}\frac{\mathbf{E}\Psi_{\mathrm{hard}}(\mathfrak{A})}{\Theta_{\mathrm{hard}}(\mathfrak{A})[p_{\mathrm{RSW}}]}\right\} \\
 & \qquad+KL\left[2^{-L/3}+C_{\mathrm{p}}L\left(2d_{\mathrm{p}}^{2}\sqrt{p_{\mathrm{RSW}}}\right)^{K}\right]+o_{K,L}(1).
\end{align*}
Finally, using the assumption that $\CV^{2}(\Psi_{\mathrm{hard}}(\mathfrak{A}))<\delta$,
if $\delta$ is chosen sufficiently small compared to $p_{\mathrm{RSW}}$,
Chebyshev's inequality (or \eqref{concentration-main}) implies the
result.
\end{proof}

\subsection{Diameter upper bound\label{subsec:diameter}}

We now turn our attention to the problem of estimating the point-to-point
distance between two points in a box, using a chaining argument to
take advantage of our good tail bound established in \propref{tailbound}.

Fix a scale $S=2^{s}$. Let $\mathfrak{R}=[0,S)\times[0,2S)$. For
$t\in[0,s]$ and $(i,j)\in[0,2^{t})^{2}$, put 
\[
\mathfrak{R}_{t;i,j}=\begin{cases}
(i\cdot2^{s-t},2\cdot j\cdot2^{s-t})+[0,2^{s-t})\times[0,2\cdot2^{s-t}) & \text{\ensuremath{t} even}\\
(2\cdot i\cdot2^{s-t},j\cdot2^{s-t})+[0,2\cdot2^{s-t})\times[0,2^{s-t}) & \text{\ensuremath{t} odd.}
\end{cases}
\]
For convenience, put $\mathfrak{A}_{t}=\mathfrak{R}_{t;0,0}$.
\begin{figure}[t]
\begin{centering}
\hfill{}\subfloat[\label{fig:pointtopoint}Using two hard crossings at each scale to
connect to any two points.]{\begin{centering}
\tiny
\begin{tikzpicture}[x=0.35in,y=0.35in]
\draw[step=1,thin] (0,0) grid (4,8);
\draw[fill=yellow,opacity=0.1] (0,0) rectangle (4,8);
\draw[fill=yellow,opacity=0.2] (0,0) rectangle (4,2);
\draw[fill=yellow,opacity=0.4] (0,0) rectangle (1,2);
\draw[fill=yellow,opacity=0.4] (2,4) rectangle (3,6);
\draw[fill=yellow,opacity=0.2] (0,4) rectangle (4,6);
\draw[ultra thick, red,style={decorate,decoration={snake,amplitude=4}}] (2.1,0) -- (1.9,8);
\draw[very thick, red,style={decorate,decoration={snake,amplitude=2}}] (0,1.1) -- (4,0.9);
\draw[very thick, red,style={decorate,decoration={snake,amplitude=2}}] (0,5.1) -- (4,4.9);
\draw[very thick, red,style={decorate,decoration={snake,amplitude=2}}] (0,5.1) -- (4,4.9);
\draw[thick, red,style={decorate,decoration={snake,amplitude=1}}] (0.4,0) -- (0.5,2);
\draw[thick, red,style={decorate,decoration={snake,amplitude=1}}] (2.5,4) -- (2.5,6);
\filldraw [blue] (2.5,5.5) circle (2.5pt);
\filldraw [blue] (0.5,1.5) circle (2.5pt);
\end{tikzpicture}
\par\end{centering}
}\hfill{}\subfloat[\label{fig:chaining}The chaining argument takes the maximum of the
hard crossings at each scale.]{\begin{centering}
\tiny
\begin{tikzpicture}[x=0.35in,y=0.35in]
\draw[step=1,thin] (0,0) grid (4,8);
\draw[ultra thick, red,style={decorate,decoration={snake,amplitude=4}}] (2,0) -- (2,8);
\draw[very thick, red,style={decorate,decoration={snake,amplitude=2}}] (0,1) -- (4,1);
\draw[very thick, red,style={decorate,decoration={snake,amplitude=2}}] (0,3) -- (4,3);
\draw[very thick, red,style={decorate,decoration={snake,amplitude=2}}] (0,5) -- (4,5);
\draw[very thick, red,style={decorate,decoration={snake,amplitude=2}}] (0,7) -- (4,7);
\draw[thick, red,style={decorate,decoration={snake,amplitude=1}}] (0.4,0) -- (0.5,2);
\draw[thick, red,style={decorate,decoration={snake,amplitude=1}}] (1.4,0) -- (1.5,2);
\draw[thick, red,style={decorate,decoration={snake,amplitude=1}}] (2.4,0) -- (2.5,2);
\draw[thick, red,style={decorate,decoration={snake,amplitude=1}}] (3.4,0) -- (3.5,2);
\draw[thick, red,style={decorate,decoration={snake,amplitude=1}}] (0.4,2) -- (0.5,4);
\draw[thick, red,style={decorate,decoration={snake,amplitude=1}}] (1.4,2) -- (1.5,4);
\draw[thick, red,style={decorate,decoration={snake,amplitude=1}}] (2.4,2) -- (2.5,4);
\draw[thick, red,style={decorate,decoration={snake,amplitude=1}}] (3.4,2) -- (3.5,4);
\draw[thick, red,style={decorate,decoration={snake,amplitude=1}}] (0.4,4) -- (0.5,6);
\draw[thick, red,style={decorate,decoration={snake,amplitude=1}}] (1.4,4) -- (1.5,6);
\draw[thick, red,style={decorate,decoration={snake,amplitude=1}}] (2.4,4) -- (2.5,6);
\draw[thick, red,style={decorate,decoration={snake,amplitude=1}}] (3.4,4) -- (3.5,6);
\draw[thick, red,style={decorate,decoration={snake,amplitude=1}}] (0.4,6) -- (0.5,8);
\draw[thick, red,style={decorate,decoration={snake,amplitude=1}}] (1.4,6) -- (1.5,8);
\draw[thick, red,style={decorate,decoration={snake,amplitude=1}}] (2.4,6) -- (2.5,8);
\draw[thick, red,style={decorate,decoration={snake,amplitude=1}}] (3.4,6) -- (3.5,8);
\end{tikzpicture}
\par\end{centering}
}\hfill{}
\par\end{centering}
\caption{}
\end{figure}

\begin{prop}
\label{prop:diamtail}There is a $\delta=\delta_{\diam}>0$ and $C_{\diam}<\infty$,
independent of the scale $S$, so that the following holds. If 
\begin{equation}
\CV^{2}(\Psi_{\mathrm{hard}}(\mathfrak{A}_{t}))<\delta\label{eq:smallcvdiam}
\end{equation}
for all $t\ge0$, and
\begin{equation}
\CV^{2}(\Psi_{\mathrm{easy}}(\mathfrak{A}))<\delta_{\mathrm{RSW}}\label{eq:smallcvdiam-1}
\end{equation}
for all $\mathfrak{A}\subseteq\mathfrak{R}$ of aspect ratio between
$1/2$ and $2$, inclusive, then, for any $\alpha\in\mathbf{N}$ we
have a $C(\alpha)\ge0$ so that, as long as $\gamma$ is sufficiently
small and $u$ is sufficiently large (both compared to $\alpha$),
\[
\mathbf{P}\left(\Psi_{\max}(\mathfrak{R})\ge u\Theta_{\mathrm{easy}}(\mathfrak{R})[q_{\mathrm{pl}}]\right)\le C_{}(\alpha)u^{-\alpha}.
\]
\end{prop}

\begin{proof}
Let $L\in\mathbf{N}$ be \emph{fixed} but chosen later. By our crossing
distance tail bound \eqref{longtailindep}, applied with $L=2^{l}$,
$K=2L$, and a union bound, for all $u\ge u_{0}$ we have
\begin{equation}
\mathbf{P}\left[\max_{(i,j)\in[0,2^{t})^{2}}\Psi_{\mathrm{hard}}(\mathfrak{R}_{t;i,j})\ge4uL\mathbf{E}\Psi_{\mathrm{hard}}(\mathfrak{A}_{t+l})\right]\le(1+o_{L}(1))\cdot4^{t}\cdot u^{-L/4},\label{eq:hardtail}
\end{equation}
Now we know that, if \eqref{smallcvdiam} holds and $\delta$ is sufficiently
small (compared to $p_{\mathrm{RSW}}$), then by \eqref{concentration-main},
\thmref{rsw} (noting the hypothesis \eqref{smallcvdiam-1}), and
\propref{exponential-crossings} (recalling \eqref{prswsmall}) there
is a constant $C_{1}$ (depending on $\delta$) so that we have
\begin{multline}
\mathbf{E}\Psi_{\mathrm{hard}}(\mathfrak{A}_{t+l})\le C_{1}\Theta_{\mathrm{hard}}(\mathfrak{A}_{t+l})[p_{\mathrm{RSW}}]\\
\le C_{1}C_{\mathrm{RSW}}\Theta_{\mathrm{easy}}(\mathfrak{A}_{t+l})[p_{\mathrm{RSW}}]\le C_{1}C_{\mathrm{pl}}C_{\mathrm{RSW}}a_{\mathrm{pl}}^{t+l}\Theta_{\mathrm{easy}}(\mathfrak{R})[q_{\mathrm{pl}}].\label{eq:diam-usehyps}
\end{multline}
Combining \eqref{hardtail} and \eqref{diam-usehyps} and putting
$C=C_{1}C_{\mathrm{pl}}C_{\mathrm{RSW}}$, we get
\begin{multline*}
\mathbf{P}\left[\max_{(i,j)\in[0,2^{t})^{2}}\Psi_{\mathrm{hard}}(\mathfrak{R}_{t;i,j})\ge4CuLa_{\mathrm{pl}}^{\frac{t+l}{2}}\Theta_{\mathrm{easy}}(\mathfrak{R})[q_{\mathrm{pl}}]\right]\\
\le(1+o_{L}(1))\cdot4^{t}\cdot u^{-L/4}\cdot a_{\mathrm{pl}}^{L(t+l)/8}.
\end{multline*}
Using \eqref{tailgff}, we derive
\begin{align*}
(1+\mathrlap{o_{L}(1))^{-1}\mathbf{P}\left[\max_{(i,j)\in[0,2^{t})^{2}}\Psi_{\mathrm{hard}}(\mathfrak{R}_{t;i,j};Y_{\mathfrak{R}})\ge8CuLa_{\mathrm{pl}}^{\frac{1}{4}(t+l)}\Theta_{\mathrm{easy}}(\mathfrak{R})[q_{\mathrm{pl}}]\right]}\\
 & \le4^{t}u^{-L/8}a_{\mathrm{pl}}^{L(t+l)/8}+4^{t}\exp\left(-\omega(1)\frac{\left(\log\big(ua_{\mathrm{pl}}^{-\frac{t+l}{4}}\big)\bigg)\right)^{2}}{\log4^{t}}\right)\\
 & \le u^{-L/8}\left[4^{t}a_{\mathrm{pl}}^{L(t+l)/8}+\exp\left(t\log4-\omega_{L}(1)\log u-\omega_{L}(1)\frac{(t+l)^{2}}{t}\right)\right].
\end{align*}
If we choose $L$ so large and $\gamma$ so small that the term is
brackets is summable in $t$, then we can conclude using a union bound
that
\begin{multline}
\mathbf{P}\left[(\exists t)\,\max_{(i,j)\in[0,2^{t})^{2}}\Psi_{\mathrm{hard}}(\mathfrak{R}_{t;i,j};Y_{\mathfrak{R}})\ge8CuLa_{\mathrm{pl}}^{\frac{1}{4}(t+l)}\Theta_{\mathrm{easy}}(\mathfrak{R})[q_{\mathrm{pl}}]\right]\\
=O_{L}(1)u^{-L/8}.\label{eq:maxleveldiam}
\end{multline}

Now for $x\in\mathfrak{R}$ and $t\in[0,s)$, let $\mathfrak{R}_{t}(x)$
be the $\mathfrak{R}_{t;i,j}$ containing $x$. Then
\[
\Psi_{x,y}(\mathfrak{R})\le\sum_{t\in[0,s]}\Psi_{\mathrm{hard}}(\mathfrak{R}_{t}(x);Y_{\mathfrak{R}})+\sum_{t\in[1,s]}\Psi_{\mathrm{hard}}(\mathfrak{R}_{t}(y);Y_{\mathfrak{R}}).
\]
(See \figref{pointtopoint}.) This means that
\begin{align*}
\Psi_{\max}(\mathfrak{R}) & \le2\sum_{t\in[0,s]}\max_{(i,j)\in[0,2^{t})^{2}}\Psi_{\mathrm{hard}}(\mathfrak{R}_{t;i,j};Y_{\mathfrak{R}});
\end{align*}
this is the chaining argument illustrated in \figref{chaining}. Applying
\eqref{maxleveldiam}, this implies
\[
\mathbf{P}\left[\Psi_{\max}(\mathfrak{R})\ge8CuL\Theta_{\mathrm{easy}}(\mathfrak{R})[q_{\mathrm{pl}}]\sum_{t=0}^{s}a_{\mathrm{pl}}^{\frac{1}{4}(t+l)}\right]\le O_{L}(1)u^{-L/8}.
\]
The sum is bounded so we obtain
\[
\mathbf{P}\left[\Psi_{\max}(\mathfrak{R})\ge u\Theta_{\mathrm{easy}}(\mathfrak{R})[q_{\mathrm{pl}}]\right]\le O_{L}(1)u^{-L/8},
\]
and the result follows since $L$ can be chosen to be arbitrarily
large.
\end{proof}

\section{Variation upper bounds\label{sec:variation}}

In this section we prove an inductive upper bound on the variance
of the crossing distance in a rectangle, which we combine with our
lower bounds on the expectation of the crossing distance in order
to prove \thmref{maintheorem}.

\subsection{Variance of the crossing distance}

Our goal in this section is to prove the following bound on the variance
of the left-right crossing distance of a rectangle. Let $\mathfrak{R}=[0,KS)\times[0,LS)$.
\begin{thm}
\label{thm:varbound-1}For any $\beta>0$, there is a $\delta=\delta_{\Var}>0$
and a constant $C_{\Var}<\infty$ so that if $S,K,L$ are sufficiently
large and $\gamma$ is sufficiently small (independent of the scale
$S$), and $\CV^{2}(\Psi_{\mathrm{easy}}(\mathfrak{A}))<\delta$ whenever
$\mathfrak{A}\subseteq[0,3S)^{2}$ has aspect ratio between $1/2$
and $2$ inclusive, then
\begin{multline}
(1-o_{K,L}(1))\Var(\Psi_{\mathrm{LR}}(\mathfrak{R}))-o_{K,L}(1)\left(\mathbf{E}\Psi_{\mathrm{LR}}(\mathfrak{R})\right)^{2}\\
\le C_{\Var}KL^{2/\beta}\left(\mathbf{E}\Psi_{\mathrm{easy}}([0,3S)^{2})\right)^{2}.\label{eq:varbound-1}
\end{multline}
\end{thm}

The proof of \thmref{varbound-1} will be be based on the following
standard Efron–Stein inequality\cite{Ste86}, which we quote here
for reference.
\begin{thm}[Efron–Stein]
\label{thm:efronstein}Let $X_{1},\ldots,X_{r},X_{1}',\ldots,X_{r}'$
be independent random variables so that $X_{j}$ and $X_{j}'$ are
identically distributed for each $j$, and $f:\mathbf{R}^{r}\to\mathbf{R}$.
Then 
\begin{multline*}
\Var(f(X_{1},\ldots,X_{r}))\\
\le\frac{1}{2}\sum_{j=1}^{r}\mathbf{E}\left(f(X_{1},\ldots,X_{r})-f(X_{1},\ldots,X_{j-1},X_{j}',X_{j+1},\ldots,X_{r})\right)^{2}.
\end{multline*}
\end{thm}

To apply \nameref{thm:efronstein}, we need a way to write our field
as a function of many independent variables, each of which has only
a small effect on the weight of a crossing. We can divide $\overline{\mathfrak{R}}$
into $9KL$ disjoint dyadic $S\times S$ sub-boxes, which we will
label $\mathfrak{\mathfrak{C}}_{1},\ldots,\mathfrak{\mathfrak{C}}_{9KL}$
in arbitrary order. Write $Y_{\mathfrak{R}}$ as a function of independent
random variables $Z_{1},\ldots,Z_{9KL}$ as in \criref{localrandomness}.
For $i=1,\ldots,9KL$, write $Y^{\mathfrak{C}_{i}}$ for the field
$Y$ with $Z_{i}$ resampled. \thmref{efronstein} implies that
\begin{equation}
\Var(\Psi_{\mathrm{LR}}(\mathfrak{R}))\le\frac{1}{2}\sum_{i=1}^{9KL}\mathbf{E}[\Psi_{\mathrm{LR}}(\mathfrak{R})-\Psi_{\mathrm{LR}}(\mathfrak{R};Y^{\mathfrak{\mathfrak{C}}_{i}})]^{2}.\label{eq:efronsteinapp-1}
\end{equation}
We will bound the terms on the right-hand side of \eqref{efronsteinapp-1}
individually in \lemref{resample-gff}. But first we need the following
lemma about the effect of resampling a box on the field distant from
that box. Let $\mathfrak{D}_{i}=\overline{\mathfrak{C}_{i}}\cap\mathfrak{R}$
and $\Xi(x,i)=Y^{\mathfrak{C}_{i}}(x)-Y(x)$.
\begin{lem}
\label{lem:ximaxbound}Define $\Xi^{*}(i)=\sup_{x\in\mathfrak{B}\setminus\mathfrak{D}_{i}}\Xi(x,i).$
Then for any $a\ge1$, there is a constant $C_{a}$ (as always, independent
of the scale) so that $\mathbf{E}|\Xi^{*}(i)|^{a}\le C_{a}.$
\end{lem}

\begin{proof}
The Borell–TIS inequality (see, for example, \cite[Theorem 7.1]{L01},
\cite[Theorem 6.1]{biskup} or \cite[Theorem 2.1]{A90}), applied
in light of \eqref{resamplevar}, tells us that there is a constant
$C$ so that
\[
\mathbf{P}(\left|\Xi^{*}(i)-\mathbf{E}\Xi^{*}(i)\right|\ge u)\le2e^{-\frac{u^{2}}{2C}}.
\]
Thus we are done as long as we can bound $\Xi^{*}(i)$ by a constant
independent of the scale. We do this using Fernique's inequality.
By \eqref{resamplediffvar} we have a constant $C$ so that
\[
\Var(\Xi(x,i)-\Xi(y,i))\le\frac{CS^{2}}{\left[(K\wedge L)S\right]^{4}}\|x-y\|^{2}=\frac{C\|x-y\|^{2}}{(K\wedge L)^{4}S^{2}}.
\]
Therefore, for a typical point $x$, we have by Fernique's inequality
(\cite{Fer75}, \cite[Theorem 4.1]{A90}, or \cite[Theorem 6.6]{biskup},
as applied in \cite[Lemma 3.5]{BDZ14}) that there exists a constant
$C'$, independent of $S$, so that $\mathbf{E}\Xi^{*}(i)\le C'$.
\end{proof}
\begin{lem}
\label{lem:resample-gff}For each $i$, let $E_{i}$ be the event
that $\pi_{\mathrm{LR}}(\mathfrak{R})\cap\mathfrak{D}_{i}\ne\emptyset.$
Then we have
\begin{multline}
\mathbf{E}\left[\Psi_{\mathrm{LR}}(\mathfrak{R};Y^{\mathfrak{\mathfrak{C}}_{i}})-\Psi_{\mathrm{LR}}(\mathfrak{R})\right]^{2}\le4\mathbf{E}\left(\Psi_{\partial}(\mathfrak{D}_{i};Y^{\mathfrak{C}_{i}})\mathbf{1}_{E_{i}}\right)^{2}\\
+o_{K,L}(1)\mathbf{E}\Psi_{\mathrm{hard}}([0,S)\times[0,2S))^{2}.\label{eq:resample-gff-result}
\end{multline}
\end{lem}

\begin{proof}
To begin, note that since $Y$ and $Y^{\mathfrak{\mathfrak{C}}_{i}}$
are exchangeable, we have 
\begin{equation}
\mathbf{E}[\Psi_{\mathrm{LR}}(\mathfrak{R})-\Psi_{\mathrm{LR}}(\mathfrak{R};Y^{\mathfrak{\mathfrak{C}}_{i}})]^{2}=2\mathbf{E}[0\vee(\Psi_{\mathrm{LR}}(\mathfrak{R})-\Psi_{\mathrm{LR}}(\mathfrak{R};Y^{\mathfrak{\mathfrak{C}}_{i}}))]^{2}.\label{eq:exchangeable}
\end{equation}
Let $\pi=\pi_{\text{\ensuremath{\mathrm{LR}}}}(\mathfrak{R})$. On
the occurrence of $E_{i}$, put $\pi=\pi_{0}\cup\pi_{1}$, where $\pi_{0}$
is the part of $\pi$ between the first time $\pi$ enters $\mathfrak{D}_{i}$
and the last time $\pi$ exits $\mathfrak{D}_{i}$, and $\pi_{1}$
is the (generally non-contiguous) set of all other vertices of $\pi$.

Note that
\[
\Psi_{\mathrm{LR}}(\mathfrak{R};Y^{\mathfrak{\mathfrak{C}}_{i}})=\inf_{\pi'}\sum_{x\in\pi'}\exp(\gamma Y^{\mathfrak{C}_{i}}(x)),
\]
where $\pi'$ ranges over all left–right crossings of $\mathfrak{R}$.
We claim that 
\begin{equation}
\Psi_{\mathrm{LR}}(\mathfrak{R};Y^{\mathfrak{\mathfrak{C}}_{i}})-\Psi_{\mathrm{LR}}(\mathfrak{R})\le\sum_{x\in\pi_{1}}e^{\gamma Y(x)}[e^{\gamma\Xi(x,i)}-1]+\Psi_{\partial}(\mathfrak{D}_{i};Y^{\mathfrak{C}_{i}})\mathbf{1}_{E_{i}}.\label{eq:resample1-1}
\end{equation}
We prove \eqref{resample1-1} by considering separately the situations
in which $E_{i}$ does and does not occur.

\emph{Case 1. }On the event $E_{i}$, we have
\begin{align*}
\Psi_{\mathrm{LR}}(\mathfrak{R};Y^{\mathfrak{\mathfrak{C}}_{i}}) & =\inf_{\pi'}\sum_{x\in\pi'}\exp(\gamma Y^{\mathfrak{C}_{i}}(x))\le\psi(\pi_{0};Y^{\mathfrak{C}_{i}})+\Psi_{x^{*},y^{*}}(\mathfrak{D}_{i};Y^{\mathfrak{C}_{i}}),
\end{align*}
where $\pi'$ ranges over all left–right crossings of $\mathfrak{R}$
and $x^{*}$ and $y^{*}$ are the first and last vertices of $\pi_{1}$,
respectively. Therefore,
\begin{align*}
\Psi_{\mathrm{LR}}(\mathfrak{R};Y^{\mathfrak{\mathfrak{C}}_{i}})-\mathrlap{\Psi_{\mathrm{LR}}(\mathfrak{R};Y)}\\
 & \le\psi(\pi_{1};Y^{\mathfrak{C}_{i}})+\Psi_{x^{*},y^{*}}(\mathfrak{D}_{i};Y^{\mathfrak{C}_{i}})-\psi(\pi_{1};Y)-\psi(\pi_{0};Y)\\
 & \le\psi(\pi_{1};Y^{\mathfrak{C}_{i}})+\Psi_{\partial}(\mathfrak{D}_{i};Y^{\mathfrak{C}_{i}})-\psi(\pi_{1};Y)\\
 & \le\Psi_{\partial}(\mathfrak{D}_{i};Y^{\mathfrak{C}_{i}})+\sum_{x\in\pi_{1}}e^{\gamma Y(x)}[1-e^{\gamma\Xi(x,i)}].
\end{align*}

\emph{Case 2. }If $E_{i}$ does not occur, then we note that since
$\pi$ is a path not passing through $\mathfrak{D}_{i}$,
\begin{align*}
\psi(\pi;Y)-\psi(\pi;Y^{\mathfrak{C}_{i}})=\sum_{x\in\pi}[e^{\gamma Y(x)}-e^{\gamma Y^{\mathfrak{C}_{i}}(x)}] & =\sum_{x\in\pi}[e^{\gamma Y(x)}-e^{\gamma Y^{\mathfrak{C}_{i}}(x)}]\\
 & =\sum_{x\in\pi}e^{\gamma Y(x)}[1-e^{\gamma\Xi(x,i)}],
\end{align*}
so we can write
\[
\inf_{\pi'}\sum_{x\in\pi'}\exp(\gamma Y^{\mathfrak{C}_{i}}(x))\le\sum_{x\in\pi}\exp(\gamma Y^{\mathfrak{C}_{i}}(x))=\Psi_{\mathrm{LR}}(\mathfrak{R})+\sum_{x\in\pi}e^{\gamma Y(x)}[1-e^{\gamma\Xi(x,i)}].
\]

The two cases together imply \eqref{resample1-1}. Now, combining
\eqref{exchangeable} and \eqref{resample1-1}, we have
\begin{align*}
\mathbf{E}[\mathrlap{[\Psi_{\mathrm{LR}}(\mathfrak{R})-\Psi_{\mathrm{LR}}(\mathfrak{R};Y^{\mathfrak{\mathfrak{C}}_{i}})]\vee0]^{2}}\\
 & \le\mathbf{E}\left[\sum_{x\in\pi_{1}}e^{\gamma Y(x)}[[e^{\gamma\Xi(x,i)}-1]\vee0]+\Psi_{\partial}(\mathfrak{D}_{i};Y^{\mathfrak{C}_{i}})\mathbf{1}_{E_{i}}\right]^{2}\\
 & \le2\mathbf{E}\left(\sum_{x\in\pi_{1}}e^{\gamma Y(x)}[[e^{\gamma\Xi(x,i)}-1]\vee0]\right)^{2}+2\mathbf{E}\left(\Psi_{\partial}(\mathfrak{D}_{i};Y^{\mathfrak{C}_{i}})^{2}\mathbf{1}_{E_{i}}\right).
\end{align*}
Considering the first term further, we have
\[
\sum_{x\in\pi_{1}}e^{\gamma Y(x)}[[e^{\gamma\Xi(x,i)}-1]\vee0]\le\Psi_{\mathrm{LR}}(\mathfrak{R})\cdot\sup_{x\in\mathfrak{R}\setminus\overline{\mathfrak{C}_{i}}}[[e^{\gamma\Xi_{i}^{*}}-1]\vee0],
\]
where $\Xi_{i}^{*}=\sup\limits _{x\in\mathfrak{R}\setminus\overline{\mathfrak{C}_{i}}}\Xi(x,i)$.
By Hölder's inequality, we have
\begin{align*}
\mathbf{E}\mathrlap{\left[\Psi_{\mathrm{LR}}(\mathfrak{R})\cdot\sup_{x\in\mathfrak{R}\setminus\overline{\mathfrak{C}_{i}}}[[e^{\gamma\Xi_{i}^{*}}-1]\vee0]\right]^{2}}\\
 & \le\left(\mathbf{E}\sup_{x\in\mathfrak{R}\setminus\overline{\mathfrak{C}_{i}}}[[e^{\gamma\Xi_{i}^{*}}-1]^{3/2}\vee0]\right)^{4/3}\left(\mathbf{E}\Psi_{\mathrm{LR}}(\mathfrak{R})^{3}\right)^{2/3}\\
 & \le o_{K,L}(1)\mathbf{E}\Psi_{\mathrm{hard}}([0,S)\times[0,2S))^{2}.
\end{align*}
with the second inequality by \eqref{crossmomentapriori} and \lemref{ximaxbound}.
Then \eqref{resample-gff-result} follows.
\end{proof}
Now we can prove our variance bound.
\begin{proof}[Proof of \thmref{varbound-1}]
Let $q'\in(q_{\mathrm{pl}},1)$. Note that we can split the event
$\Psi_{\partial}(\mathfrak{D}_{i};Y^{\mathfrak{C}_{i}})^{2}\mathbf{1}_{E_{i}}$
into cases as follows:
\begin{align}
\Psi_{\partial}(\mathrlap{\mathfrak{D}_{i};Y^{\mathfrak{C}_{i}})^{2}\mathbf{1}_{E_{i}}}\label{eq:diamtailapp1}\\
 & =\Psi_{\partial}(\mathfrak{D}_{i};Y^{\mathfrak{C}_{i}})^{2}\mathbf{1}_{E_{i}}\mathbf{1}\{\Psi_{\partial}(\mathfrak{D}_{i};Y^{\mathfrak{C}_{i}})\ge u\Theta_{\mathrm{easy}}(\mathfrak{D}_{i};Y^{\mathfrak{C}_{i}})[q']\}\nonumber \\
 & \qquad+\Psi_{\partial}(\mathfrak{D}_{i};Y^{\mathfrak{C}_{i}})^{2}\mathbf{1}_{E_{i}}\mathbf{1}\{\Psi_{\partial}(\mathfrak{D}_{i};Y^{\mathfrak{C}_{i}})<u\Theta_{\mathrm{easy}}(\mathfrak{D}_{i};Y^{\mathfrak{C}_{i}})[q']\}\nonumber \\
 & \le\Psi_{\partial}(\mathfrak{D}_{i};Y^{\mathfrak{C}_{i}})^{2}\mathbf{1}\{\Psi_{\partial}(\mathfrak{D}_{i};Y^{\mathfrak{C}_{i}})\ge u\Theta_{\mathrm{easy}}(\mathfrak{D}_{i};Y^{\mathfrak{C}_{i}})[q']\}\nonumber \\
 & \qquad+u^{2}\Theta_{\mathrm{easy}}(\mathfrak{D}_{i};Y^{\mathfrak{C}_{i}})[q']^{2}\mathbf{1}_{E_{i}}.\nonumber 
\end{align}
Moreover, we have by \eqref{relatescale} and \propref{diamtail},
as long as $u$ is sufficiently large,
\begin{align}
\mathbf{E}\mathrlap{\left[\Psi_{\partial}(\mathfrak{D}_{i};Y^{\mathfrak{C}_{i}})^{2}\mathbf{1}\{\Psi_{\partial}(\mathfrak{D}_{i};Y^{\mathfrak{C}_{i}})\ge u\Theta_{\mathrm{easy}}(\mathfrak{D}_{i};Y^{\mathfrak{C}_{i}})[q']\}\right]}\label{eq:diamtailapp2}\\
 & \le(1+o(1))\mathbf{E}\left[\Psi_{\partial}(\mathfrak{D}_{i})^{2}\mathbf{1}\{\Psi_{\partial}(\mathfrak{D}_{i})\ge\tfrac{1}{2}u\Theta_{\mathrm{easy}}(\mathfrak{D}_{i})[q_{\mathrm{pl}}]\}\right]\nonumber \\
 & \le O_{\alpha}(1)\cdot\Theta_{\mathrm{easy}}(\mathfrak{D}_{i})[q_{\mathrm{pl}}]^{2}\cdot\int_{u/2}^{\infty}v^{2-\alpha}\,dv\nonumber \\
 & =O_{\alpha}(1)\cdot\Theta_{\mathrm{easy}}(\mathfrak{D}_{i})[q_{\mathrm{pl}}]^{2}\cdot u^{3-\alpha}.\nonumber 
\end{align}
Also, by \eqref{concentration-quantiles}, as long as $\delta$ is
sufficiently small we have
\begin{equation}
\Theta_{\mathrm{easy}}(\mathfrak{D}_{i};Y^{\mathfrak{C}_{i}})[q']^{2}\le O(1)\cdot\Theta_{\mathrm{easy}}(\mathfrak{D}_{i};Y^{\mathfrak{C}_{i}})[q_{\mathrm{pl}}]^{2}.\label{eq:diamtailapp3}
\end{equation}
Combining \lemref{resample-gff}, \eqref{diamtailapp1}, \eqref{diamtailapp2},
\eqref{diamtailapp3}, and \propref{crossinglengthexp}, and assuming
that $K$ and $L$ are sufficiently large and $\delta,\gamma$ sufficiently
small, we have
\begin{align*}
\frac{1}{2}\mathrlap{\sum_{i=1}^{9KL}\mathbf{E}[\Psi_{\mathrm{LR}}(\mathfrak{R};Y)-\Psi_{\mathrm{LR}}(\mathfrak{R};Y^{\mathfrak{\mathfrak{C}}_{i}})]^{2}}\\
 & \le\frac{1}{2}\sum_{i=1}^{9KL}4\left(\mathbf{E}\left(\Psi_{\partial}(\mathfrak{D}_{i};Y^{\mathfrak{C}_{i}})^{2}\mathbf{1}_{E_{i}}\right)+o_{K,L}(1)\mathbf{E}\Psi_{\mathrm{LR}}(\mathfrak{R})^{2}\right)\\
 & \le\sum_{i=1}^{9KL}O_{\alpha}(1)\Theta_{\mathrm{easy}}(\mathfrak{D}_{i})[q_{\mathrm{pl}}]^{2}u^{3-\alpha}+\frac{1}{2}u^{2}\Theta_{\mathrm{easy}}([0,3S)^{2})[q_{\mathrm{pl}}]^{2}\mathbf{E}M_{\mathrm{LR};S}(\mathfrak{R})\\
 & \qquad\qquad+o_{K,L}(1)\mathbf{E}\Psi_{\mathrm{hard}}([0,S)\times[0,2S))^{2}\\
 & \le O_{\beta}(1)KL\Theta_{\mathrm{easy}}([0,3S)^{2})[q_{\mathrm{pl}}]^{2}u^{-\beta}+C_{\mathrm{CL}}Ku^{2}\Theta_{\mathrm{easy}}([0,3S)^{2})[q_{\mathrm{pl}}]^{2}\\
 & \qquad\qquad+o_{K,L}(1)\mathbf{E}\Psi_{\mathrm{hard}}([0,S)\times[0,2S))^{2},
\end{align*}
where $\beta=\alpha-3$. Then if we put $u=L^{1/\beta}$, then we
obtain
\begin{align*}
\Var\Psi_{\mathrm{LR}}(\mathfrak{R}) & \le\frac{1}{2}\sum_{i=1}^{9KL}\mathbf{E}[\Psi_{\mathrm{LR}}(\mathfrak{R})-\Psi_{\mathrm{LR}}(\mathfrak{R};Y^{\mathfrak{\mathfrak{C}}_{i}})]^{2}\\
 & \le K\Theta_{\mathrm{easy}}([0,3S)^{2})[q_{\mathrm{pl}}]^{2}[C_{\mathrm{CL}}L^{2/\beta}+O_{\beta}(1)]\\
 & \qquad\qquad+o_{K,L}(1)\mathbf{E}\Psi_{\mathrm{hard}}([0,S)\times[0,2S))^{2}.
\end{align*}
Then \eqref{varbound-1} follows from another application of \eqref{concentration-main},
along with the hypothesis on the coefficient of variation and \thmref{rsw}
and \lemref{relatequantiles} to bound the last term in the last equation.
\end{proof}

\subsection{Coefficient of variation}

Armed with our inductive upper bound on crossing distance variance
from the previous subsection, and inductive lower bound on expected
crossing distance from \secref{lowerbounds}, we are now ready to
work towards a proof of \thmref{maintheorem} by induction.
\begin{lem}
\label{lem:inductivestep}There is a $\delta_{0}>0$ so that if $0<\delta<\delta_{0}$
then the following holds. Fix a scale $S=2^{s}$. Suppose that
\[
\CV^{2}(\Psi_{\mathrm{LR}}(\mathfrak{A}))<\delta
\]
for all $\mathfrak{A}\subseteq[0,S)\times[0,2S)$ of aspect ratio
between $1/2$ and $2$, inclusive. If $K$ is sufficiently large
compared to $\delta$ and $K/2\le L\le2K$ and $\gamma$ is sufficiently
small compared to $\delta$, $K$, and $L$, then if $\mathfrak{R}=[0,KS)\times[0,LS)$,
we have
\[
\CV^{2}(\Psi_{\mathrm{LR}}(\mathfrak{R}))<\delta.
\]
\end{lem}

\begin{proof}
By \thmref{varbound-1}, if $K$ and $L$ are sufficiently large,
we have
\begin{multline*}
(1-o_{K,L}(1))\cdot\Var(\Psi_{\mathrm{LR}}(\mathfrak{R}))-o_{K,L}(1)\left(\mathbf{E}\Psi_{\mathrm{LR}}(\mathfrak{R})\right)^{2}\\
\le C_{\Var}\cdot K\cdot L^{2/\beta}\cdot\left(\mathbf{E}\Psi_{\mathrm{easy}}([0,3S)^{2};Y^{\mathfrak{C}_{i}})\right)^{2}.
\end{multline*}
Moreover, by \corrref{quantile-easy-relation}, we have
\[
\mathbf{E}\Psi_{\mathrm{LR}}(\mathfrak{R})\ge\frac{K}{2u_{0}}\Theta_{\mathrm{easy}}(\mathfrak{A})[p_{\mathrm{RSW}}]\cdot\left(1-C_{\mathrm{p}}L\left(2d_{\mathrm{p}}^{2}\sqrt{p_{\mathrm{RSW}}}\right)^{K}-o_{K,L}(1)\right),
\]
so (again recalling \eqref{prswsmall}) if $K$ and $L$ are sufficiently
large and $\gamma$ is sufficiently small then we have
\[
\mathbf{E}\Psi_{\mathrm{LR}}(\mathfrak{R})\ge\frac{K}{4u_{0}}\Theta_{\mathrm{easy}}(\mathfrak{A})[p_{\mathrm{RSW}}].
\]
Therefore, we have a constant $C$ so that
\begin{align*}
\CV^{2}(\Psi_{\mathrm{LR}}(\mathfrak{R})) & =\frac{\Var(\Psi_{\mathrm{LR}}(\mathfrak{R}))}{\left(\mathbf{E}\Psi_{\mathrm{LR}}(\mathfrak{R})\right)^{2}}\\
 & \le\frac{C_{\Var}KL^{2/\beta}\left(\mathbf{E}\Psi_{\mathrm{easy}}(\mathfrak{A})\right)^{2}}{(1-o_{K,L}(1))\cdot K^{2}\left(\Theta_{\mathrm{easy}}(\mathfrak{A})[p_{\mathrm{RSW}}]\right)}+o_{K,L}(1)\\
 & \le\frac{CL^{2/\beta}}{K}+o_{K,L}(1).
\end{align*}
If we choose $K$ sufficiently large compared to $\delta$, and $\beta$
sufficiently large, then this yields $\CV^{2}(\Psi_{\mathrm{LR}}(\mathfrak{R}))<\delta$
for all $K/2\le L\le2K$.
\end{proof}
\begin{lem}
\label{lem:basecase}For a \emph{fixed} scale $S_{0}$, we have $\CV^{2}(\Psi_{\mathrm{LR}}(\mathfrak{A}))=o_{S_{0}}(1)$
for all $\mathfrak{A}\subset[0,S_{0})\times[0,2S_{0}).$
\end{lem}

\begin{proof}
Without loss of generality, let $\mathfrak{A}=[0,S]\times[0,T]$.
We note that $\Psi_{\mathrm{LR}}(\mathfrak{A})\le\psi(\pi_{0};Y_{\mathfrak{A}}),$
where $\pi_{0}$ is a straight-line path across $\mathfrak{A}$. Therefore,
\[
\mathbf{E}\Psi_{\mathrm{LR}}(\mathfrak{A})^{2}\le\mathbf{E}\psi(\pi_{0};Y_{\mathfrak{A}})^{2}=S^{2}+o_{S_{0}}(1).
\]
On the other hand,
\[
\Psi_{\mathrm{LR}}(\mathfrak{A})\ge S\min_{x\in\mathfrak{A}}\exp(\gamma Y(x)),
\]
so
\[
\mathbf{E}\Psi_{\mathrm{LR}}(\mathfrak{A})\ge S\mathbf{E}\left[\min_{x\in\mathfrak{A}}\exp(\gamma Y(x))\right]=S+o_{S_{0}}(1).
\]
Therefore,
\[
\CV^{2}\left(\Psi_{\mathrm{LR}}(\mathfrak{A})\right)\le\frac{\mathbf{E}\Psi_{\mathrm{LR}}(\mathfrak{A})^{2}-\left(\mathbf{E}\Psi_{\mathrm{LR}}(\mathfrak{A})\right)^{2}}{\left(\mathbf{E}\Psi_{\mathrm{LR}}(\mathfrak{A})\right)^{2}}=o_{S_{0}}(1).\qedhere
\]
\end{proof}
We have now assembled all of the pieces necessary for the proof of
\thmref{maintheorem}.
\begin{proof}[Proof of \thmref{maintheorem}]
Apply \lemref{basecase} for some $S_{0}>K$, with $K$ chosen large
enough compared to $\delta$ to satisfy the assumptions of \lemref{inductivestep}.
Then inductively applying \lemref{inductivestep} allows us to bound
the coefficient of variation of every box of the given aspect ratios.
\end{proof}

\section{Subsequential limits of FPP metrics\label{sec:limits}}

All of the necessary estimates in hand, we now proceed to establish
existence and continuity properties of the scaling limit metrics of
Liouville FPP.

\subsection{Tightness and subsequential convergence}

As a corollary of \thmref{maintheorem}, we will derive a tightness
result for the first-passage percolation metric, properly scaled.

For each $S=2^{s}$, let $\mathfrak{R}_{s}=[0,S)^{2}$. For $x,y\in[0,1]_{\mathbf{R}}^{2}\cap\frac{1}{2^{s}}\mathbf{Z}^{2}$,
let 
\[
d_{s}(x,y)=\frac{\Psi_{Sx,Sy}(\mathfrak{R}_{s})}{\Theta_{\mathrm{easy}}(\mathfrak{R}_{s})[q_{\mathrm{pl}}]}.
\]
For arbitrary $x,y\in[0,1]_{\mathbf{R}}^{2}$, define $d_{s}(x,y)$
by linear interpolation, namely (as in (4.2) of \cite{miermont})
\begin{align}
d_{s}(x,y) & =(\lceil Sx\rceil-Sx)(\lceil Sy\rceil-Sy)d_{s}(\tfrac{1}{S}\lfloor Sx\rfloor,\tfrac{1}{S}\lfloor Sy\rfloor)\nonumber \\
 & \qquad+(\lceil Sx\rceil-Sx)(Sy-\lfloor Sy\rfloor)d_{s}(\tfrac{1}{S}\lfloor Sx\rfloor,\tfrac{1}{S}\lceil Sy\rceil)\nonumber \\
 & \qquad+(Sy-\lfloor Sx\rfloor)(\lceil Sy\rceil-Sy)d_{s}(\tfrac{1}{S}\lceil Sx\rceil,\tfrac{1}{S}\lfloor Sy\rfloor)\nonumber \\
 & \qquad+(Sy-\lfloor Sx\rfloor)(Sy-\lfloor Sy\rfloor)d_{s}(\tfrac{1}{S}\lceil Sx\rceil,\tfrac{1}{S}\lceil Sy\rceil).\label{eq:linear-interpolation}
\end{align}

\begin{thm}
\label{thm:metrictightness}If $\gamma$ is sufficiently small, then
the sequence $\{d_{s}\}_{s\in\mathbf{N}}$ is tight in the Gromov–Hausdorff
topology.
\end{thm}

Note that the first part of \thmref{subseqconv} follows from \thmref{metrictightness}
by Prokhorov's theorem.
\begin{prop}
\label{prop:holdersquares}There exists $\xi>0$ so that, if $\gamma$
is sufficiently small then for any $\varepsilon>0$, there exists
$C(\varepsilon)>0$ such that, for each $S=2^{s}$, the probability
is at most $\varepsilon$ that there exists a dyadic square $\mathfrak{C}\subset[0,1]_{\mathbf{R}}^{2}$
such that $\diam_{d_{s}}(\mathfrak{C}\cap\frac{1}{S}\mathbf{Z}^{2})\ge C(\varepsilon)(\diam_{\|\cdot\|_{\infty}}\mathfrak{C})^{\xi}$,
where $\|\cdot\|_{\infty}$ denotes the max norm.
\end{prop}

\begin{proof}
Let $\mathfrak{B}=[0,S)^{2}$ and let $\mathfrak{C}$ be a dyadic
$T\times T$ square contained in $\mathfrak{B}$ where $T=2^{t}$.
By \propref{diamtail}, as long as $\delta$ is sufficiently small
(and $\gamma$ is chosen small enough, in particular so that \thmref{maintheorem}
holds for $\delta$) we have a $C$ (independent of the scale) so
that
\[
\mathbf{P}\left(\Psi_{\max}(\mathfrak{C})\ge u\Theta_{\mathrm{easy}}(\mathfrak{C})[q_{\mathrm{pl}}]\right)\le Cu^{-\alpha}.
\]
for any dyadic square $\mathfrak{C}\subset\mathfrak{B}$. This means
that, using \propref{exponential-crossings} and \eqref{concentration-quantiles},
we have
\begin{align*}
\mathbf{P}\mathrlap{\left(\Psi_{\max}(\mathfrak{C})\ge u\Theta_{\mathrm{easy}}(\mathfrak{B})[q_{\mathrm{pl}}]\right)}\\
 & =\mathbf{P}\left(\Psi_{\max}(\mathfrak{C})\ge u\frac{\Theta_{\mathrm{easy}}(\mathfrak{B})[q_{\mathrm{pl}}]}{\Theta_{\mathrm{easy}}(\mathfrak{C})[q_{\mathrm{pl}}]}\Theta_{\mathrm{easy}}(\mathfrak{C})[q_{\mathrm{pl}}]\right)\\
 & \le C'_{\alpha}u^{-\alpha}a_{\mathrm{pl}}^{\alpha(s-t)}.
\end{align*}
(Recall from \propref{exponential-crossings} that $a_{\mathrm{pl}}\in(0,1)$,
so the right-hand is a \emph{decreasing} function of $s-t$.) Putting
\[
u=va_{\mathrm{pl}}^{\beta(s-t)}
\]
for some $\beta\in(0,1)$ to be chosen, this yields 
\[
\mathbf{P}\left(\Psi_{\max}(\mathfrak{C})\ge va_{\mathrm{pl}}^{\beta(s-t)}\Theta_{\mathrm{easy}}(\mathfrak{B})[q_{\mathrm{pl}}]\right)\le C'_{\alpha}v^{-\alpha}a_{\mathrm{pl}}^{\alpha(1-\beta)(s-t)}.
\]
Moreover, we have, for $0<\beta'<\beta$, (using \eqref{tailgff})
\begin{align*}
\mathbf{P}\mathrlap{\left(\Psi_{\max}(\mathfrak{C};Y_{\mathfrak{B}})\ge va_{\mathrm{pl}}^{\beta'(s-t)}\Theta_{\mathrm{easy}}(\mathfrak{B})[q_{\mathrm{pl}}]\right)}\\
 & \le\mathbf{P}\left(\Psi_{\max}(\mathfrak{C})\ge\sqrt{v}a_{\mathrm{pl}}^{\beta(s-t)}\Theta_{\mathrm{easy}}(\mathfrak{B})[q_{\mathrm{pl}}]\right)\\
 & \qquad\qquad+\exp\left(-\omega(1)\cdot\frac{\left(\log(\sqrt{v}\cdot a_{\mathrm{pl}}^{(\beta-\beta')(s-t)})\right)^{2}}{s-t}\right)\\
 & \le C'_{\alpha}v^{-\alpha/2}a_{\mathrm{pl}}^{\alpha(1-\beta)(s-t)}+\exp\left(-\omega(1)\cdot\left[(\beta-\beta')^{2}(s-t)+\log v\right]\right).
\end{align*}

Therefore, using a union bound, the probability that there exists
a dyadic square $\mathfrak{C}\subset\mathfrak{B}$ such that $\Psi_{\max}(\mathfrak{C};X)\ge va_{\mathrm{pl}}^{\beta'(s-t)}\Theta_{\mathrm{easy}}(\mathfrak{B})[q_{\mathrm{pl}}]$
is bounded by
\[
C'_{\alpha}v^{-\alpha/2}\sum_{t=0}^{s}4^{s-t}\left(a_{\mathrm{pl}}^{\alpha(1-\beta)(s-t)}+\exp\left(-\omega(1)\left[(\beta-\beta')^{2}(s-t)+\log v\right]\right)\right).
\]
If we choose $\alpha$ large enough and $\gamma$ small enough (but
both fixed), then the sum on the right is bounded in $s$, and so
the right-hand side can be made arbitrarily small, uniformly in $s$,
by increasing $v$. Now note that
\[
a_{\mathrm{pl}}^{\beta'(s-t)}=e^{-\beta'\log_{2}(T/S)\log a_{\mathrm{pl}}}=e^{-\beta'\log(T/S)\log_{2}a_{\mathrm{pl}}}=(T/S)^{\beta'\log_{2}(1/a_{\mathrm{pl}})}.
\]
Therefore, the probability is at most $C''_{\alpha}v^{-\alpha/2}$
that there exists a dyadic square $\mathfrak{C}\subset[0,1]_{\mathbf{R}}^{2}$,
of side length at least $1/S$, such that $\diam_{d_{s}}(\mathfrak{C}\cap\frac{1}{S}\mathbf{Z}^{2})\ge v(\diam_{\|\cdot\|_{\infty}}\mathfrak{C})^{\beta'\log_{2}(1/a_{\mathrm{pl}})}$.
Since this independent of $S$, the proof is complete (with $\xi=\beta'\log_{2}(1/a_{\mathrm{pl}})$
and $C(\varepsilon)=v$ chosen large enough so that $C''_{\alpha}v^{-\alpha/2}<\varepsilon$).
\end{proof}
\begin{cor}
\label{corr:allincompact}There exists a $\xi>0$ so that if $\gamma$
is sufficiently small then the following holds. For any $\varepsilon>0$,
there exists exists $C(\varepsilon)>0$ such that, for each $S=2^{s}$,
the probability is at most $\varepsilon$ that there exists a dyadic
square $\mathfrak{C}\subset[0,1]_{\mathbf{R}}^{2}$ such that $\diam_{d_{s}}(\mathfrak{C})\ge C(\varepsilon)(\diam_{\|\cdot\|_{\infty}}\mathfrak{C})^{\xi}$.
\end{cor}

\begin{proof}
Hölder conditions are preserved under the linear interpolation scheme~\eqref{linear-interpolation}.
\end{proof}
\begin{cor}
\label{corr:twoboxes}If $\gamma$ is sufficiently small then the
following holds. For any $\varepsilon>0$, there exists $C'(\varepsilon)>0$
such that, for each $S=2^{s}$, we have
\[
\mathbf{P}\left(\text{there exist \ensuremath{x,y\in[0,1]_{\mathbf{R}}^{2}}\text{ s.t. }}d_{s}(x,y)\ge C'(\varepsilon)\cdot\|x-y\|_{\infty}^{\xi}\right)\le\varepsilon
\]
with $\xi$ as above.
\end{cor}

\begin{proof}
Any two $x,y\in[0,1]_{\mathbf{R}}^{2}$ are contained within one or
two adjacent dyadic boxes of side length at most twice $\|x-y\|_{\infty}$.
Then the result follows from \corrref{allincompact}.
\end{proof}
We are now ready to prove our theorem.
\begin{proof}[Proof of \thmref{metrictightness}]
By \corrref{twoboxes} and the compact embedding of Hölder spaces,
for each $\varepsilon>0$ and $\xi'<\xi$ there is a compact set $A_{\varepsilon}$
in the Holder-$\xi$' topology of Hölder-$\xi$ functions on $[0,1]^{4}$
so that $\mathbf{P}(d_{s}\not\in A_{\varepsilon})<\varepsilon$. Since
the Gromov–Hausdorff topology is weaker than the uniform topology,
which is in turn weaker than the Hölder-$\xi$ topology (see for example
\cite[Proposition 3.3.2]{miermont}), $A_{\varepsilon}$ is also compact
in the Gromov–Hausdorff topology. This implies that $\{d_{s}\}$ is
tight with respect to the Gromov–Hausdorff topology.
\end{proof}

\subsection{Hölder-continuity of limiting metrics\label{subsec:holder}}

In this section we prove that $[0,1]_{\mathbf{R}}^{2}$, equipped
with the topology induced by any limit point metric, is homeomorphic
to $[0,1]_{\mathbf{R}}^{2}$ with the standard topology by a Hölder-continuous
homeomorphism with Hölder-continuous inverse. In fact, one of the
necessary maps was obtained in the coarse of the proof in the previous
section. The other direction follows from a similar chaining argument,
but using lower bounds instead of upper bounds.
\begin{prop}
\label{prop:holdercts}Any limit point of $\{d_{s}\}$ is almost surely
Hölder-$\xi'$ continuous with respect to the Euclidean metric for
any $\xi'<\xi$ as in \propref{holdersquares}.
\end{prop}

\begin{proof}
Follows from the proof of \thmref{metrictightness}.
\end{proof}
\begin{prop}
\label{prop:invholder}If $\gamma$ is sufficiently small, then there
exists a $\xi'>1$ so that for all $\varepsilon>0$ there exist $C(\varepsilon)>0$
such that for any scale $s$ we have
\[
\mathbf{P}\left(\text{there exist }(x,y)\in[0,1]_{\mathbf{R}}^{2}\text{ s.t. }d_{s}(x,y)\le\frac{1}{C(\varepsilon)}\|x-y\|_{\infty}^{\xi'}\right)\le\varepsilon.
\]
Moreover, we can take $\xi'\to1$ as $\gamma\to0$.
\end{prop}

\begin{proof}
We will use the notation $S=2^{s}$ and $T=2^{t}$ throughout. Let
$\mathfrak{R}=[0,S)^{2}$. Fix a scale $t<s$. Let $\mathfrak{A}_{t}=[0,T)\times[0,2T)$.
By \propref{lowerbound-tail}, for any $p\in(0,1/2)$ we have
\begin{multline*}
\mathbf{P}\left[\min_{|\mathcal{P}(\pi)|\ge N}\psi(\pi;Y_{\mathfrak{R}})\le\frac{N}{2u}\Theta_{\mathrm{easy}}(\mathfrak{A}_{t})[p]\right]\\
\le(S/T)^{2}\left[O(1)\left(2d_{\mathrm{p}}^{2}\sqrt{p}\right)^{N}+\exp\left(-\omega(1)\cdot\frac{(\log u)^{2}}{s-t}\right)\right],
\end{multline*}
where in the notation $\mathcal{P}(\pi)$ we consider passes of size
$2T\times T$ and $T\times2T$. Fixing $0<\beta'<\beta<1$ and summing
over all scales and putting $N=(S/T)^{\beta}v$, $u=(S/T)^{\beta'}v^{2}$,
this gives, whenever $u\ge u_{0}$,
\begin{multline*}
\mathbf{P}\left[(\exists t\in[0,s))\min_{|\mathcal{P}_{T}(\pi)|\ge\left(\frac{S}{T}\right)^{\beta}v}\psi(\pi;Y_{\mathfrak{R}})\le\tfrac{1}{2v}(\tfrac{S}{T})^{\beta-\beta'}\Theta_{\mathrm{easy}}(\mathfrak{A}_{t})[p]\right]\\
\le\sum_{t=0}^{s-1}(\tfrac{S}{T})^{2}\left[O(1)\left(2d_{\mathrm{p}}^{2}\sqrt{p}\right)^{(S/T)^{\beta}v}+e^{-\omega(1)\cdot(s-t+\log v)}\right].
\end{multline*}
As long as $v$ is large enough and $p$ and $\gamma$ are small enough,
the last sum is finite as $s\to\infty$ and goes to $0$, uniformly
in $s$, as $v\to\infty$.

By \corrref{polyupper}, \thmref{maintheorem}, and \thmref{rsw},
as long as $\gamma$ and $\delta$ are sufficiently small relative
to $p$ we have
\[
\frac{\Theta_{\mathrm{easy}}(\mathfrak{R})[p]}{\Theta_{\mathrm{easy}}(\mathfrak{A}_{t})[p]}\le C(S/T)^{1+o(1)}.
\]
Therefore, we obtain $\lim\limits _{v\to\infty}\mathbf{P}[E_{v}]=0$
uniformly in $s$, where $E_{v}$ is the event that there exists a
$t\in[0,s)$ such that
\begin{equation}
\min_{|\mathcal{P}(\pi)|\ge(S/T)^{\beta}v}\psi(\pi;Y_{\mathfrak{R}})\le\tfrac{1}{2v}\left(\tfrac{S}{T}\right)^{\beta-\beta'-1-o(1)}\Theta_{\mathrm{easy}}(\mathfrak{R})[p],\label{eq:Esprob}
\end{equation}
where again $\mathcal{P}(\pi)$ considers passes of size $2T\times T$
and $T\times2T$. Using the normalized metric $d_{s}$, we see that
$E_{v}$ contains the event that there exist a $t\in[0,s)$ and $x_{1},x_{2}\in[0,1]_{\mathbf{R}}^{2}\cap\frac{1}{S}\mathbf{Z}^{2}$
such that both
\[
\|x_{1}-x_{2}\|_{\infty}\ge c_{\mathrm{PD}}^{-1}v\left(\tfrac{T}{S}\right)^{1-\beta}
\]
and
\[
d_{s}(x_{1},x_{2})\le\tfrac{1}{2v}\left(\tfrac{T}{S}\right)^{1-(\beta-\beta')+o(1)}.
\]
This means that there are constants $C'$, $C''$ so that, with probability
going to $1$ as $v\to\infty$, for all $x_{1},x_{2}\in[0,1]_{\mathbf{R}}^{2}\cap\frac{1}{C'S}\mathbf{Z}^{2}$
we have
\begin{equation}
d_{s}(x_{1},x_{2})\ge\frac{C}{v^{2+\alpha+o(1)}}\|x_{1}-x_{2}\|_{\infty}^{1+\alpha+o(1)},\label{eq:inverseholder}
\end{equation}
where $\alpha=\beta'/(1-\beta)$. Since this property is preserved
(up to constants) by the linear interpolation, we in fact have \eqref{inverseholder}
for all $x_{1},x_{2}\in[0,1]_{\mathbf{R}}^{2}$ and all scales $s$.
By choosing $\beta,\beta'$ appropriately, we can make $\alpha$ arbitrarily
small as long as $\gamma$ is small enough. This completes the proof
of the proposition.
\end{proof}
\begin{prop}
\label{prop:invholdercts}Any limit point $d$ of $\{d_{s}\}$ almost
surely has the property that
\begin{equation}
d(x,y)\ge\frac{1}{C}\|x-y\|_{\infty}^{\xi'}\label{eq:invholderform}
\end{equation}
for some constant $\xi'\in(0,1)$ and some (random) $C$.
\end{prop}

\begin{proof}
Let 
\[
C_{s}=\sup_{x,y\in[0,1]_{\mathbf{R}}^{2}}\frac{\|x-y\|_{\infty}^{\xi'}}{d_{s}(x,y)}.
\]
By \propref{invholder}, $C_{s}<\infty$ almost surely, and moreover
the sequence $\{C_{s}\}_{s}$ is tight. This means that the sequence
$\{(d_{s},C_{s})\}_{s}$, where the space of metrics is given the
uniform topology, is tight as well, so $\{(d_{s},C_{s})\}_{s}$ converges
along subsequences. By the Skorohod representation theorem (noting
that $C^{\infty}([0,1]^{4})\times\mathbf{R}$ is a separable Fréchet
space) we can put all of the $(d_{s},C_{s})$s on a common probability
space and get almost-sure convergence along subsequences. But convergence
along an almost-surely convergent subsequence preserves bounds of
the form \eqref{invholderform}, and such a bound holds for $d_{s}$
along any almost-surely convergent subsequence of $\{(d_{s},C_{s})\}_{s}$
since in such a case the $C_{s}$s will be bounded. Thus the proposition
is proved.
\end{proof}
The second statement of \thmref{subseqconv} is the combination of
the results of \propref{holdercts} and \propref{invholdercts}.

\bibliographystyle{imsart-number}
\bibliography{citations}

\end{document}